%% file: main_LQZSMFTG_arxiv.tex
\documentclass[11pt]{amsart}
\usepackage[foot]{amsaddr}

\usepackage{amsmath}
  \usepackage{paralist}
  \usepackage{graphics}
  \usepackage{epsfig} 
\usepackage{graphicx}  
\usepackage{epstopdf}

\usepackage{xcolor}
\usepackage{empheq}
\usepackage{algorithm,algorithmic}

\usepackage{cite}
\usepackage{subcaption}
\usepackage{adjustbox}
\usepackage{textcomp, amssymb}
\usepackage[T1]{fontenc}
\usepackage{mathtools}
\newcommand{\commentJDG}[1]{}

\usepackage{blindtext}
\newcommand\blfootnote[1]{%
\begingroup
\renewcommand\thefootnote{}\footnote{#1}%
\addtocounter{footnote}{-1}%
\endgroup
}

\usepackage{url}
\usepackage[citecolor=red,linkcolor=blue,colorlinks=true,pdftex,backref,pagebackref,plainpages=false]{hyperref}

\newtheorem{theorem}{Theorem}
\newtheorem{corollary}[theorem]{Corollary}
\newtheorem{lemma}[theorem]{Lemma}
\newtheorem{proposition}[theorem]{Proposition}
\newtheorem{assumption}{Assumption}
\theoremstyle{definition}
\newtheorem{definition}[theorem]{Definition}
\newtheorem{remark}[theorem]{Remark}

\newcommand{\EE}{\mathbb{E}}

\newcommand{\PP}{\mathbb{P}}
\newcommand{\RR}{\mathbb{R}}
\newcommand{\NN}{\mathbb{N}}

\newcommand{\iiter}{\mathbf{n}}
\newcommand{\Iiter}{\mathcal{N}^{max}}
\newcommand{\yut}[1]{ u_{#1,t}^{(y)} }
\newcommand{\zut}[1]{ u_{#1,t}^{(z)} }

\def\bX{\boldsymbol{X}}
\def\bZ{\boldsymbol{Z}}

\newcommand\bp{\mathbf p}
\newcommand\bu{\mathbf u}

\newcommand\bx{\mathbf x}
\newcommand\by{\mathbf y}
\newcommand\bz{\mathbf z}
\newcommand\bV{\mathbf V}

\def\bbeta{\boldsymbol{\beta}}

\newcommand\cD{\mathcal D}

\newcommand\cF{\mathcal F}

\newcommand\cL{\mathcal L}
\newcommand\cM{\mathcal M}
\newcommand\cN{\mathcal N}
\newcommand\cP{\mathcal P}

\newcommand\cR{\mathcal R}
\newcommand\cS{\mathcal S}

\newcommand\cU{\mathcal U}

\newcommand\cX{\mathcal X}

\newcommand{\diag}{\mathrm{diag}}

\newlength\myindent
\title[Linear-Quadratic Zero-Sum MFTG] 
      {Linear-Quadratic Zero-Sum Mean-Field Type Games: \\ Optimality Conditions and Policy Optimization}

\author{R. Carmona}
\author{K. Hamidouche}
\author{M. Lauri\`ere}
\author{Z. Tan}
\address[R. Carmona, K. Hamidouche, M. Lauri\`ere, Z. Tan]{Department of Operations Research and Financial Engineering, Princeton University, Princeton, NJ 08540, USA}
\email[R. Carmona, K. Hamidouche, M. Lauri\`ere, Z. Tan]{\{rcarmona, kenzah, lauriere, zongjun.tan\}@princeton.edu }
\thanks{The research of M. Lauri\`ere is supported by NSF grant DMS--1716673 and ARO grant W911NF--17--1--0578.}

\subjclass{Primary: xxx; Secondary: xxx.}
 \keywords{Mean field games, Mean field control, Mean field type games, Zero sum games.}

\begin{document}
\maketitle

\medskip

\bigskip

\blfootnote{{\it A preliminary version of this work has been accepted to the 59th Conference on Decision and Control~\cite{carmona2020lqzsmftg}.}}

\begin{abstract}
In this paper, zero-sum mean-field type games (ZSMFTG) with linear dynamics and quadratic cost are studied under infinite-horizon discounted utility function. ZSMFTG are a class of games in which two decision makers whose utilities sum to zero, compete to influence a large population of indistinguishable agents. In particular, the case in which the transition and utility functions depend on the state, the action of the controllers, and the mean of the state and the actions, is investigated. The optimality conditions of the game are analysed for both open-loop and closed-loop controls, and explicit expressions for the Nash equilibrium strategies are derived. Moreover, two policy optimization methods that rely on policy gradient are proposed for both model-based and sample-based frameworks. In the model-based case, the gradients are computed exactly using the model, whereas they are estimated using Monte-Carlo simulations in the sample-based case. Numerical experiments are conducted to show the convergence of the utility function as well as the two players' controls.
\end{abstract}

\tableofcontents

\allowdisplaybreaks

\input{sec_intro}

\input{sec_model}

\input{sec_proba-setup}

\input{sec_open-loop}

\input{sec_closed-loop}

\input{sec_link-open-closed}

\input{sec_numerics}

\section{Conclusion}
\label{sec6:conc}
In this paper, we have studied zero-sum mean-field type games with linear quadratic model under infinite-horizon discounted  utility  function. We have identified the  closed-form expression  of the  Nash equilibrium controls as linear combinations  of  the  state and  its  mean. Moreover, we have proposed two policy optimization methods to learn the equilibrium. Numerical results have shown the convergence of the two methods in both model-based and sample-based settings. The question of convergence of the algorithms proposed here as well as model-free methods for non-LQ or general-sum MFTG will be studied in future works.

\bibliographystyle{plain}
\bibliography{references}

\end{document}

%% file: sec_intro.tex
\section{Introduction}

\sloppy
Decision making in multi-agent systems has recently received an increasing interest from both theoretical and empirical viewpoints. 
For instance, multi-agent reinforcement learning (MARL) has been applied successfully to problems ranging from self-driving cars and robotics to games, while game-theoretic models have been exploited to study several prominent decision-making problems in engineering, economics and finance. 

In multi-agent systems, a large number of interacting agents either cooperate or compete to optimize a certain individual or common goal. MARL and stochastic games were shown to model well systems with a small number of agents. However, as the number of agents becomes large, analysing such systems 
becomes intractable due to the exponential growth of agent interactions and the prohibitive computational cost. To tackle this issue, mean-field approximations, borrowed from statistical physics, were considered to study the limit behaviour of systems in which the agents are indistinguishable and their decisions are influenced by the empirical distribution
of the other agents. 

Mean-field games (MFGs) \cite{lasry2007mean,HuangMalhameCaines2006_MR2346927} and their variants mean-field type control (MFC) \cite{bensoussan2007representation} and mean-field type games (MFTG) \cite{barreiro2020discrete} consist of studying the global behaviour of systems composed of infinitely many agents which interact in a symmetric manner. In particular, the mean-field approximation captures all agent-to-agent interactions that, individually, have a negligible influence on the overall system's evolution.

An MFG corresponds to the asymptotic limit of the situation in which all the agents compete to minimize their individual utility. In this case, the solution concept is a Nash equilibrium, in which a typical agent is worse-off if she deviates unilaterally. From the point of view of the global system, a better solution can be found by a central planner who tries to minimize the social utility by prescribing the control that each agent should use. This leads to the notion of MFC, which can be viewed as the optimal control of a McKean-Vlasov (MKV) dynamics, in which the evolution of the state process is influenced by its own distribution. Last, mean-field type games are a framework that models control problems involving several decision makers and mean-field interactions. Typical motivations are problems in which large coalitions compete or in which several agents try to influence a large population \cite{djehiche2016mean,bensoussan2018mean}. These three types of models have found numerous applications \cite{bauso2016game}, e.g. in finance \cite{cardaliaguet2018mean}, energy production \cite{bauso2012robust,alasseur2017extended}, crowd motion~\cite{achdou2019mean,MR3392611}, wireless communications \cite{shi2020mean,kim2019mean,meriaux2012mean}, distributed robotics \cite{liu2018mean} and systemic risk~\cite{MR3325083,elie2020large}.

In the past decade, many contributions have contributed to develop the theory of such mean-field problems. In order to study their solutions, a key point is the derivation of optimality conditions, which are typically phrased either in terms of partial differential equations (PDEs) or in terms of forward-backward stochastic differential equations (FBSDEs). For a detailed account, see e.g.~\cite{Cardaliaguet-2013-notes,MR3134900,carmona2018probabilisticI-II} and the references therein. As a cornerstone for applications, the development of numerical methods for these mean-field problems has also attracted a growing interest. Assuming full knowledge of the model, methods for which convergence guarantees have been established include finite difference schemes for partial differential equations~\cite{MR2679575,achdou2012mean}, semi-Lagrangian schemes~\cite{MR3148086}, augmented Lagrangian or primal-dual methods~\cite{MR3575615,MR3772008,BricenoAriasetalCEMRACS2017}, value iteration algorithm~\cite{anahtarci2019value}, or neural network based stochastic methods~\cite{CarmonaLauriere_DL_periodic,CarmonaLauriere_DL}; see e.g.~\cite{achdoulauriere2020mfgnumerical} for a recent overview. However, in many practical situations, the model is not fully known and these methods can not be employed. Hence model-free or sample-based methods, in which the optimization is performed while having only access to a simulator instead of knowing the model, have recently been investigated. For mean-field games, fixed-point~\cite{guo2019learning} or fictitious play scheme ~\cite{elie2020convergence} have been combined with model-free methods to compute the best response, whereas for mean-field control problems, the solution has been approximated using policy gradient~\cite{carmona2019linear} or Q-learning~\cite{carmona2019model,gu2020q}. Despite recent progress, these methods remain restricted to mean-field problems with simple structures which have a common point: the decision makers are either infinitesimal and identical players or a single central planner. More complex models are often needed to tackle applications, such as settings in which a mean-field dynamics is influenced by several distinguishable decision makers. Such situations can typically be modeled by a MFTG.

An archetypal MFTG is the case of mean-field zero-sum games. Two-player zero-sum games in their standard stochastic form, with no mean-field interactions, have been extensively studied in the literature \cite{von2007theory}. In this class of games, two decision makers compete to respectively maximize and minimize the same utility function. The large literature on this topic is motivated by many applications and by connections with robust control \cite{bacsar2008h}. Recently, generalizations to the case where the state dynamics is of MKV type have been introduced in continuous time over a finite time horizon. Optimality conditions have been derived using the theory of backward stochastic differential equations (BSDEs)  in~\cite{xu2012zerosum}, using the dynamic programming principle and partial differential equations (PDEs) in~\cite{cosso2019zero} or using a weak formulation in~\cite{MR4104525}. All these works assume that the controls take values in a compact space, and hence are not applicable to a general linear-quadratic setting. Along a different line, zero-sum games with mean-field interactions have also attracted interest for their connections with generative adversarial nets (GANs)~\cite{domingo2020mean,cao2020connecting}.

Although general stochastic problems with mean-field interactions can be studied from a theoretical perspective, explicit computation of the solution and numerical illustration of the Nash equilibrium are challenging. In standard optimal control, linear-quadratic (LQ) models, where the dynamics are linear and the cost is quadratic, usually have analytical or easily tractable solutions, which makes them very popular. These problems have also been considered in the optimization and machine learning communities, since algorithms with proof of convergence can be developed, see e.g. \cite{fazel2018global} where the authors prove convergence of model-based and sample-based policy gradient methods for a LQ optimal control problem. Sample-based methods have also been used to solve (standard) LQ zero-sum games. In \cite{al2007model}, a discrete-time linear quadratic zero-sum game with infinite time horizon is studied and a Q-learning algorithm is proposed, which is proved to converge to the Nash equilibrium.   
In \cite{zhang2019policy}, the authors study LQ zero-sum games and propose three projected nested-gradient methods that are shown to converge to the Nash equilibrium of the game. However, none of these contributions tackle mean-field interactions in a zero-sum setting.

In the present work, under a discrete time, infinite-horizon
and discounted utility function, we investigate zero-sum mean-field type games (ZSMFTG) of linear-quadratic type, which, to the best of our knowledge, had not been the focus of any work before. In particular, we address the case in which the transition and utility functions do not only depend on the state and the action of the controllers, but also the mean of the state and the actions. Moreover, the state is subject to a common noise. The structure of the problem and the infinite horizon regime allow us to identify the form of the equilibrium controls as linear combinations of the state and its mean conditioned on the common noise, both in the open-loop and the closed-loop settings. To learn the equilibrium, we extend the policy-gradient techniques developed in~\cite{carmona2019linear} for MFC, to the ZSMFTG framework. We design policy optimization methods in which the gradients are either computed exactly using the LQ model or estimated using Monte-Carlo samples when the model is not fully known.

The rest of the paper is organized as follows. In Section~\ref{sec1:prob}, the zero-sum mean-field type game is formulated, preceded by a $N$-agent control problem which motivates this setting. In Section~\ref{app:proba-framework}, we present the rigorous probabilistic setup for the zero-sum mean-field type game under consideration. Open-loop controls are investigated in Section~\ref{sec:open-loop-structure}. After defining the set of admissible controls, we prove a Pontryagin maximum principle giving necessary and sufficient conditions of optimality, see Propositions~\ref{proposition:Pontryagin_maximum_principle_necessary_condition} and~\ref{prop:sufficient_Pontryagin}. Section~\ref{sec:closed-loop-structure} considers closed-loop controls which are linear in the state and the mean. Focusing on the coefficients of the linear combination, we define a notion of admissible controls and prove sufficient conditions of optimality, see Proposition~\ref{prop:sufficient_CLSP_y} and Corollary~\ref{cor:sufficient_CLSP_z}. The connection between equilibria in the open-loop and the closed-loop information structures are studied in Section~\ref{sec:connect-open-closed}, see Lemma~\ref{lemma:connect_ARE_open_closed} and Remark~\ref{rem:from-Pc-to-Po}.
Focusing on closed-loop controls, expressions for the gradient of the utility function and a necessary condition of optimality are derived  in Section~\ref{sec4:algo}, and both model-based and model-free policy optimization methods are proposed. In Subsection~\ref{sec5:num}, we report numerical experiments to show the convergence of the controls and the utility function. Section~\ref{sec6:conc} concludes the paper.

%% file: sec_model.tex
\section{Model and Problem Formulation}
In this section, we first present a zero-sum game in which two controllers compete to influence a population of agents. The agents interact in a symmetric way, through the empirical distribution of their states and actions. We then present a mean-field version of the game (corresponding to the situation where $N \to +\infty$), in which the two controllers influence a state whose dynamics is of MKV type. 
\label{sec1:prob}

\subsection{$N$-agent problem}
Consider a system composed of a population $\{1,\dots,N\}$ with $N$ indistinguishable \emph{agents}. We investigate the case in which these agents have symmetric interactions and are influenced by two \emph{decision makers}, also called \emph{controllers} or \emph{players}, competing to optimize a criterion. In particular, we are interested in the linear-quadratic zero-sum case. Here, the state evolution of an agent $i\in\{1,\dots,N\}$ is given by
\begin{multline}
\label{eq:dyn-N-agent-general}
    x^i_{t+1} = A x^i_t + \bar{A} \bar{x}_t + B_{1} u^i_{1,t} + \bar{B}_{1} \bar{u}_{1,t} + B_{2}u^i_{2,t} + \bar{B}_{2}\bar{u}_{2,t} 
    + \epsilon^i_{t+1} + \epsilon^0_{t+1},
\end{multline}
with initial condition $x^i_0 = \epsilon^i_0 + \epsilon^0_0$, where  $x^i_0$ is the initial state of agent $i$ to which we introduce randomness with $\epsilon^i_0$ and $ \epsilon^0_0$. At each time $t$, $x_t^i \in \RR^d$ corresponds to the state of the $i$-th agent in the population, and $u^i_{1,t} \in \RR^\ell$ and $u^i_{2,t} \in \RR^\ell$ are the controls prescribed to this agent respectively by the first and the second decision maker. The noise terms $\epsilon^0_{t+1}$ and $\epsilon^i_{t+1}$ are independent of each other and of $\epsilon^0_0$ and $ \epsilon^i_0$. Moreover, the noise terms $\epsilon^0_{t+1}$ for $t \ge 0$ are assumed to be identically distributed with mean $0$, and similarly for $\epsilon^i_{t+1}$ for $t \ge 0$. The interpretation of the noise terms is that $\epsilon^0_{t}$ is a common noise affecting the position of all the agents, whereas  $\epsilon^i_{t}$ is an indiosyncratic noise affecting only the position of the $i$-th agent. $A, \bar{A}, B_i, \bar{B}_i$ are fixed matrices with suitable dimensions.  Here, $\bar{x}_t=\frac{1}{N}\sum_{i=1}^{N} x_t^i$, is the sample average of the individual states, and similarly for $\mathrm{u}_1$ and $\mathrm{u}_2$: $\bar{u}_{j,t}=\frac{1}{N}\sum_{i=1}^{N} u_{j,t}^i$. The instantaneous utility is defined by 
\begin{equation}
\label{eq:MKV_running_cost_c}
\begin{aligned}
    c(x, \bar{x}, \mathrm{u}_{1}, \bar{\mathrm{u}}_{1}, \mathrm{u}_{2}, \bar{\mathrm{u}}_{2}) 
    &=
     (x-\bar{x})^{\top} Q (x-\bar{x}) + \bar{x}^{\top} (Q+\bar{Q}) \bar{x}
    \\ 
    &\qquad + (\mathrm{u}_{1}-\bar{\mathrm{u}}_{1})^{\top} R_{1} (\mathrm{u}_{1}-\bar{\mathrm{\bf \mathrm{u}}}_{1}) + \bar{\mathrm{u}}^\top_{1}
    + (R_{1}+\bar{R}_{1}) \bar{\mathrm{u}}_{1}
    \\
    &\qquad - (\mathrm{u}_{2} - \bar{\mathrm{u}}_{2})^\top R_{2} (\mathrm{u}_{2} - \bar{\mathrm{u}}_{2}) - \bar{\mathrm{u}}^\top_{2} (R_{2}+\bar{R}_{2}) \bar{\mathrm{u}}_{2},
\end{aligned}
\end{equation}
where $Q,\bar Q, R_i, \bar R_i$ are deterministic symmetric matrices of suitable sizes such that $R_i, R_i + \bar R_i$ for $i=1,2$ are positive definite.

The objective of each controller in this zero-sum problem is to minimize (resp. maximize) the $N$-agent utility functional 
\begin{equation*}
    J^N(\mathrm{\bf {\underline u}}_1, \mathrm{\bf {\underline u}}_2) = \EE\left[\sum_{t=0}^{+\infty} \gamma^{t} \bar{c}^N({\underline x}_t, {\underline u}_{1,t}, {\underline u}_{2,t}) \right],
\end{equation*}
where ${\underline x}_t = (x_t^1,\dots,x_t^N)$, and ${\bf {\underline u}}_i = ({\underline u}_{i,t})_t$ with ${\underline u}_{i,t} = (u^1_{i,t},\dots,u^N_{i,t})$ (we use a boldface to denote a function of time and an underline to denote a vector of size $N$), and $\bar c^N$ is the average utility, defined by
\begin{equation*}
    \bar{c}^N({\underline x}_t, {\underline u}_{1,t}, {\underline u}_{2,t}) 
    = \frac{1}{N} \sum_{i=1}^{N} c(x^i_t,\bar{x}_t, u^i_{1,t}, \bar{u}_{1,t}, u^i_{2,t}, \bar{u}_{2,t}).
\end{equation*}

 The minimax problem is defined as follows,
 \begin{equation}
        \inf_{\mathrm{\bf \underline{u}}_1} \sup_{\mathrm{\bf \underline{u}}_2} J^N(\mathrm{\bf \underline{u}}_1, \mathrm{\bf \underline{u}}_2).
\end{equation}

This problem is a generalization of the mean-field control setup, in which there is a single decision maker. It can also be viewed as a variant of Nash mean-field control setup studied in~\cite{bensoussan2018mean} or mean-field type games~\cite{djehiche2016mean} in which several mean-field decision makers compete in a general-sum game.

\begin{remark}
\label{rem:2pop-case}
An interesting special case is the situation in which each decision maker controls a different population. This corresponds to a zero-sum game between two large coalitions. This setting can be covered in the following way. Assume that $d = 2 d'$ for some integer $d'$. Consider, for the dynamics, block matrices of the form:
$$
    A = \begin{pmatrix}
    A_1 & 0 \\
    0 & A_2
    \end{pmatrix}, \,\,\, 
    B_1 = \begin{pmatrix}
    B_1^1  \\
    0 
    \end{pmatrix}, \,\,\, 
    B_2 = \begin{pmatrix}
    0 \\
    B_2^2
    \end{pmatrix},
$$
and
$$
    \bar{A} = \begin{pmatrix}
    \bar{A}_{11} & \bar{A}_{12} \\
    \bar{A}_{21} & \bar{A}_{22}
    \end{pmatrix}, \,\,\, 
    \bar{B}_1 = \begin{pmatrix}
    \bar{B}_1^1 \\
    \bar{B}_1^2
    \end{pmatrix}, \,\,\, 
    \bar{B}_2 = \begin{pmatrix}
    \bar{B}_2^1 \\
    \bar{B}_2^2
    \end{pmatrix}.
$$
Then the dynamics~\eqref{eq:dyn-N-agent-general} rewrites, with the notation $x = (x_1, x_2)$ where $x_i \in \RR^{d'}$ and similarly for $\epsilon^0, \epsilon^i$, 
\begin{align*}
    d x^i_{1,t} &= \left[A_1 x^i_{1,t} + \bar{A}_{11} \bar{x}_{1,t} + \bar{A}_{12} \bar{x}_{2,t}
    + B_1^1 u^i_{1,t}  + \bar{B}_1^1 \bar{u}_{1,t}  + \bar{B}_1^2 \bar{u}_{2,t} 
     \right] + \epsilon^i_{1,t} + \epsilon^0_{1,t},
    \\
    d x^i_{2,t} &= \left[A_2 x^i_{2,t} + \bar{A}_{21} \bar{x}_{1,t} + \bar{A}_{22} \bar{x}_{2,t}
    + B_2^2 u^i_{2,t}  + \bar{B}_1^2 \bar{u}_{1,t}  + \bar{B}_2^2 \bar{u}_{2,t} 
     \right] + \epsilon^i_{2,t} + \epsilon^0_{2,t}.
\end{align*}
Note that the evolution of the two halves of vector $x$ are coupled only through their expectations and the expectation of the control used for the other half. We can thus interpret each half as the state of a player in a different population where each population has $N$ indistinguishable agents.
\end{remark}

\subsection{Mean-field problem}

Here, we consider the limit of the $N$-agent case. The dynamics is given by
\begin{multline}
\label{eq:MKV-state_ZS}
    x_{t+1} = A x_t + \bar{A} \bar{x}_t + B_1 u_{1,t} + \bar{B}_1 \bar{u}_{1,t}
    + B_2 u_{2,t}+\bar{B}_2\bar{u}_{2,t}  
    + \epsilon^0_{t+1} + \epsilon^1_{t+1},
\end{multline}
with initial condition 
$$
    x_0 = \epsilon^0_0 + \epsilon^1_0.
$$

Here and thereafter, when considering the mean-field problem, we use the notation $\bar{x}_t = \EE[x_t | (\epsilon^0_{s})_{0 \le s \le t}]$ for the expectation of the state conditional on the realization of the common noise, and likewise for $\mathrm{\bf u}_1$ and $\mathrm{\bf u}_2$. Note that~\eqref{eq:MKV-state_ZS} is a dynamics of MKV type since it is influenced by its own distribution and by the distribution of the actions.
The utility function takes the form 
\begin{equation}
\label{fo:MKV-discounted_utility_ZS}
    J(\mathrm{\bf u}_1, \mathrm{\bf u}_2) = \EE\left[\sum_{t=0}^{+\infty} \gamma^{t} c_t \right],
\end{equation}
where $\gamma \in [0,1]$ is a discount factor, and the instantaneous utility at time $t$ is defined as
\begin{align}
\label{eq:instantaneous-utility}
    c_t 
    & = c(x_t, \bar{x}_t, u_{1,t}, \bar{u}_{1,t}, u_{2,t}, \bar{u}_{2,t}),
\end{align}
where the function $c$ is as in the $N$-agent problem. 

The goal is to find a Nash equilibrium, namely a pair of $( \mathrm{\bf u}^*_1, \mathrm{\bf u}^*_2)$ such that
\begin{align}
\label{eq:prob_for}
    J( \mathrm{\bf u}^*_1, \mathrm{\bf u}^*_2)
    =
    \inf_{\mathrm{\bf u}_1} \sup_{\mathrm{\bf u}_2} J(\mathrm{\bf u}_1, \mathrm{\bf u}_2).
\end{align}

Next, we study the existence of the Nash equilibrium and derive its closed-form expression for the formulated ZSMFTG.

%% file: sec_proba-setup.tex
\section{Probabilistic setup}
\label{app:proba-framework}

In this section we rigorously define the model of MKV dynamics with common noise. It is analogous to the one considered in~\cite{carmona2019linear}, except for the fact that there are two decision makers instead of one.
A convenient way to think about this model is to view the state $x_t$ of the system at time $t$ as a random variable defined on the probability space $(\Omega,\cF,\PP)$ where $\Omega=\Omega^0\times\Omega^1$, $\cF=\cF^0\times\cF^1$ and $\PP=\PP^0\times\PP^1$. In this set-up, if $\omega=(\omega^0,\omega^1)$,  $\epsilon^0_t(\omega)=\tilde\epsilon^0_t(\omega^0)$ and  $\epsilon^1_t(\omega)=\tilde\epsilon^1_t(\omega^1)$ where $(\tilde\epsilon^0_{t})_{t=1,2,\ldots}$ and $(\tilde\epsilon^1_{t})_{t=1,2,\ldots}$ are i.i.d. sequences of mean-zero random variables on $(\Omega^0,\cF^0,\PP^0)$ and $(\Omega^1,\cF^1,\PP^1)$ respectively, while the initial sources of randomness  $\tilde\epsilon^0_{0}$ and $\tilde\epsilon^1_{0}$ are random variables on $(\Omega^0,\cF^0,\PP^0)$ and $(\Omega^1,\cF^1,\PP^1)$ with distributions $\mu^0_0 $ and $\mu^1_0$ respectively, which are independent of each other and independent of $(\tilde\epsilon^0_{t})_{t=1,2,\ldots}$ and $(\tilde\epsilon^1_{t})_{t=1,2,\ldots}$.
We denote by $\cF_t$ the filtration generated by the noise up until time $t$, that is $\cF_t = \sigma(\epsilon^0_0, \epsilon^1_0, \epsilon^0_1, \epsilon^1_1, \dots , \epsilon^0_t, \epsilon^1_t)$.
We assume that the variance of random variables $\epsilon_t^0$ and $\epsilon_t^1$ are constant along time, and these variances are denoted by $\Sigma^0 = \EE[ (\epsilon_t^0)^\top \epsilon_t^0 ]$ and $\Sigma^1 = \EE[ (\epsilon_t^1)^\top \epsilon_t^1 ]$ for every $t \geq 1$.

At each time $t\ge 0$, $x_t$ and $u_{i,t}$ with $i=1,2$ are random elements defined on $(\Omega,\cF,\PP)$ representing the state of the system and the controls exerted by a pair of generic agents. Using the fact that the idiosyncratic noise and the common noise are independent, the quantities $\bar{x}_t$ and $\bar{u}_{i,t}$ with $i= 1,2$ appearing in \eqref{eq:MKV-state_ZS} are random variables on $(\Omega,\cF,\PP)$ defined by: for $\omega = (\omega^0, \omega^1)$,
\begin{equation*}
\label{fo:bars}
\bar{x}_t(\omega^0,\omega^1)=\int_{\Omega^1}x_t(\omega^0,\tilde\omega^1)\PP^1(d\tilde\omega^1),
\quad
\bar{u}_{i,t}(\omega^0,\omega^1)=\int_{\Omega^1}u_{i,t}(\omega^0,\tilde\omega^1)\PP^1(d\tilde\omega^1),  \   i = 1,2.
\end{equation*}
Notice that $\bar{x}_t$, $\bar{u}_{1,t}$ and $\bar{u}_{2,t}$ depend only upon $\omega^0$. In fact, the best way to think of $\bar{x}_t$ and $\bar{u}_{i,t}$ with $i=1,2$ is to keep in mind the following fact: 
$$
\bar{x}_t=\EE[x_t|\cF^0], \qquad \text{and} \qquad  \bar{u}_{i,t}=\EE[u_{i,t}|\cF^0].
$$ 
These are the mean field terms appearing in the (stochastic) dynamics of the state~\eqref{eq:MKV-state_ZS}:
\begin{equation*}
	x_{t+1} = A x_{t} + \bar A \bar x_t +  B_1 u_{1,t} + \bar B_1 \bar{u}_{1,t} + B_2 u_{2,t} + \bar B_2 \bar u_{2,t} + \epsilon_{t+1}^0 + \epsilon_{t+1}^1.
\end{equation*}

%% file: sec_open-loop.tex
\section{Open-loop information structure}
\label{sec:open-loop-structure}

In this section, we consider open-loop controls, that is, controls available if the controllers can directly see the noise terms. We start with this class of controls because it is somehow ``larger'' than the class of closed-loop controls that will be considered in the next section (any closed-loop control gives rise to an open-loop control, but the converse is not always true). The main point of this section is to show that, under suitable conditions, the Nash equilibrium controls in the open-loop setting can in fact be written as linear combinations of the state and the conditional mean. 

\subsection{Admissible controls}

We introduce the following sets: for $T \geq 0$,
\begin{equation*}
    \begin{aligned}
        \cU_T & := \Big\{  \bu: \{0,\ldots,T\} \times \Omega \mapsto \RR^{\ell} \  | \  u_t  
        \text{ is } \cF_t-\text{measurable}, 
        \  \EE\left[ \sup_{t=0,\ldots,T} \gamma^t \| u_t \|^2 \right] < \infty \Big\}, 
        \\
        \cU_{loc} &:= \bigcup_{T \geq 0} \cU_T,
        \qquad
        \cU := \left\{ \bu: \NN \times \Omega \mapsto \RR^{\ell} \  \left| \  \bu \in \cU_{loc}, \   \EE\left[ \sum_{t=0}^\infty \gamma^t \| u_t \|^2 \right] < \infty \right. \right\},
\end{aligned}
\end{equation*}
where we use the notation $u_t(\cdot) = \bu(t, \cdot)$ for every $t \in \NN$ and we identify $\bu$ to an $\cF$-adapted process $(u_t)_{t\geq 0}$. A process $\bu$ is called $L^2-$discounted globally integrable, or $L^2-$integrable for short, if $\bu \in \cU$. 
Also, for $T \geq 0$,
\begin{equation*}
\begin{aligned}
\cX_T & := \Big\{  \bx: \{0,\ldots,T\} \times \Omega \mapsto \RR^d \  | \  x_t   
\text{ is } \cF_t-\text{measurable}, 
    \  \EE\left[ \sup_{t=0,\ldots,T} \gamma^t \| x_t \|^2 \right] < \infty \Big\},  
\\
\cX_{loc} &:= \bigcup_{T \geq 0} \cX_T,
\qquad
\cX  := \left\{  \bx \in \NN \times \Omega \mapsto \RR^d \  \left| \  \bu \in \cX_{loc},  \  \EE\left[ \sum_{t=0}^\infty \gamma^t \| x_t \|^2 \right] < \infty \right. \right\}.
\end{aligned}
\end{equation*}
Similarly, we identify $\bx \in \cX_{loc}$ or $\bx \in \cX$ to an $\cF-$adapted process $(x_t)_{t\geq 0}$ in $\RR^d$. We also call a state process is $L^2-$discounted globally integrable, or simply $L^2-$integrable, if $\bx \in \cX$. Let $\cS^d$ stand for the set of symmetric matrices in $\RR^{d\times d}$.\\

In the open-loop information structure, we consider the following subset of $\cU \times \cU$:
\begin{equation*}
\cU_{ad}^{open} := \left\{ (\bu_1, \bu_2) \in \cU \times \cU \  | \  (x_t^{\bu, \bu_2})_{t \geq 0} \in \cX \right\}
\end{equation*}
where the state process $(x_t^{\bu_1, \bu_2})_{x \geq 0}$ follows the dynamics \eqref{eq:MKV-state_ZS}. We call every element $(\bu_1, \bu_2) \in \cU_{ad}^{open}$ an admissible (open-loop) control pair for the two players.

We collect here a few useful results, stated without proof for the sake of brevity.

	The following proposition is about the $L^2$-integrability of the state processes.

	\begin{proposition}
		\label{prop:stability_with_extra_L2_term}
		We consider a state processes $\bX= (X_t)_{t \geq 0}$ following the dynamics 
		\begin{equation}
		\label{eq:sys_A}
		X_{t+1} = A X_t +  q_t + W_{t+1}, \qquad X_0 \sim \mu_0
		\end{equation}
		where $A \in \RR^{d \times d}$, $ q = (q_t)_{t \geq 0} \in \cU$, and $\mu_0 \in \cP^2(\RR^d)$ such that $\EE[X_0] = 0$ and $\EE[\| X_0 \|^2] < \infty$. Assume the noise process $(W_{t})_{t \geq 1}$ satisfies $ \EE[W_{i} W_j] = 0$ for $1 \leq i < j < \infty$, $\EE[W_t] = 0$ for every $t \geq 1$, and also
		\begin{equation*}
			\EE\left[ \sum_{t=1}^\infty \gamma^t \| W_{t} \|^2 \right]	< \infty.	
		\end{equation*}
		Assume that the matrix $A$ satisfies 
		$
			\gamma \| A \|^2 < 1.
		$
		Then the state process $\bX = (X_t)_{t \geq 0}$ is in $\mathcal{X}$, i.e.
		$
			\EE \left[ \sum_{t=0}^\infty \gamma^t \| X_t \|^2 \right] < \infty.
		$
	\end{proposition}

    The following result gives a link with $L^2-$asymptotical stability.
    \begin{lemma}
    \label{lemma:x_L2_integ_L2_asym_stable}
        If $\bx \in \cX$, then we have $\lim_{t \to \infty} \EE [ \gamma^t \| x_t \|^2 ] = 0$.
    \end{lemma}
    
    \commentJDG{
    \begin{proof}
        If $\EE[ \gamma^t \| X_t \|^2 ] $ does not tend to zero, then there exists an $\epsilon > 0$ such that for any $N \in \NN$, these is a time point $s > N$ satisfying $\EE[ \gamma^{s} \| X_{s} \|^2 ] > \epsilon$.
        Thus, we can construct a sequence of time points $t_0 < t_1 < ...$ such that for every $i \in \NN$, $ \EE[ \gamma^{t_i} \| X_{t_i} \|^2 ] > \epsilon$. Their sum tends to infinity, contradicting the fact that $\EE[ \sum_{t=0}^\infty \gamma^t \| X_t \|^2 ] < \infty$.
    \end{proof}

    \begin{remark}
        The above result is usually referred to as the $L^2-$asymptotically stability for the state process $\bx$ in a continuous-time situation. See Proposition~3.3 in \cite{huang2012linear} for a detailed discussion on the $L^2-$globally integrability and the $L^2-$asymptotically stability in a continuous-time model without the discount factor.
        \end{remark}
        }
    
    We have the following two lemmas related to the $L^2-$discounted globally integrability for processes $(x_t - \bar{x}_t)_{t \geq 0}$ and $(\bar{x}_t)_{t \geq 0}$. 
	\begin{lemma} 
	\label{lemma:L2_x_and_x_bar}
	A process $\bx \in \cX$ if and only if both processes $(x_t - \bar x_t)_{t \geq 0} \in \cX$ and $(\bar x_t)_{t \geq 0} \in \cX$.
	Similarly, a control process $\bu \in \cU$ if and only if both processes $(u_t - \bar u_t)_{t \geq 0} \in \cU$ and $(\bar u_t)_{t \geq 0} \in \cU$.
	\end{lemma}
	
	\commentJDG{
	    This comes from the fact that
	    $
	        \EE[ \| x_t - \bar{x}_t \|^2 ] = \EE[ \| x_t \|^2 - \| \bar{x}_t \|^2 ]
	    $
	    and by Jensen's inequality  $ \EE[ \| \bar x_t \|^2 ] = \EE[ \| \EE[ x_t | \cF_t^0] \|^2 ] \leq \EE[ \| x_t \|^2 ]$.
    }
    
    Using Proposition~\ref{prop:stability_with_extra_L2_term} and Lemma~\ref{lemma:L2_x_and_x_bar}, we deduce the following result. 
	\begin{lemma} 
	\label{lemma:L2_y_and_z}
	For any given $(\bu_1, \bu_2) \in \cU \times \cU$, let $(q_t^{(y)})_{t \geq 0}$ and $(q_t^{(z)})_{t \geq 0} $ be two processes in $\cU$ given by: for every $t \geq 0 $,
		 \begin{equation}
		 \label{eq:qy_qz}
		 	\left\{ 
		 	\begin{aligned}
		 	q^{(y)}_t & = B_1 (u_{1,t}- \bar{u}_{1,t} ) + B_2 (u_{2,t}- \bar{u}_{2,t} ),
		  	\\
		  	q^{(z)}_t &= (B_1 + \bar{B}_1) \bar u_{1,t} + (B_2 + \bar{B}_2) \bar{u}_{2,t}.
		  	\end{aligned}
		   \right.
		  \end{equation}
		We have:
		\begin{itemize}
		\item if $\gamma \| A \|^2 < 1$, then the state process $\by = (y_t)_{t \geq 0}$ following the dynamics
		\begin{equation}
		    \label{eq:dyn_y_qy}
			y_{t+1} = A y_t + q_t^{(y)} +  \epsilon_{t+1}^1, \qquad y_0 \sim \mu_0^1
		\end{equation} 
		is $L^2-$discounted globally integrable, i.e. $\EE\left[ \sum_{t=0}^\infty \gamma^t \| y_t \|^2 \right] < \infty$;
		
		\item 
		if $\gamma \| A + \bar A \|^2 < 1$, the state process $\bz = (z_{t})_{t \geq 0}$ following the dynamics
		\begin{equation}
		\label{eq:dyn_z_qz}
		z_{t+1} = (A + \bar A) z_t + q_t^{(z)} +  \epsilon_{t+1}^0, \qquad z_0 \sim \mu_0^0
		\end{equation} 
		is $L^2-$discounted globally integrable, i.e. $\EE\left[ \sum_{t=0}^\infty \gamma^t \| z_t \|^2 \right] < \infty$.
		\end{itemize}
	\end{lemma}
	
	\commentJDG{
	\begin{proof}
	    Since the noise processes $\boldsymbol{\epsilon}^0, \boldsymbol{\epsilon}^1$ and the initial distributions $\mu_0^0, \mu_0^1$ satisfy conditions assumed in Proposition~\ref{prop:stability_with_extra_L2_term}, and the processes $(q_t^{(y)})_{t \geq 0}, (q_t^{(z)})_{t \geq 0}$ are in $\cU$ by Lemma~\ref{lemma:L2_x_and_x_bar}, the results hold. \\
	\end{proof}
    }
    
    Now, we are ready to provide the $L^2-$discounted global integrability for the state process $\bx^{\bu_1, \bu_2} = (x_t^{\bu_1, \bu_2})_{t \geq 0}$ following dynamics \eqref{eq:MKV-state_ZS} controlled by two processes $\bu_1, \bu_2 \in \cU$.
	\begin{assumption} 
	\label{assumption:A_openloop}
	$ \gamma \| A \|^2 < 1	$ and $\gamma \| A + \bar A \|^2 < 1$.
	\end{assumption}
    
    \begin{proposition}
	Under Assumption~\ref{assumption:A_openloop}, we have $\cU_{ad}^{open} = \cU \times \cU$. In particular, the set of admissible controls $\cU_{ad}^{open}$ is convex.
	\end{proposition}
	    
	\begin{proof}
	    By definition, $\cU_{ad}^{open} \subseteq \cU \times \cU$. For the other inclusion, let us consider a pair of control processes $(\bu_1, \bu_2) \in \cU \times \cU$. We know that the corresponding state process $\bx^{\bu_1, \bu_2} \in \cX_{loc}$. Taking the conditional expectation with respect to $\cF^0$ and denoting $\bar{x}_t^{\bu_1, \bu_2} = \EE[x_t^{\bu_1, \bu_2} | \mathcal{F}^0]$, we notice that, for every $t\geq 0$,
	    \begin{equation*}
	    \left\{
	    \begin{aligned}
	        x^{\bu_1, \bu_2}_t - \bar{x}_t^{\bu_1, \bu_2} &= A (x^{\bu_1, \bu_2}_t - \bar{x}_t^{\bu_1, \bu_2}) + q_t^{(y)} + \epsilon_{t+1}^1
	        \\
	        \bar{x}_t^{\bu_1, \bu_2} &= (A + \bar A) \bar{x}_t^{\bu_1, \bu_2} + q_t^{(z)} + \epsilon_{t+1}^0
	    \end{aligned}
	    \right.
	    \end{equation*}
	    where $q_t^{(y)}$ and $q_t^{(z)}$ are given by \eqref{eq:qy_qz}. Let us denote $y_t = x^{\bu_1, \bu_2}_t - \bar{x}_t^{\bu_1, \bu_2} $ and $z_t = \bar{x}_t^{\bu_1, \bu_2}$ for every $t \geq 0$. Under Assumption~\ref{assumption:A_openloop}, by Lemma~\ref{lemma:L2_y_and_z}, we obtain that the processes $(y_t)_{t \geq 0} \in \cX$ and $(z_t)_{t \geq 0} \in \cX$. which implies $\bx^{\bu_1, \bu_2} \in \cX$. Thus, $\cU \times \cU \subseteq \cU_{ad}^{open}$. The convexity of $\cU_{ad}^{open}$ is a direct consequence of the convexity of $\cU$.
	\end{proof}
	
	\begin{remark}
	    In the closed-loop information structure (c.f. Section \ref{sec:closed-loop-structure}), we will see that the set of closed-loop admissible policy is not convex.
	\end{remark}

\begin{definition}
    A pair of admissible control processes $(\bu_1^*, \bu_2^*) \in \cU_{ad}^{open}$ is an open-loop Nash equilibrium (or open-loop saddle point, OLSP for short) for the zero-sum game if for any process $\bu_1' \in \cU$ and $\bu_2' \in \cU$, we have
    \begin{equation}
        \label{eq:open_loop_Nash}
        J(\bu_1^*, \bu_2') \leq J(\bu_1^*, \bu_2^*) \leq J(\bu_1', \bu_2^*),
    \end{equation}
    where $(\bu_1, \bu_2) \mapsto J(\bu_1, \bu_2)$ is the utility function defined in equation \eqref{fo:MKV-discounted_utility_ZS}.
\end{definition}

\subsection{Equilibrium condition}

For the sake of convenience, we use the notation $\check{A} = A - I_d$ where $I_d$ denotes the $d\times d$ identity matrix, and $\zeta = (x,\bar{x},u_1,\bar{u}_1,u_2,\bar{u}_2)$, so that, if we define the function $b$ by:
\begin{equation}
\label{fo:MKV-b_tilde}
b(\zeta) = b(x,\bar{x},u_1,\bar{u}_1,u_2,\bar{u}_2) = \check{A} x + \bar{A}\bar{x}  + B_1 u_1 + \bar{B}_1 \bar{u}_1 + B_2 u_2 + \bar{B}_2 \bar{u}_2.
\end{equation}
The state equation~\eqref{eq:MKV-state_ZS} can be rewritten as:
\begin{equation}
\label{eq:MKV-state_ZS_with_theta}
x_{t+1}-x_t 
= b(x_t ,\bar{x}_t,u_{1,t},\bar{u}_{1,t},u_{2,t},\bar{u}_{2,t}) 
+\epsilon^0_{t+1}+\epsilon^1_{t+1} 
= b(\zeta_t) +\epsilon^0_{t+1}+\epsilon^1_{t+1}.
\end{equation}

We define the Hamiltonian function $h$ by:
\begin{equation}
\label{fo:MKV-hamiltonian_ZS}
\begin{split}
&h(\zeta, p)
\\
&= h(x,\bar{x},u_1,\bar{u}_1,u_2,\bar{u}_2, \eta)
\\
&=[\check{A} x + \bar{A}\bar{x}  + B_1 u_1 + \bar{B}_1 \bar{u}_1 + B_2 u_2 + \bar{B}_2 \bar{u}_2]\cdot p +c(x,\bar{x},u_1,\bar{u}_1,u_2,\bar{u}_2) - \delta x\cdot p
\\
&= b(\zeta)\cdot p +c(\zeta)-\delta x\cdot p
\end{split}
\end{equation}
for $p \in \RR^d$, where $\delta = (1 - \gamma)/\gamma$ is a positive constant representing the discount rate, $\gamma \in [0,1]$ being the discount factor. Throughout, we use the notation $\cdot$ for the scalar product in Euclidean space.
We will use the following property of the Hamiltonian, under the following assumption, where $\succeq 0$ (resp. $\succ$) means that the matrix is non-negative semi definite (resp. positive definite).

	\begin{lemma}
		\label{le:MKV-convexity_ZS} 
		If $R_1 \succeq 0,$ $ R_1+\bar{R}_1 \succeq 0$ (resp. $ R_2 \succeq 0,$ and $ R_2+\bar{R}_2 \succeq 0$),
		the function $h$ is convex w.r.t. $(u_1,\bar{u}_1)$ (resp. concave w.r.t. $(u_2,\bar{u}_2)$). It is strictly convex (resp. strictly concave) if $R_1 \succ 0$ and $(R_1+\bar{R}_1) \succ 0$ (resp. $R_2 \succ 0$ and $(R_2+\bar{R}_2) \succ 0$).
	\end{lemma}
\begin{proof}
	For the purpose of computing  gradients, Hessians and partial derivatives, we treat $\zeta$ as a $(2d+4 \ell)\times 1$ column vector by specifying its definition as $\zeta=[x^{\top}, \bar x^{\top}, u_1^{\top}, \bar u_1^{\top}, u_2^{\top},\bar u_2^{\top}]^{\top}$. Now, for every fixed $p \in\RR^d$, we have:
	\begin{equation}
	\label{fo:gradient}
	\nabla_\zeta h(\zeta,p)=
	\begin{pmatrix}
	\partial_x h(\zeta,p)\\
	\partial_{\bar{x}} h(\zeta,p)\\
	\partial_{u_1} h(\zeta,p)\\
	\partial_{\bar{u}_1} h(\zeta,p)\\
	\partial_{u_2} h(\zeta,p)\\
	\partial_{\bar{u}_2} h(\zeta,p)
	\end{pmatrix}
	=
	\begin{pmatrix}
	p^{\top}(\check{A}-\delta I_d)+2(x-\bar{x})^{\top}Q\\
	p^{\top}\bar{A} +2(\bar{x}-x)^{\top}Q+2\bar x^{\top}(Q+\bar{Q})
	\\
	p^{\top}B_1+2(u_1-\bar{u}_1)^{\top}R_1\\
	p^{\top}\bar{B}_1 +2(\bar{u}_1-u_1)^{\top}R_1+2\bar u_1^{\top}(R_1+\bar{R}_1)
	\\
	p^{\top}B_2 - 2(u_2-\bar{u}_2)^{\top}R_2\\
	p^{\top}\bar{B}_2 - 2(\bar{u}_2-u_2)^{\top}R_2 - 2\bar u_2^{\top}(R_2+\bar{R}_2).
	\end{pmatrix}
	\end{equation}
	\commentJDG{
	so that:
	\begin{equation}
	\label{fo:hessian}
	\nabla^2_{\zeta\zeta} h(\zeta,p)=
	\begin{pmatrix}
	2Q&-2Q&0&0 &0&0\\
	-2Q&2(2Q+\bar{Q})&0&0 &0&0\\
	0&0&2R_1&-2R_1 &0&0\\
	0&0&-2R_1&2(2R_1+\bar{R}_1) &0&0\\
	0&0 &0&0 & -2R_2& 2R_2\\
	0&0 &0&0 & 2R_2& -2(2R_2+\bar{R}_2)
	\end{pmatrix}.
	\end{equation}
	}
	It can be seen that
	\begin{equation*}
	\nabla^2_{(u_1,\bar{u}_1),(u_1,\bar{u}_1)} h(\zeta,p)=
	\begin{pmatrix}
	2R_1&-2R_1\\
	-2R_1&2(2R_1+\bar{R}_1)
	\end{pmatrix},
  	\end{equation*}
	is non-negative definite if the inequalities $R_1 \succeq 0$ and $R_1 + \bar R_1 \succeq 0$ are satisfied, and positive definite if $R_1 \succ 0$ and $(R_1+\bar{R}_1) \succ 0$. Likewise for the second order derivatives w.r.t. $(u_2, \bar{u}_2)$.
\end{proof}

In the spirit of the analysis of the stochastic version of the Pontryagin maximum principle, we introduce the notion of adjoint process associated to a given admissible pair of control processes.

\begin{definition}
	\label{de:MKVadjoint}
	If $\bu_i=(u_{i,t})_{t=0,1,\ldots}, i=1,2$ is a pair of  admissible control processes and $\bx=(x_t)_{t=0,1,\ldots}$ is the corresponding state process, we say that an $\RR^d$-valued $(\cF_t)_{t\geq 0}$-adapted process $\boldsymbol{p}=(p_t)_{t=0,1,\ldots}$ is an adjoint process if it satisfies:
	\begin{equation}
	\label{fo:MKV-adjoint_ZS}
	p_t = \EE\Bigl[p_{t+1}+ \gamma\bigl[(\check{A}^{\top}-\delta I_d) p_{t+1} + 2Q x_{t+1} + \bar{A}^{\top}\bar{p}_{t+1}+2\bar{Q}\bar{x}_{t+1}\bigr]|\cF_t\Bigr],\qquad t=0,1,\ldots.
	\end{equation}
\end{definition}

It will be useful to note that the above expression can equivalently be written as:
\begin{equation}
\label{eq:useful-expression-pt}
p_t = \gamma \EE [ A^\top p_{t+1} + 2 Q x_{t+1} + \bar A^\top \bar p_{t+1} + 2 \bar Q \bar x_{t+1} \ | \  \cF_t ].
\end{equation}

Using this notion, we can express the derivative of the function $J$ as follows.

	\begin{lemma}
		\label{lem:differential-J}
		The Gateaux derivative of $J$ at $(\bu_1,\bu_2)$ in the direction $(\bbeta_1,\bbeta_2) \in \cU \times \cU$ exists and is given by 
		\begin{equation}
		\label{fo:second}
		\begin{split}
		D J(\bu_1, \bu_2)(\bbeta_1, \bbeta_2)
		& =
		\EE \left[ \sum_{t=0}^\infty\gamma^t \bigl( p_t^{\top} B_1 +2 u_{1,t}^{\top} R_1 + \bar p_t^{\top}\bar{B}_1 +2\bar{u}_{1,t}^{\top}\bar{R}_1\bigr)  \beta_{1,t}
		\right]
		\\
		& \qquad +
		\EE \left[ \sum_{t=0}^\infty\gamma^t \bigl( p_t^{\top} B_2 -2 u_{2,t}^{\top} R_2 + \bar p_t^{\top}\bar{B}_2 -2\bar{u}_{2,t}^{\top}\bar{R}_2\bigr)  \beta_{2,t}
		\right] \, .
		\end{split}
		\end{equation}
	\end{lemma}
\begin{proof}
	We start by computing the difference between the values of $J$ evaluated on two pairs of controls. 
	Let $\bu_i=(u_{i,t})_{t=0,1,\ldots}$ and $\bu'_i=(u_{i,t}')_{t=0,1,\ldots}$, $i=1,2$ be two pairs of admissible control processes and let us denote by 
	$x_t$ and $x'_t$  the corresponding states of the system at time $t$, as given by the state equation \eqref{eq:MKV-state_ZS} with the same initial point and the same realizations of the noise sequences $(\epsilon^0_{t})_{t=0,1,\ldots}$ and $(\epsilon^1_{t})_{t=0,1,\ldots}$. Note that as a consequence:
	\begin{align*}
	x_{t+1}- x'_{t+1}
	&= x_t-  x'_t + \check{A} (x_t- x'_t) + \bar{A}(\bar{x}_t - \bar{x}'_t) 
	+ \sum_{i=1,2} \left[ B_i (u_{i,t}- u'_{i,t}) + \bar{B}_i( \bar{u}_{i,t}- \bar{u}'_{i,t})\right],
	\end{align*}
	which shows that $x_{t+1}- x'_{t+1}$ is in fact $\cF_t$-measurable. 
	As before, we use the convenient notations $\zeta_t=(x_t,\bar{x}_t,u_{1,t},\bar{u}_{1,t},u_{2,t},\bar{u}_{2,t})$ and $\zeta'_t=( x'_t,\bar{x}'_t, u'_{1,t},\bar{u}'_{1,t}, u'_{2,t},\bar{u}'_{2,t})$. In order to estimate $J(\bu'_1,\bu'_2)-J(\bu_1, \bu_2)$ we first notice that, if $\bp=(p_t)_{t=0,1,\ldots}$ is any sequence of real valued random variables, then:
	\begin{equation}
	\label{fo:first}
	\begin{split}
	\sum_{t=0}^\infty\gamma^t [c(\zeta'_t)-c(\zeta_t)]
	&=
	\sum_{t=0}^\infty\gamma^t \Bigl[c(\zeta'_t)-c(\zeta_t)+[b(\zeta'_t)-b(\zeta_t)]\cdot p_t-\delta(x'_t- x_t)\cdot p_t
	\\
	&\hskip 125pt
	-[b(\zeta'_t)-b(\zeta_t)]\cdot p_t+\delta(x'_t- x_t)\cdot p_t\Bigr]
	\\
	&=\sum_{t=0}^\infty\gamma^t \bigl[ h(\zeta'_t,p_t)- h( \zeta_t,p_t)\bigr]
	-\sum_{t=0}^\infty\gamma^t[(x'_{t+1}-  x_{t+1})-(x'_t- x_t)]\cdot p_t
	\\
	&\hskip 125pt +\delta\sum_{t=0}^\infty\gamma^t(x'_t- x_t)\cdot p_t\\
	&=\sum_{t=0}^\infty\gamma^t \bigl[ h(\zeta'_t,p_t)- h( \zeta_t,p_t)\bigr]
	+\sum_{t=0}^\infty\gamma^t(x'_{t+1}-  x_{t+1})\cdot (p_{t+1}-p_t)\\
	\end{split}
	\end{equation}
	where we used the fact that $\delta=(1-\gamma)/\gamma$.

	We now turn to computing the Gateaux derivative of $J$. Let $\bu_i,\bbeta_i$, $i=1,2$, as in the statement. To alleviate the notation, we denote
	$$
	V_t = \lim_{\epsilon \to 0} \frac{1}{\epsilon} \left( x^{\bu_1+\epsilon \bbeta_1, \bu_2+\epsilon \bbeta_2}_t - x^{\bu_1, \bu_2}_t \right),
	$$
	where $x^{\bu_1+\epsilon \bbeta_1, \bu_2+\epsilon \bbeta_2}$ is the state process controlled by $(\bu_1+\epsilon \bbeta_1, \bu_2+\epsilon \bbeta_2) \in \cU_{ad}^{open}$, and $x^{\bu_1, \bu_2}$ is the state process controlled by $(\bu_1, \bu_2) \in \cU_{ad}^{open}$. Let $(\bu'_1, \bu'_2) = (\bu_1 + \epsilon \bbeta_1, \bu_2 + \epsilon \bbeta_2)$.
	
	We then compute, using the expressions of the partial derivatives of $ h$ already computed in the proof of Lemma~\ref{le:MKV-convexity_ZS}
	\begin{equation}
	\label{fo:second_ZS}
	\begin{split}
	&D J(\bu_1, \bu_2)(\bbeta_1, \bbeta_2) = \lim_{\epsilon \to 0} \frac{1}{\epsilon} \left[ J(\bu_1 + \epsilon \bbeta_1, \bu_2 + \epsilon \bbeta_2) - J(\bu_1, \bu_2) \right]
	\\
	&= \sum_{t=0}^\infty \EE\left[ \gamma^t V_{t+1} \cdot (p_{t+1}-p_t) \right] + \sum_{t=0}^\infty\gamma^t \EE\Bigl[\partial_x  h(\zeta_t,p_t) V_t + \partial_{\bar{x}}  h(\zeta_t,p_t) \bar V_t
	\\
	&\qquad\qquad\qquad
	+\partial_{u_1}  h(\zeta_t,p_t) \beta_{1,t} + \partial_{\bar{u}_1}  h(\zeta_t,p_t) \bar \beta_{1,t}+\partial_{u_2}  h(\zeta_t,p_t) \beta_{2,t} + \partial_{\bar{u}_2}  h(\zeta_t,p_t) \bar \beta_{2,t}
	\Bigr]
	\\
	&= \sum_{t=0}^\infty \gamma^t \EE\Big[ V_{t+1} \cdot (p_{t+1}-p_t) 
	\\
	&\quad
	+ \left( p_t^{\top}(\check{A}-\delta I_d)+2(x_t-\bar{x}_t)^{\top}Q\right) V_t 
	+ \left(p_t^{\top}\bar{A} +2(\bar{x}_t-x_t)^{\top}Q+2\bar x_t^{\top}(Q+\bar{Q})\right) \bar V_t 
	\\
	&\quad
	+ \left(p_t^{\top}B_1+2(u_{1,t}-\bar{u}_{1,t})^{\top}R_1\right) \beta_{1,t} 
	+ \left(p_t^{\top}\bar{B}_1 +2(\bar{u}_{1,t}-u_{1,t})^{\top}R_1+2\bar u_{1,t}^{\top}(R_1+\bar{R}_1)\right) \bar{\beta}_{1,t}
	\\
	&\quad
	+\left(p_t^{\top}B_2 -2(u_{2,t}-\bar{u}_{2,t})^{\top}R_2\right) \beta_{2,t} 
	+ \left(p_t^{\top}\bar{B}_2 -2(\bar{u}_{2,t}-u_{2,t})^{\top}R_2 -2\bar u_{2,t}^{\top}(R_2+\bar{R}_2)\right) \bar{\beta}_{2,t} \Big]
	\\
	&=\sum_{t=0}^\infty\gamma^t \EE\bigl[(0) + (i)+(ii)+(iii)_1+(iv)_1 +(iii)_2+(iv)_2\bigr].
	\end{split}
	\end{equation}
	We now use Fubini's theorem to compute two of the six terms above. Recall that 
	$
	\bar V_t=\EE[V_t|\cF^0],
	$
	which we choose to express in the form
	$$
	\bar{V}_t
	= \tilde\EE[ \tilde V_t|\cF^0]
	= \int_{\tilde{\Omega}^1} \tilde{V}_t(\omega^0,\tilde\omega^1)\tilde{\PP}^1(d{\tilde \omega^1}),
	$$
	where $(\tilde{\Omega}^1,\tilde{\cF}^1,\tilde{\PP}^1)$ is an identical copy of $(\Omega^1,\cF^1,\PP^1)$ and the probability space 
	$(\tilde{\Omega},\tilde{\cF},\tilde{\PP})$ is defined as $\tilde\Omega=\Omega^0\times\tilde{\Omega}^1$, $\tilde{\cF}=\cF^0\times\tilde{\cF}^1$,and $\tilde{\PP}=\PP^0\times \tilde{\PP}^1$. For the sake of ease of notation, we introduce still another notation for the conditional expectations: we shall denote by $\EE_{\cF^0}$ and $\tilde\EE_{\cF^0}$ the conditional expectations usually denoted by $\EE[\,\cdot\,|\cF^0]$ and $\tilde\EE[\,\cdot\,|\cF^0]$ respectively. With this new notation $\bar{x}_t=\EE_{\cF^0}[x_t]=\tilde \EE_{\cF^0}[\tilde{x}_t]=\overline{\tilde{x}}_t$ and similarly for the other random variables. Consequently:
	
	\begin{equation*}
	\begin{split}
	\EE\sum_{t=0}^\infty\gamma^t (ii) 
	&= \EE \EE_{\cF^0}  \sum_{t=0}^\infty\gamma^t \bigl(p_t^{\top}\bar{A} +2(\bar{x}_t-x_t)^{\top}Q+2\bar x_t^{\top}(Q+\bar{Q})\bigr) \bar V_t 
	\\ 
	&= \EE \EE_{\cF^0} \tilde\EE_{\cF^0}  \sum_{t=0}^\infty\gamma^t \bigl(p_t^{\top}\bar{A} +2(\bar{x}_t-x_t)^{\top}Q+2\bar x_t^{\top}(Q+\bar{Q})\bigr) \tilde V_t 
	\\
	&= \EE\sum_{t=0}^\infty\gamma^t\bigl(\bar p_t^{\top}\bar A +2\bar{x}_t^{\top}(Q+\bar{Q})\bigr) V_t,
	\end{split}
	\end{equation*}
	where we used Fubini's theorem for the last equality. So:
	\begin{equation}
	\label{fo:i+ii_ZS}
	\begin{split}
	\EE\sum_{t=0}^\infty\gamma^t [(i)+(ii)] &=\EE\sum_{t=0}^\infty\gamma^t \bigl( p_t^{\top}(\check{A}-\delta I_d) + 2(x_t-\bar{x}_t)^{\top}Q
	+\bar p_t^{\top}\bar{A} +2\bar{x}_t^{\top}(Q+\bar{Q})\bigr) V_t
	\\
	&=\EE\sum_{t=0}^\infty\gamma^t \bigl(p_t^{\top}(\check{A}-\delta I_d) + 2x_t^{\top}Q+\bar p_t^{\top}\bar{A} +2\bar{x}_t^{\top}\bar{Q}\bigr) V_t.
	\end{split}
	\end{equation}
	As a consequence, 
	$$
	\sum_{t=0}^\infty\gamma^t \EE\bigl[(0) + (i)+(ii)\bigr]
	=
	0
	$$
	because of \eqref{fo:i+ii_ZS} and the definition \eqref{fo:MKV-adjoint_ZS} of the adjoint process.

	Furthermore, using Fubini's theorem on an identical copy of $\bu_1$ and $\bbeta_1$ we get:
	\begin{equation}
	\label{fo:iii+iv_1}
	\begin{split}
	\EE\sum_{t=0}^\infty\gamma^t [(iii)_1+(iv)_1] 
	&=
	\EE\sum_{t=0}^\infty\gamma^t \bigl( p_t^{\top} B_1 +2 u_{1,t}^{\top} R_1 + \bar p_t^{\top}\bar{B}_1 +2\bar{u}_{1,t}^{\top}\bar{R}_1\bigr) \beta_{1,t},
	\end{split}
	\end{equation}
	and likewise for $\bu_2, \bbeta_2$. 
	
\end{proof}

We are now in a position to prove the following condition for optimality:

\begin{proposition}[Pontryagin's maximum principle, necessary condition]
\label{proposition:Pontryagin_maximum_principle_necessary_condition}
	Assuming that Assumption~\ref{assumption:A_openloop} holds, if $\bu_i=(u_{i,t})_{t=0,1,\ldots},$ $i=1,2$ is a pair of admissible control processes such that it is an open-loop Nash equilibrium for the zero-sum game and $\bp=(p_t)_{t=0,1,\ldots}$ is the corresponding adjoint process, then it holds
	\begin{equation}
	\label{fo:MKV-adjoint_ass_ZS}
	\begin{cases}
	& B_1^{\top} p_t + 2 R_1 u_{1,t} + \bar{B}_1^{\top} \bar{p}_t + 2\bar{R}_1 \bar{u}_{1,t} = 0 
	\\
	& B_2^{\top} p_t - 2 R_2 u_{2,t} + \bar{B}_2^{\top} \bar{p}_t - 2\bar{R}_2 \bar{u}_{2,t} = 0 
	\end{cases}
	\end{equation}
	for all $t\ge 0$, $\PP$-almost surely. 
\end{proposition}

\begin{proof}
	By Lemma~\ref{lem:differential-J}, for any pair of processes $(\bbeta_1, \bbeta_2) \in \cU \times \cU$ we have the 
	\begin{equation*}
		\begin{split}
		D J(\bu_1, \bu_2)(\bbeta_1, \bbeta_2)
		&=
		\EE \left[ \sum_{t=0}^\infty\gamma^t \bigl( p_t^{\top} B_1 +2 u_{1,t}^{\top} R_1 + \bar p_t^{\top}\bar{B}_1 +2\bar{u}_{1,t}^{\top}\bar{R}_1\bigr)  \beta_{1,t}
		\right]
		\\
		& \qquad +
		\EE \left[ \sum_{t=0}^\infty\gamma^t \bigl( p_t^{\top} B_2 -2 u_{2,t}^{\top} R_2 + \bar p_t^{\top}\bar{B}_2 -2\bar{u}_{2,t}^{\top}\bar{R}_2\bigr)  \beta_{2,t}
	\right]\, .
		\end{split}
		\end{equation*}
	Since $(\bu_1, \bu_2) \in \cU_{ad}^{open}$ is an open-loop Nash equilibrium for the zero-sum game, then for every $\bu_1' \in \cU$ and $\bu_2' \in \cU$, we have 
	\begin{equation*}
		J( \bu_1, \bu'_2) \leq J(\bu_1, \bu_2) \leq J(\bu'_1, \bu_2).
	\end{equation*}
	Let us denote the state processes in the above inequalities by $\bX^{\bu_1, \bu'_2}, \bX^{\bu_1, \bu_2}, \bX^{\bu_1', \bu_2}$, which are all $L^2-$discounted globally integrable according to Proposition~\ref{prop:stability_with_extra_L2_term}.
	
	If we choose $\bbeta_2 = 0$, then for every $\bbeta_1 \in \cU$, 
	$$
		DJ(\bu_1, \bu_2)(\bbeta_1, 0) = \lim_{\epsilon \to 0} \frac{1}{\epsilon} \left[ J(\bu_1 + \epsilon \bbeta_1, \bu_2) - J(\bu_1, \bu_2) \right] \geq 0
	$$
	which implies, 
	\begin{equation*}
		\EE \left[ \sum_{t=0}^\infty\gamma^t \bigl( p_t^{\top} B_1 +2 u_{1,t}^{\top} R_1 + \bar p_t^{\top}\bar{B}_1 +2\bar{u}_{1,t}^{\top}\bar{R}_1\bigr)  \beta_{1,t}
		\right] \geq 0.
	\end{equation*}
	Thus, the corresponding adjoint process $\bp$ satisfies: $\PP-$almost surely, for every $t \geq 0$,
	\begin{equation*}
		B_1^{\top} p_t + 2 R_1 u_{1,t} + \bar{B}_1^{\top} \bar{p}_t + 2\bar{R}_1 \bar{u}_{1,t} = 0.
	\end{equation*}
	Similarly, we have 
	$
		B_2^{\top} p_t - 2 R_2 u_{2,t} + \bar{B}_2^{\top} \bar{p}_t - 2\bar{R}_2 \bar{u}_{2,t} = 0 
	$
	for every $t \geq 0$, $\PP$-almost surely.
\end{proof}


\subsection{Identification of the equilibrium}

Let us introduce the notations
\begin{equation}
\label{eq:def-GammaXiLambda_i}
\begin{cases}
&\Gamma_i=(-1)^{i}\frac12 R_i^{-1}B_i^{\top},
\qquad 
\Xi_i =  (-1)^{i}\frac12 R_i^{-1}\bigl[\bar{B}_i^{\top} - \bar{R}_i(R_i+\bar R_i)^{-1}(B_i+\bar{B}_i)^{\top}\bigr],
\\
&\Lambda_i =\Gamma_i+\Xi_i= (-1)^{i} \frac12 (R_i+\bar R_i)^{-1}(B_i+\bar{B}_i)^{\top}, i=1,2.
\end{cases}
\end{equation} 
We then  consider the following Riccati equations: \commentJDG{the Riccati equations introduced in~\eqref{eq:main_ARE_P_ZS}--\eqref{eq:main_ARE_Pbar_ZS}, namely:}
\begin{equation} 
\label{eq:main_ARE_P_ZS}
\gamma [A^{\top} P + 2Q]\left( A + \big(B_1 \Gamma_1 + B_2 \Gamma_2\big)  P \right) = P,
\end{equation}
and
\begin{equation}
\label{eq:main_ARE_Pbar_ZS} 
\gamma\bigl[ (A^{\top}+\bar{A}^{\top})\bar{P}+2(Q+\bar{Q})\bigr]\left[(A+\bar{A}) +\Big((B_1+\bar{B}_1)\Lambda_1 + (B_2+\bar{B}_2)\Lambda_2\Big)\bar{P} \right] =\bar{P}.
\end{equation}
We shall assume that there exists solutions $P, \bar P\in \cS^d$ to these equations, which can be proved under suitable conditions. (for example, by contraction arguments with ``small coefficients" of the problem.) We also discuss in section~\ref{sec:connect-open-closed} a way to construct $P$ and $\bar P$ with the help of other Algebraic Riccati equations.

\begin{proposition}
\label{proposition:open_loop_Nash_expression}
Assume there exists $P,\bar P$ solving~\eqref{eq:main_ARE_P_ZS}--\eqref{eq:main_ARE_Pbar_ZS}. 
The following pair of controls satisfies the equilibrium condition~\eqref{fo:MKV-adjoint_ass_ZS}:
\begin{equation}
    \label{eq:MKV-opt-ctrl-formula}
		u_{i,t} = \Gamma_i P(x_t - \bar x_t) + \Lambda_i \bar P \bar x_t,
\end{equation}
where $x_t$ is controlled by $u_{i,t}, i=1,2$ and $\bar x_t = \EE_{\cF^0}[x_t]$.

Moreover, the process $\bp = (p_t)_{t \geq 0}$ defined by: 
\begin{equation}
	 \label{eq:def_pt_Riccati}
	    p_t = P (x_t - \bar{x}_t) + \bar{P} \bar{x}_t, \qquad t \ge 0,
\end{equation}
is an adjoint process satisfying~\eqref{fo:MKV-adjoint_ZS}.
\end{proposition}

\begin{proof}
	Let $\bZ^0$ and $\bZ^1$ be the processes defined by
	$$
	Z^0_t = \gamma [ (A^{\top}+\bar{A}^{\top}) \bar{P} +2(Q+\bar{Q}) ],
	\qquad
	Z^1_{t} = \gamma [A^{\top} P  + 2Q].
	$$
	Notice that such processes are deterministic (hence predictable) and independent of time. 
	
	Let us consider a process $\bp$ satisfying
	\begin{equation}
	\label{fo:adjoint_bsde_special_ZS}
	p_{t+1}-p_t
	=-\gamma \bigl[ (\check{A}^{\top}- \delta I_d)p_{t+1}+2Qx_{t+1} + \bar A^{\top}\bar p_{t+1}+2\bar Q\bar x_{t+1}\bigr]  + Z^0_{t+1}\epsilon^0_{t+1}+Z^1_{t+1}\epsilon^1_{t+1}.
	\end{equation}
	Then, $\bp$ also satisfies~\eqref{fo:MKV-adjoint_ZS} so it is an adjoint process.
	
	We now rewrite the equilibrium condition \eqref{fo:MKV-adjoint_ass_ZS}. Taking conditional expectations $\EE_{\cF^0}$ in the first equation, we get:
	$$
	(B_1+\bar{B}_1)^{\top}\bar{p}_t+2(R_1+\bar{R}_1)\bar{u}_{1,t}=0
	$$
	from which we derive:
	\begin{equation}
	\label{fo:u_t_bar_ZS_1}
	\bar{u}_{1,t}= - \frac12 (R_1+\bar R_1)^{-1}(B_1+\bar{B}_1)^{\top}\bar{p}_t.
	\end{equation}
	Plugging this expression back into the first equation of \eqref{fo:MKV-adjoint_ass_ZS} we get:
	$$
	B_1^{\top} p_t+2 R_1 u_{1,t}+\bar{B}_1^{\top}\bar{p}_t - \bar{R}_1(R_1+\bar R_1)^{-1}(B_1+\bar{B}_1)^{\top}\bar{p}_t=0
	$$
	from which we deduce:
	\begin{equation}
	\label{fo:u_t_fct_y_t_ZS_1}
	u_{1,t}
	= \Gamma_1 p_t + \Xi_1 \bar{p}_t,
	\qquad\text{and}\qquad 
	\bar{u}_{1,t} 
	= \Lambda_1 \bar{p}_t
	\end{equation}
	for $\Gamma_1, \Xi_1, \Lambda_1$ introduced in~\eqref{eq:def-GammaXiLambda_i}. Similarly, we find 
	\begin{equation}
	\label{fo:u_t_fct_y_t_ZS_2}
	u_{2,t}
	= \Gamma_2 p_t + \Xi_2 \bar{p}_t,
	\qquad\text{and}\qquad 
	\bar{u}_{2,t} 
	= \Lambda_2 \bar{p}_t.
	\end{equation}
	
	\vskip 2pt\noindent
	
	Plugging expressions \eqref{fo:u_t_fct_y_t_ZS_1}, \eqref{fo:u_t_fct_y_t_ZS_2} for $u_{1,t}, u_{2,t}$ and $\bar{u}_{1,t}, \bar{u}_{2,t}$ into the state dynamics equation \eqref{eq:MKV-state_ZS} and the definition \eqref{fo:adjoint_bsde_special_ZS} of the adjoint process we find that at the equilibrium, the system should satisfy, for $t=0,1,\ldots$,
	\begin{equation}
	\label{fo:fbsde1}
	\begin{cases}
	&x_{t+1}-x_t = \check{A} x_t +\bar{A}\bar{x}_t+(B_1\Gamma_1 + B_2\Gamma_2) p_t
	\\
	& \qquad\qquad\qquad\qquad\qquad+(\bar{B}_1\Lambda_1 +B_1\Xi_1 + \bar{B}_2\Lambda_2 +B_2\Xi_2)\bar{p}_t +\epsilon^0_{t+1}+\epsilon^1_{t+1},
	\\
	&p_{t+1}-p_t=-\gamma \bigl[(\check{A}^{\top}- \delta I_d)p_{t+1} + 2Qx_{t+1} + \bar A^{\top}\bar p_{t+1}+2\bar Q\bar x_{t+1} \bigr]+Z^0_{t+1}\epsilon^0_{t+1}
	\\
	&\qquad\qquad\qquad\qquad\qquad+Z^1_{t+1}\epsilon^1_{t+1}. 
	\end{cases}
	\end{equation}
	Taking conditional expectations $\EE_{\cF^0}$ on all sides we get: for $t=0,1,\ldots$
	\begin{equation}
	\label{fo:fbsde-mean1}
	\begin{cases}
	&\bar{x}_{t+1}-\bar{x}_t = (\check{A}+\bar{A}) \bar{x}_t + \Big[ (B_1+\bar B_1)\Lambda_1 + (B_2+\bar B_2)\Lambda_2 \Big] \bar{p}_t +\epsilon^0_{t+1}, 
	\\
	&\bar{p}_{t+1}-\bar{p}_t = -\gamma[(\check{A}^{\top}+\bar{A}^{\top}-\delta I_d) \bar{p}_{t+1} +2(Q+\bar{Q})\bar{x}_{t+1}] +Z^0_{t+1}\epsilon^0_{t+1},  
	\end{cases}
	\end{equation}
	which is a linear forward-backward stochastic difference equation in infinite horizon on the probability space $(\Omega^0,\cF^0,\PP^0)$ of the common noise.
	
	 Using the solutions $P,\bar P$ solving~\eqref{eq:main_ARE_P_ZS}--\eqref{eq:main_ARE_Pbar_ZS}, let us define 
	 \begin{equation*}
	    p_t = P (x_t - \bar{x}_t) + \bar{P} \bar{x}_t.
	 \end{equation*}
	
	We now check the pair of process $(x_t, p_t)_{t \geq 0}$ provides a solution to~\eqref{fo:fbsde1}, so that $(p_t)_{t \geq 0}$ given by equation~\eqref{eq:def_pt_Riccati} is an adjoint process. 
	
	First, from our choice of $\bZ^0$ and using the fact that $\bar{P}$ solves~\eqref{eq:main_ARE_Pbar_ZS}, there holds: for $t=0,1,\ldots$
	\begin{equation*} 
	\begin{cases}
	&\bar{x}_{t+1} =  \left[(A+\bar{A}) +\Big[(B_1+\bar{B}_1)\Lambda_1 + (B_2+\bar{B}_2)\Lambda_2 \Big]\bar{P} \right] \bar{x}_t +\epsilon^0_{t+1},  
	\\
	&\gamma\bigl[ (A^{\top}+\bar{A}^{\top})\bar{P}+2(Q+\bar{Q})\bigr] \bar{x}_{t+1} 
	= \bar{P}\bar{x}_t + Z^0_{t+1}\epsilon^0_{t+1}, 
	\end{cases}
	\end{equation*}
	or equivalently
	\begin{equation*}
	\begin{cases}
	&\bar{x}_{t+1}-\bar{x}_t = \left[(\check{A}+\bar{A}) +\Big[(B_1+\bar{B}_1)\Lambda_1 + (B_2+\bar{B}_2)\Lambda_2 \Big] \bar{P} \right] \bar{x}_t +\epsilon^0_{t+1},  
	\\
	&\bar{P}(\bar{x}_{t+1}-\bar{x}_t) = -\gamma\left[(\check{A}^{\top}+\bar{A}^{\top}-\delta I_d) \bar{P} +2(Q+\bar{Q})\right]\bar{x}_{t+1} +Z^0_{t+1}\epsilon^0_{t+1}.  
	\end{cases}
	\end{equation*}
	Combining with $\bar p_t = \bar P \bar x_t$, we deduce that~\eqref{fo:fbsde-mean1} is satisfied.
	
	We proceed similarly for the other equation. 
	\commentJDG{Using our choice of $\bZ^1$
 and the fact that $P$ solves~\eqref{eq:main_ARE_P_ZS}, we have: for $t=0,1,\ldots$
	\begin{equation*}
	\begin{cases}
	&(x_{t+1}-\bar{x}_{t+1}) = ( A + ( B_1 \Gamma_1 + B_2 \Gamma_2 )   P )  (x_t - \bar{x}_t) + \epsilon^1_{t+1},  
	\\
	&\gamma [A^\top P+2Q](x_{t+1}-\bar{x}_{t+1}) =  P (x_t - \bar{x}_t) +Z^1_{t+1} \epsilon^1_{t+1}, 
	\end{cases}
	\end{equation*} 
	or equivalently
	\begin{equation*}
	\begin{cases}
	&(x_{t+1}-\bar{x}_{t+1}) - (x_t - \bar{x}_t) = ( \check{A} + ( B_1 \Gamma_1 + B_2 \Gamma_2 ) P )  (x_t - \bar{x}_t) + \epsilon^1_{t+1},  
	\\
	&  P [(x_{t+1}-\bar{x}_{t+1}) - (x_t - \bar{x}_t)] = - \gamma [(\check{A}^\top-\delta I_d)   P  + 2Q ](x_{t+1} - \bar{x}_{t+1}) + Z^1_{t+1}\epsilon^1_{t+1}, 
	\end{cases}
	\end{equation*}
	from which we deduce, using~\eqref{eq:def_pt_Riccati} and~\eqref{fo:fbsde-mean1}, that~\eqref{fo:fbsde1} is satisfied.
	
	For every $t \geq 0$, by replacing the definition of $p_t$ in equations~\eqref{fo:u_t_fct_y_t_ZS_1}--\eqref{fo:u_t_fct_y_t_ZS_2}, we obtain the expression~\eqref{eq:MKV-opt-ctrl-formula_ZS} for $(u_{1,t}, u_{2,t})$.
	}
	
\end{proof}

\subsection{A Convexity-concavity sufficient condition}

Consider two deterministic processes $\bV_1 = (V_{1,t})_{t \geq 0}$ and $\bV_2 = (V_{2,t})_{t \geq 0}$ following the dynamics
\begin{subequations}
    \begin{empheq}[left=\empheqlbrace]{align}
        V_{1, t+1} &= A V_{1,t} + \bar{A} \bar V_{1,t} + B_1 \beta_{1,t} + \bar B_1 \bar{\beta}_{1,t}, \qquad V_{1, t=0} = 0,
        \label{eq:dyn_V1}\\
        V_{2, t+1} &= A V_{2,t} + \bar{A} \bar V_{2,t} + B_2 \beta_{2,t} + \bar B_2 \bar{\beta}_{2,t}, \qquad V_{2, t=0} = 0,
        \label{eq:dyn_V2}
    \end{empheq}
\end{subequations}
where $(\bbeta_1, \bbeta_2) \in \cU \times \cU$ are two $L^2-$integrable control processes. Under Assumption~\ref{assumption:A_openloop}, by Lemma~\ref{lemma:L2_x_and_x_bar}, we have $\bV_1 \in \cX$ and  $\bV_2 \in \cX$.

\begin{proposition}[Pontryagin's maximum principle, sufficient condition]
\label{prop:sufficient_Pontryagin}
We assume the following conditions:
\begin{enumerate}
    \item There exists a state process $\bx = (x_t)_{t \geq 0}$ and its adjoint processes $(\bp, \bZ^0, \bZ^1) = ( p_t, Z^0_t, Z^1_t)_{t \geq 0}$ such that $\bx, \bp$ are $(\cF_t)_{t \geq 0}$-adapted and $(Z^0_t, Z^1_t)_{t \geq 1}$ are $(\cF_t)_{t \geq 0}$-predictable processes, and they satisfy the forward-backward system of equations: for every $t \geq 0$,
    \begin{equation}
    \label{eq:FBSDE_pontryagin_sufficient}
    \left\{ 
    \begin{array}{rl}
      x_{t+1} & = A x_t + \bar{A} \bar{x}_t + ( B_1 \Gamma_1 + B_2 \Gamma_2) p_t  
      \\
      &\qquad + \Big( (B_1 + \bar{B}_1) \Lambda_1 + (B_2 + \bar B_2) \Lambda_2 - B_1 \Gamma_1 - B_2 \Gamma_2 \Big) \bar p_t  
    + \epsilon^0_{t+1} + \epsilon^1_{t+1},      
         \\
	    p_t &= \gamma \left( A^\top p_{t+1} + 2Q x_{t+1} + \bar{A}^{\top}\bar{p}_{t+1}+2\bar{Q}\bar{x}_{t+1} \right) + Z_{t+1}^0 \epsilon_{t+1}^0 + Z_{t+1}^1 \epsilon^1_{t+1}
	  \end{array}
	\right.
	\end{equation}
	with initial values $x_0 = \epsilon_0^0 + \epsilon_0^1$ and $Z_0^0 = Z_0^1 = 0$.

    \item For any control processes $(\bbeta_1, \bbeta_2) \in \cU \times \cU$, we have the following convexity-concavity condition for the zero-sum game:
        \begin{align}
            & \EE\left[ \sum_{t=0}^\infty \gamma^t \left( V_{1,t}^\top Q V_{1,t}  + \bar V_{1,t}^\top  \bar{Q} \bar V_{1,t} + \beta_{1,t}^\top R_1 \beta_{1,t} + \bar \beta_{1,t}^\top  \bar{R}_1 \bar{\beta}_{1,t} \right) \right] \geq 0
            \label{eq:cond_convex_1}\\
            & \EE\left[ \sum_{t=0}^\infty \gamma^t \left( V_{2,t}^\top Q V_{2,t} + \bar V_{2,t}^\top \bar Q \bar V_{2,t} - \beta_{2,t}^\top R_2 \beta_{2,t}  - \bar \beta_{2,t}^\top \bar{R}_2 \bar{\beta}_{2,t} \right) \right] \leq 0
            \label{eq:cond_concave_2}
    \end{align}
      where the processes $(\bV_1, \bV_2) \in \cX \times \cX$ follows the dynamics defined in \eqref{eq:dyn_V1} and \eqref{eq:dyn_V2}.
\end{enumerate}
    Then, the pair of control processes $(\bu_1, \bu_2)\in \cU_{ad}^{open}$ given by:
    \begin{equation}
        \label{eq:open_loop_Nash_expression_BSDE}
                 u_{i,t} = \Gamma_i p_t + (\Lambda_i - \Gamma_i) \bar p_t, \qquad i=1,2
    \end{equation}
    is an open-loop Nash equilibrium for the zero-sum game. Moreover, $(\bu_1, \bu_2)$ satisfies the equilibrium condition \eqref{fo:MKV-adjoint_ass_ZS} of the Pontryagin maximum principle.
\end{proposition}

\begin{proof}
    The backward equation for process $\bp$ implies that it satisfies the conditional expectation condition \eqref{fo:MKV-adjoint_ZS}. In the proof of Proposition~\ref{proposition:open_loop_Nash_expression}, we show with equations~\eqref{fo:u_t_fct_y_t_ZS_1}--\eqref{fo:u_t_fct_y_t_ZS_2} that the pair of control processes $(\bu_1, \bu_2)$ defined by equation \eqref{eq:open_loop_Nash_expression_BSDE} satisfies the equilibrium condition \eqref{fo:MKV-adjoint_ass_ZS}. By substituting the right hand side of~\eqref{eq:open_loop_Nash_expression_BSDE} with $(u_{1,t}, u_{2,t})$ in the forward equation for $(x_t)_{t \geq 0}$ in~\eqref{eq:FBSDE_pontryagin_sufficient}, we get that the process $\bx$ follows dynamics~\eqref{eq:MKV-state_ZS} which is controlled exactly by $(\bu_1, \bu_2) \in \cU_{ad}^{open}$.

Based on the proof of Lemma~\ref{lem:differential-J} for the Gateaux derivative of $J$, we write a second-order expansion for the value function $J$ at a point $(\bu_1, \bu_2) \in \cU_{ad}^{open}$ in the direction $(\bbeta_1, \bbeta_2) \in \cU \times \cU$. To alleviate the notation, we introduce a deterministic process $\bV = (V_t)_{t \geq 0}$ following a dynamics
    \begin{equation*}
        V_{t+1} = A V_t + \bar{A} \bar V_t + B_1 \beta_1 + \bar B_1 \bar \beta_1 + B_2 \beta_2 + \bar B_2\beta_2
    \end{equation*}
    with initial value $V_0 = 0$. The linearity of the dynamics \eqref{eq:MKV-state_ZS} shows that $V_t = (x^{\bu_1 + \epsilon \bbeta_1, \bu_2 + \epsilon \bbeta_2}_t - x^{\bu_1, \bu_2}_t) / \epsilon$ for every $\epsilon > 0$ and $t \geq 0$.
    According to equation \eqref{fo:first}, the difference between the values of $J$ at point $(\bu_1 + \epsilon \bbeta_1, \bu_2 + \epsilon \bbeta_2)$ and at point $(\bu_1, \bu_2)$ can be expressed by 
    \begin{equation}
    \label{eq:difference_J_epsilon_beta}
        \begin{aligned}
            & J (\bu_1 + \epsilon \bbeta_1, \bu_2 + \epsilon \bbeta_2) - J(\bu_1, \bu_2) 
            \\
            = & \epsilon \left( \sum_{t=0}^\infty \EE \left[ \gamma^t V_{t+1} \cdot (p_{t+1} - p_t)\right] + \sum_{t=0}^\infty  \gamma^t  \EE \left[\nabla_{\zeta} h(\zeta_t, p_t) \cdot \check{\zeta}_t \right] \right)
            \\
            &\qquad\qquad\qquad\qquad+ \frac{1}{2} \epsilon^2 \sum_{t=0}^{\infty} \gamma^t \EE \Big[ \nabla^2_{\zeta \zeta} h( \eta_t, p_t) \check{\zeta}_t \cdot \check{\zeta}_t \Big] 
            \\
            = & \epsilon (i) + \epsilon^2  (ii)
        \end{aligned}
    \end{equation}
    where $\check{\zeta}_t = (\zeta'_t - \zeta_t) / \epsilon = \left[ V_t^\top, \bar V_t^\top, \beta_{1,t}^\top, \bar \beta_{1,t}^\top, \beta_{2,t}^\top, \bar \beta_{2,t}^\top \right]^\top \in \RR^{2d \times 4\ell}$ and $\eta_t = (1-\lambda_t) \zeta_t' + \lambda_t \zeta_t \in \RR^{2d + 4\ell}$ for some $\lambda_t \in [0,1]$. 
    Since we assume that the pair of admissible control processes $(\bu_1, \bu_2)$ satisfies the system of equations \eqref{fo:MKV-adjoint_ass_ZS} at every time $t\geq 0$, then by applying Lemma~\ref{lem:differential-J}, we have $(i)=0$. We also notice that the Hessian  \commentJDG{\eqref{fo:hessian}} of the Hamiltonian function $h(\cdot)$ with respect to $\zeta$ is a constant matrix depending only on model parameters. Thus, we obtain
    \begin{equation}
    \label{eq:hessian_h_with_V}
    \begin{aligned}
         (ii) & =  \sum_{t=0}^\infty \gamma^t \EE \Big[ V_{t}^\top Q V_{t} + \bar V_{t}^\top \bar (2 Q+ \bar Q) \bar V_{t} - V_t^\top Q \bar V_t - \bar{V}_t^\top Q V_t 
          + \beta_{1, t}^\top R_1 \beta_{1, t}  
          \\
          &\qquad +\bar \beta_{1,t}^\top (2 R_1 + \bar{R}_1) \bar{\beta}_{1,t} 
            -  \beta_{1,t}^\top R_1 \bar \beta_{1,t} - \bar \beta_{1,t}^\top R_1  \beta_{1,t}  -  \beta_{2, t}^\top R_2 \beta_{2, t}  
            \\
            &\qquad - \bar \beta_{2,t}^\top (2 R_2 + \bar{R}_2) \bar{\beta}_{2,t} +  \beta_{2,t}^\top R_2 \bar \beta_{2,t} + \bar \beta_{2,t}^\top R_2  \beta_{2,t} \Big] 
          \\
          & = \sum_{t=0}^\infty \gamma^t \EE \Big[ V_t^\top Q V_t+ \bar V_t^\top \bar Q \bar V_t + \beta_{1, t}^\top R_1 \beta_{1, t}  + \bar \beta_{1,t}^\top \bar{R}_1 \bar{\beta}_{1,t} -  \beta_{2, t}^\top R_2 \beta_{2, t}  
           - \bar \beta_{2,t}^\top \bar{R}_2 \bar{\beta}_{2,t} \Big].
    \end{aligned}
    \end{equation}
    
    Consider a fixed control process $\bu_2$ for player 2. For every control process $\bu'_1 \in \cU$, we choose $\bbeta_1 = (\bu_1' - \bu_1) / \epsilon \in \cU$ and $\bbeta_2 = 0$. The convexity condition \eqref{eq:cond_convex_1}, together with equations \eqref{eq:difference_J_epsilon_beta} and \eqref{eq:hessian_h_with_V}, yield that 
    \begin{align*}
        &J(\bu_1 + \epsilon \bbeta_1, \bu_2) - J( \bu_1, \bu_2) 
        \\
        &= \epsilon^2 \EE\left[ \sum_{t=0}^\infty \gamma^t \Big(  V_{1,t}^\top Q V_{1,t}  + \bar V_{1,t}^\top  \bar{Q} \bar V_{1,t} + \beta_{1,t}^\top R_1 \beta_{1,t} + \bar \beta_{1,t}^\top  \bar{R}_1 \bar{\beta}_{1,t} \Big) \right] \geq 0,
    \end{align*}
    where the process $\bV_1 = (V_{1,t})_{t \geq 0}$ follows the dynamics \eqref{eq:dyn_V1}. Consequently, we have for every $\bu'_1 \in \cU$, 
    $$  
        J(\bu'_1, \bu_2) \geq J(\bu_1, \bu_2).
    $$
    Similarly, the concavity condition \eqref{eq:cond_concave_2} implies that, for every $\bbeta_2 \in \cU$,
    \begin{align*}
        &J(\bu_1, \bu_2 + \epsilon \bbeta_2) - J( \bu_1, \bu_2) 
        \\
        &=  \epsilon^2 \EE\left[ \sum_{t=0}^\infty \gamma^t \Big(  V_{2,t}^\top Q V_{2,t} + \bar V_{2,t}^\top \bar Q \bar V_{2,t} - \beta_{2,t}^\top R_2 \beta_{2,t}  - \bar \beta_{2,t}^\top \bar{R}_2 \bar{\beta}_{2,t} \Big) \right] \leq 0,
    \end{align*}
    where $\bV_2 = (V_{2,t})_{t \geq 0}$ follows the dynamics \eqref{eq:dyn_V2}. Thus, with the same argument, we have for every $\bu_2' \in \cU$, $J(\bu_1, \bu_2') \leq J(\bu_1, \bu_2)$. 
    
    Therefore, we conclude that under the convexity-concavity condition for the two processes $(\bV_1, \bV_2)$, a pair of control processes $(\bu_1, \bu_2)$ satisfying the system of equations \eqref{fo:MKV-adjoint_ass_ZS} with adjoint process $\bp$ is an open-loop Nash equilibrium for the zero-sum game.
\end{proof}

\begin{remark}
 \label{remark:uniqueness_adjoint} Assume that the convexity-concavity condition for the zero-sum game holds true under a suitable choice of model parameters.
If there exists a unique solution $(\bx, \bp)$ to the forward-backward system of equations~\eqref{eq:FBSDE_pontryagin_sufficient}, then by the necessary condition of the Pontryagin maximum principle proven in Proposition~\ref{proposition:Pontryagin_maximum_principle_necessary_condition}, the open-loop Nash equilibrium given by~\eqref{eq:open_loop_Nash_expression_BSDE} is the unique one.
In this case, if we suppose in addition that the Riccati equations~\eqref{eq:main_ARE_P_ZS}--\eqref{eq:main_ARE_Pbar_ZS} admit unique solutions $P$ and $\bar P$, then by Proposition~\ref{proposition:open_loop_Nash_expression} the adjoint process $\bp$ must be a linear function of $x_t$ and $\bar x_t$ defined by equation~\eqref{eq:def_pt_Riccati}, namely $p_t = P (x_t - \bar x_t) + \bar P \bar x_t$ for every $t \geq 0$.
\end{remark}

\begin{remark}
    We can see from equations \eqref{eq:difference_J_epsilon_beta} and \eqref{eq:hessian_h_with_V} that the convexity-concavity condition is also a necessary condition if $(\bu_1, \bu_2) \in \cU_{ad}^{open}$ is an open-loop Nash equilibrium for the zero-sum game.
\end{remark}

Taking a closer look at the convexity condition \eqref{eq:cond_convex_1} (resp. the concavity condition \eqref{eq:cond_concave_2}), it is indeed a quadratic function of the process $\bV_1$ (resp. $\bV_2$) and the control $\bbeta_1 \in \cU$ (resp. $\bbeta_2 \in \cU$). So, we can apply results from the deterministic Linear-Quadratic control problems to derive a sufficient condition for the convexity-concavity condition. Let us define some new value functions $C_{V_i}(\bbeta_i - \bar \bbeta_i)$ and $\bar C_{V_i}(\bar \bbeta_i)$ for $i=1,2$:
\begin{align*}
    C_{V_i}(\bbeta_i - \bar{\bbeta}_i ) 
    & =  \EE\Big[ \sum_{t=0}^\infty  \gamma^t \Big( ( V_{i,t} - \bar V_{i,t})^\top \Big( (-1)^{i-1} Q \Big) (V_{i,t} - \bar V_{i,t})  
    \\
    &\qquad\qquad+  (\beta_{i,t} - \bar \beta_{i,t}) ^\top R_i (\beta_{i,t} - \bar \beta_{i,t} ) \Big) \Big],
    \\
    \bar C_{V_i}(\bar \bbeta_i) & =  \EE \left[ \sum_{t=0}^\infty \gamma^t \left( \bar V_{i,t}^\top \Big( (-1)^{i-1} (Q + \bar{Q}) \Big) \bar V_{i,t} + \bar \beta_{i,t}^\top  ( R_i + \bar{R}_i) \bar{\beta}_{i,t} \right) \right] \ .
\end{align*}
Then the convexity-concavity condition~\eqref{eq:cond_convex_1}--\eqref{eq:cond_concave_2} is equivalent to 
\begin{equation*}
    \min_{\bbeta_1 \in \cU} C_{V_1}(\bbeta_1 - \bar \bbeta_1) + \bar C_{V_1}(\bar \bbeta_1)  \geq 0, \qquad \text{and} \qquad  \min_{\bbeta_2 \in \cU} C_{V_2}(\bbeta_2 - \bar \bbeta_2) + \bar C_{V_2}(\bar \bbeta_2) \geq 0. 
\end{equation*}
Here, we multiply $Q$ and $Q + \bar Q$ by $-1$ for player $i=2$ so that the concavity condition is connected to a minimization problem. 
Let us assume that the following discrete Algebraic Riccati equation (DARE-i):
\begin{equation}
\label{eq:DARE}
	0 = (-1)^{i-1} Q - P_i + \gamma A^\top P_i A - \gamma^2 A^\top P_i B_i \Big( \gamma B_i^\top P_i B_i + R_i \Big)^{-1} B_i^\top P_i A,
\end{equation}
 admits a symmetric matrix $P_i \in \mathcal{S}^{d}$ as solution satisfying $\gamma B_i^\top P_i B_i + R_i \succ 0 $
and $\gamma \| A - B_i K_i \|^2 < 1$ where $K_i = \gamma (\gamma B_i^\top P_i B_i + R_i)^{-1} B_i^\top P_i A$. Then, by applying the Dynamically Programming principle and the expression of optimal value function \cite{kuvcera1972discrete} starting at time $t=1$ with an initial $V_{i,1} = B_i (\beta_{i,0} - \bar \beta_{i,0} )$, the value $C_{V_i}( \bbeta_i - \bar \bbeta_i)$ can be expressed as:
\begin{align*}
    \min_{\bbeta_i} C_{V_i}( \bbeta_i - \bar \bbeta_i) 
    &= \min_{\beta_{i,0} = \check \beta_i} \EE\left[ \min_{\bbeta_i: \beta_{i,0} = \check\beta_{i} } C_{V_i}(\bbeta_i - \bar \bbeta_i) \right] 
    \\
    &=
    \min_{\check\beta_{i}}  \EE\left[  (\check\beta_{i} - \bar{\check\beta}_{i})^\top \Big( R_i + \gamma B_i^\top P_i B_i \Big) (\check\beta_{i} - \bar{\check\beta}_{i} ) \right] \geq  0.
\end{align*}
Existence of such a solution $P_i$ can be guaranteed under suitable conditions, see  \cite{ran1988existence} or Lemma~\ref{lemma:existence_DARE_Thm_3.1} below. 
Moreover, for every given random variable $\check\beta_{i}$, the value $\min_{\bbeta_i: \beta_{i,0} = \check\beta_{i} } C_{V_i}(\bbeta_i - \bar \bbeta_i)$ is attained at $\beta_{i, t} - \bar \beta_{i,t} = - K_i (V_{i,t} - \bar{V}_{i,t})$ for $t \geq 1$.

Similarly, if the discrete Algebraic Riccati equation for $i=1, 2$ (DARE-MF-i):
\begin{align}
    \label{eq:DARE_MF}
	0 & = (-1)^{i-1} (Q + \bar Q) - \bar P_i + \gamma (A + \bar A)^\top \bar P_i (A + \bar A) 
	\\
	\nonumber
	& \qquad - \gamma^2 (A + \bar A)^\top \bar P_i (B_i + \bar B_i) \Big( \gamma (B_i+ \bar B_i)^\top \bar P_i (B_i + \bar B_i) 
	\\
	&\qquad +(R_i + \bar R_i) \Big)^{-1} (B_i+ \bar B_i)^\top \bar P_i (A + \bar A).
	\nonumber
\end{align}
has a solution $\bar P_i$ such that $\gamma (B_i + \bar B_i)^\top \bar P_i (B_i + \bar B_i) + (R_i + \bar R_i) \succ 0 $ and $\gamma \| A + \bar A - (B_i + \bar B_i) L_i \|^2 < 1$ where $L_i = \gamma \Big(\gamma (B_i + \bar B_i)^\top \bar P_i (B_i + \bar B_i) + (R_i + \bar R_i) \Big)^{-1} (B_i + \bar B_i)^\top \bar P_i (A + \bar A)$, then the value function $\bar C_{V_i}( \bar \bbeta_i)$ can be expressed as 
\begin{equation*}
    \min_{\bar \bbeta_i} \bar C_{V_i}( \bar \bbeta_i) 
    = \min_{\check{\bar \beta}_{i}}  \EE\left[  (\check{\bar \beta}_{i})^\top \Big( (R_i + \bar R_i) + \gamma (B_i + \bar B_i)^\top \bar P_i (B_i + \bar B_i) \Big) \check{\bar \beta}_{i}  \right] \geq 0.
\end{equation*}
Furthermore, for every given random variable $\check{\bar \beta}_{i}$, the value $\min_{\bar \bbeta_i : \bar \beta_{i,0} = \check{\bar \beta}_i} \bar C_{V_i}(\bar \bbeta_i)$ is attained at $\bar \beta_{i,t} = - L_i \bar{V}_{i,t}$ for $t \geq 1$.\\

We have directly the following sufficient condition for the convexity-concavity condition.
\begin{lemma}
\label{lemma:sufficient_condition_convexity_concavity}
    If the four discrete Algebraic Riccati equations (DARE-i) and (DARE-MF-i) for $i=1,2$, i.e., \eqref{eq:DARE} and \eqref{eq:DARE_MF}, have solutions $(P_i, \bar P_i)$ such that 
    \begin{equation}
        \label{eq:condition_DARE_P_i}
        \gamma B_i^\top P_i B_i + R_i \succ 0, \qquad \gamma (B_i + \bar B_i)^\top \bar P_i (B_i + \bar B_i) + (R_i + \bar R_i) \succ 0,
    \end{equation}
     then the convexity-concavity condition~\eqref{eq:cond_convex_1}--\eqref{eq:cond_concave_2} for the process $\bV_1$ and $\bV_2$ holds.
\end{lemma}

Together with the sufficient condition of the Pontryagin maximum principle (Proposition~\ref{prop:sufficient_Pontryagin}), we have 
\begin{corollary}
\label{corollary:sufficient_open_loop_Nash_eq}
 Let $(P_1, P_2, \bar P_1, \bar P_2)$ be solutions to the four discrete Algebraic Riccati equations in Lemma~\ref{lemma:sufficient_condition_convexity_concavity}, and they satisfy conditions \eqref{eq:condition_DARE_P_i}, then the pair of control processes $(\bu_1, \bu_2) \in \cU_{ad}^{open}$ defined in Proposition~\ref{prop:sufficient_Pontryagin} by~\eqref{eq:open_loop_Nash_expression_BSDE} is an open-loop Nash equilibrium for the zero-sum game.
\end{corollary}

To conclude this section, we cite a result from \cite{ran1988existence}, tailored to our setting, which provides a sufficient condition for the existence of solutions to discrete Algebraic Riccati equation. Based on this result, we propose a sufficient condition for the existence of $(P_1, P_2, \bar P_1, \bar P_2)$ in Corollary~\ref{corollary:sufficient_open_loop_Nash_eq}.
We say that $(A, B_i)$ is $\gamma-$stabilizable if there exists a matrix $K \in \RR^{\ell \times d}$ such that all eigenvalues of $\gamma^{1/2} (A - B_i K)$ in the complex plan lie inside the unit circle, i.e. $\gamma \| A - B_i K \|^2 < 1$. 

\begin{lemma}[Theorem 3.1 in \cite{ran1988existence}]
\label{lemma:existence_DARE_Thm_3.1}
Assume that $Q \in \cS^d$, $R_i \succ 0$, and $(A, B_i)$ is $\gamma-$stabilizable, for $i = 1,2$. For a matrix $\eta \in \cS^d$, let us denote by
\begin{equation}
\label{eq:DARE_i_cR_i}
\cR_i(\eta) := (-1)^{i-1} Q - \eta + \gamma A^\top \eta A - \gamma^2 A^\top \eta B_i \Big( \gamma B_i^\top \eta B_i + R_i \Big)^{-1} B_i^\top \eta A.
\end{equation}
Then the (DARE-i) has a symmetric solution $P_i \in \cS^d$ satisfying 
\begin{equation*}
    	\cR_i(P_i) = 0, \qquad \text{and} \qquad \gamma B_i^\top P_i B_i + R_i \succ 0
\end{equation*}
if and only if the set
\begin{equation}
\label{eq:DARE_LMI_i}
    \cD_i := \left\{ \eta \in \cS^d \  | \   \gamma B_i^\top \eta B_i + R_i \succ 0, \  \cR_i(\eta) \succeq 0 \right\}
\end{equation}
is not empty. Moreover, we have $P_i \succeq \eta$ for all $\eta \in \cD_i$, and $\gamma \| A - B_i K_i \|^2 < 1$ with $K_i = (\gamma B_i^\top P_i B_i + R_i)^{-1} (\gamma B_i^\top P_i A)$.
\end{lemma}
Similarly, for the existence of $\bar{P}_i$ to (DARE-MF-i) satisfying condition \eqref{eq:condition_DARE_P_i}, we can define $\bar \cR_i(\eta)$ as the right hand side of \eqref{eq:DARE_MF} and consider the set 
$$
    \bar{\cD}_i := \left\{ \bar \eta \in \cS^d \  | \   \gamma (B_i + \bar B_i)^\top \bar \eta (B_i + \bar B_i) + (R_i + \bar R_i) \succ 0, \  \bar \cR_i(\bar \eta) \succeq 0 \right\}.
$$
We notice that if $\gamma \| A \|^2 < 1$, the pair $(A, B_i)$ will be $\gamma-$stabilizable, for $i=1,2$. So, under Assumption \ref{assumption:A_openloop}, we have $(A, B_1), (A, B_2), (\tilde{A}, \tilde{B}_1)$ and $(\tilde A, \tilde B_2)$ are all $\gamma-$stabilizable. 

\begin{corollary}
\label{corollary:existence_open_loop_Nash}
    Under Assumption~\ref{assumption:A_openloop}, and we assume that $Q, \bar Q \in \cS^d$, $R_i, R_i + \bar R_i \succ 0$ for $i=1,2$. If $\cD_1, \cD_2, \bar \cD_1, \bar \cD_2 \neq \emptyset$, and if the forward-backward system of equations \eqref{eq:FBSDE_pontryagin_sufficient} holds for $\bx$ and $\bp$, then there exists an open-loop Nash equilibrium for the zero-sum game.
\end{corollary}

%% file: sec_closed-loop.tex
\section{Closed-loop information structure}
\label{sec:closed-loop-structure}

In this section, we turn our attention to closed-loop controls, that is, controls which are functions of the state and the conditional mean. We will in fact focus on a specific class of such functions.

	\subsection{Admissible set of controls}

We start by defining the class of functions we will consider in the rest of this section. 

	 \begin{definition}
	    For $i=1,2$, a closed-loop feedback strategy (or policy) for player $i$ is a function $v_{\theta_i}: \RR^d \times \RR^d \to \RR^\ell$, $(x, \bar x) \mapsto (-1)^i K_i (x - \bar x) + (-1)^i L_i \bar x$ parameterized by a tuple $\theta_i = (K_i, L_i)$ where $K_i$ and $L_i$ are (deterministic) matrices in $\RR^{\ell \times d}$. 
	    A pair of policies given by parameters $(\theta_1, \theta_2) \in (\RR^{\ell \times d})^2 \times (\RR^{\ell \times d})^2$ for the two players is called a closed-loop feedback policy profile.
	    
	 \end{definition}
	 For simplicity, in the sequel, we will  use interchangeably the terms strategy, policy and parameters. In other words, for any $\theta$, we identify the parameters $\theta$ with the induced closed-loop policy $v_{\theta}$.

    We consider the set of admissible policy in the closed-loop information structure as follow:
	\begin{equation}
	\Theta_{ad}^{close} = \Bigg\{  (\theta_1 , \theta_2) \in (\RR^{\ell \times d})^2 \times (\RR^{\ell \times d})^2 \  | \  \bx^{\theta_1, \theta_2} \in \cX  
	\Bigg\},
	\end{equation}
	where the state process $\bx^{\theta_1, \theta_2} = (x_t^{\theta_1, \theta_2})_{t \geq 0} $ is controlled by the pair of closed-loop feedback control processes $(\bu_1, \bu_2) \in \cU_{loc} \times \cU_{loc}$ defined by
	\begin{equation}
	\label{eq:CLFB_control}
	u_{i,t} = (-1)^i K_i (x_t^{\theta_1, \theta_2} - \bar{x}_t^{\theta_1, \theta_2} )  + (-1)^i L_i \bar{x}_t^{\theta_1, \theta_2}.
	\end{equation}

	When we plug in the above closed-loop feedback controls \eqref{eq:CLFB_control} into the state process dynamics \eqref{eq:MKV-state_ZS}, we obtain that, for every $t \geq 0$,
	\begin{equation*}
	    x_{t+1}^{\theta_1, \theta_2} = (A - B_1 K_1 + B_2 K_2) \left(x_t^{\theta_1, \theta_2} - \bar{x}_t^{\theta_1, \theta_2} \right) + \left( \tilde{A} - \tilde{B}_1 L_1 + \tilde{B}_2 L_2 \right) \bar{x}_t^{\theta_1, \theta_2} + \epsilon_{t+1}^0 + \epsilon_{t+1}^1,
 	\end{equation*}
 	where $\tilde{A} = A + \bar A$, $\tilde{B}_1 = B_1 + \bar B_1$, and $\tilde B_2 = B_2 + \bar B_2$.
 	By Proposition~\ref{prop:stability_with_extra_L2_term}, the process $\bx^{\theta_1, \theta_2}$ is $L^2-$discounted globally integrable under the assumption:
 	$$
 	    \gamma \| A - B_1 K_1 + B_2 K_2 \|^2 < 1, \qquad \text{and} \qquad \gamma \| \tilde A - \tilde B_1 L_1 + \tilde B_2 L_2 \|^2 < 1.
 	$$
 	Thus, it is reasonable to consider the following subset of $\Theta_{ad}^{close}$: 
    \begin{align}
    \label{eq:Theta_stable_set}
    	\Theta &= \left\{ (\theta_1, \theta_2) \in  (\RR^{\ell \times d})^2 \times (\RR^{\ell \times d})^2 \  | \  \gamma \| A - B_1 K_1 + B_2 K_2 \|^2 < 1, \right.
    	\\
    	& \qquad\qquad\qquad\qquad \left.\   \gamma \| \tilde A - \tilde B_1 L_1 + \tilde B_2 L_2 \|^2 < 1 \right\}.
    	\notag
   \end{align}
 	
	Despite its simple definition, one can check that the set $\Theta$ is not convex (see e.g. the Appendix of \cite{fazel2018global}). Moreover, the set $\Theta_{ad}^{close}$ does not have a simple expression in terms of the model parameters. Without any additional assumptions, the two players need to decide together the set of admissible policy profiles $\Theta_{ad}^{close}$ before playing against each other in a zero-sum game.
	However, in some situations, we can consider a subset of $\Theta_{ad}^{close}$ of the form $\Theta_1 \times \Theta_2$ where $\Theta_1, \Theta_2$ are two independent closed subsets in $\cU_{ad}^{close}$, so that a player is able to choose freely and independently her admissible strategy without being affected by the $L^2$-integrability issue of the state process caused by the choice of strategy of her opponent.\\
 	
 	Under Assumption~\ref{assumption:A_openloop}, namely $\gamma \| A \|^2 < 1$ and $\gamma \| \tilde{A} \|^2 < 1$, there exists two pairs of real numbers $(\eta_1, \eta_2) \in \RR^2$ and $(\tilde \eta_1, \tilde \eta_2) \in \RR^2$ such that
        \begin{equation*}
            \kappa :=  \gamma \| A \|^2 + \gamma \left( \eta_1^2 \| B_1 \|^2 + \eta_2^2 \| B_2 \|^2\right)   < 1, 
            \tilde \kappa :=  \gamma \| \tilde A \|^2 + \gamma \left( \tilde \eta_1^2 \| \tilde B_1 \|^2 + \tilde \eta_2^2 \| \tilde B_2 \|^2  \right) < 1.
        \end{equation*}
    For $i=1,2$, let us denote 
    $
        r_K^{(i)} = \eta_i \sqrt{ \frac{1}{2} \left( \frac{1}{ \kappa } -1 \right)},$ and $r_L^{(i)} = \tilde \eta_i \sqrt{ \frac{1}{2} \left( \frac{1}{ \tilde \kappa } -1 \right)} .$
    
    The following lemma, obtained by Cauchy-Schwarz inequality, provides an example in which the two players are able to choose their admissible strategies independently of each other.
 	
 	\begin{lemma}
        Assuming the closed-loop feedback policies $\theta_1 = (K_1, L_1) \in \RR^{\ell \times d} \times \RR^{\ell \times d}$ and $\theta_2 = (K_2, L_2) \in \RR^{\ell \times d} \times \RR^{\ell \times d}$ satisfy  $\| K_i \| \leq r_K^{(i)}$ and  $\| L_i \| \leq r_L^{(i)}$ 
        for $i = 1, 2$, then $(\theta_1, \theta_2) \in \Theta$. 
    \end{lemma}
 	\commentJDG{
 	\begin{proof}
        The Cauchy-Schwarz inequality implies that
     	\begin{equation*}
     	    \gamma \| A - B_1 K_1 + B_2 K_2 \|^2  \leq \left( \gamma \| A \|^2 + \gamma \eta_1^2 \| B_1 \|^2 + \gamma \eta_2^2 \| B_2 \|^2 \right) \left( 1 + \eta_1^{-2} \| K_1 \|^2 + \eta_2^{-2} \| K_2 \|^2 \right) < 1.
     	\end{equation*}
     	Similarly, we have $\gamma \| \tilde A - \tilde B_1 L_1 + \tilde B_2 L_2 \|^2 < 1$. The result follows.
 	\end{proof}
 	}
 	If the context is clear, we omit in the following sections the superscript $(\theta_1, \theta_2)$ in state processes $(x_t^{\theta_1, \theta_2} )_{t\geq 0}$.

    \subsection{Auxiliary processes}
	
	We apply a re-parametrization on the state variable $x_t$ as follow: for every $t \geq 0$, let 
	\begin{equation*}
		y_t = x_t - \bar{x}_t,  \qquad  z_t = \bar{x}_t.
	\end{equation*}
    We denote $\by = (y_t)_{t \geq 0}$ and $\bz = (z_t)_{t \geq 0}$ the two auxiliary state processes derived from $\bx$.
	For the sake of clarity, we introduce some new notations on the control processes using the sample re-parametrization method:
	\begin{align*}
		& u_{1,t}^{(y)} := u_{1,t} - \bar{u}_{1,t} = -K_1 y_t, && u_{2,t}^{(y)} := u_{2,t} - \bar{u}_{2,t} = K_2 y_t
		\\
		& 
		u_{1,t}^{(z)} := \bar u_{1,t} = - L_1 z_t, && u_{2,t}^{(z)} := \bar{u}_{2,t} = L_2 z_t,
	\end{align*}  
	The processes $(y_t)_{t\geq 0}$ and $(z_t)_{t\geq 0}$ defined in this way follow the dynamics	
	\begin{align}	
	    \label{eq:dyn_y_closed}
		y_{t+1} & = A y_{t} + B_1 \yut{1} + B_2 \yut{2} + \epsilon_{t+1}^1 , \qquad y_0 \sim \epsilon_0^1,
		\\
		\label{eq:dyn_z_closed}
		z_{t+1} & = \tilde{A} z_t + \tilde{B}_1 \zut{1} + \tilde{B}_2 \zut{2} + \epsilon_{t+1}^0, \qquad z_0 \sim \epsilon_0^0,
	\end{align}
	where $\epsilon_0^0, \epsilon_0^1$ are random variables with distributions $\mu_0^0$ and $\mu_0^1$ respectively. We then observe that for every $ t, t' \geq 0$, the random variables $y_t$ and $z_t$ are respectively $\cF^1-$measurable and  $\cF^0-$measurable, and they are independent.
	
	The running cost at time $t$ defined by equation \eqref{eq:MKV_running_cost_c} can also expressed using the above notations:
	\begin{align*}
		& c(x_t, \bar{x}_t, u_{1,t}, \bar{u}_{1,t}, u_{2,t}, \bar{u}_{2,t} ) 
		\nonumber \\ 
		= & y_t^\top Q y_t + (\yut{1})^\top R_1 \yut{1} - (\yut{2})^\top R_2 \yut{2} + z_t^\top \tilde Q z_t + (\zut{1})^\top \tilde R_1 \zut{1} - (\zut{2})^\top \tilde R_2 \zut{2}
		\nonumber \\
		= & c_y( y_t, \yut{1}, \yut{2} ) + c_z( z_t, \zut{1}, \zut{2} ),
	\end{align*}
	where $\tilde Q = Q + \bar Q$, $\tilde R_1 = R_1 + \bar R_1$, $\tilde R_2 = R_2 + \bar R_2$, and $c_y : \RR^d \times \RR^{\ell} \times \RR^{\ell} \to \RR $ and $c_z : \RR^{d} \times \RR^{\ell} \times \RR^{\ell} \to \RR$ are the running cost functions associated to $(y_t)_{t\geq 0}$ and $(z_t)_{t\geq 0}$ defined by
	\begin{align*}
		c_y(y_t, \yut{1}, \yut{2}) 
		&: = y_t^\top Q y_t + ( \yut{1} )^\top R_1 \yut{1} - ( \yut{2} )^\top R_2 \yut{2} 
		\commentJDG{
		= 
		\left[ \begin{array}{c} y_t \\ \yut{1} \\ \yut{2} \end{array} \right]^\top 
		\left[ 
		\begin{array}{ccc} Q & 0 & 0 \\
					  0 & R_1 & 0 \\
					  0 & 0 & -R_2 \\
		\end{array}
		\right]
		\left[ \begin{array}{c} y_t \\ \yut{1} \\ \yut{2} \end{array} \right]
		}
	\end{align*}
	and 
	\begin{align*}
		c_z(z_t, \zut{1}, \zut{2}) 
		&: = 
		z_t^\top \tilde{Q} z_t + ( \zut{1} )^\top \tilde{R}_1 \zut{1} - ( \zut{2} )^\top \tilde{R}_2 \zut{2} 
		 \commentJDG{= 
		 \left[ \begin{array}{c} z_t \\ \zut{1} \\ \zut{2} \end{array} \right]^\top 
		\left[ 
		\begin{array}{ccc} \tilde Q & 0 & 0 \\
		0 & \tilde R_1 & 0 \\
		0 & 0 & - \tilde R_2 \\
		\end{array}
		\right]
		\left[ \begin{array}{c} z_t \\ \zut{1} \\ \zut{2} \end{array} \right]
		}.
	\end{align*}

	We denote by $C(\theta_1, \theta_2) = J(\bu_1, \bu_2)$ the utility function associated to a closed-loop feedback policy profile  $(\theta_1, \theta_2) \in \Theta$. As what has been presented in the proof of Proposition~\ref{lem:PG_expression}, we introduce two auxiliary utility functions $C_y(K_1, K_2, \tilde{y})$ and $C_z(L_1, L_2, \tilde{z})$ defined as 
	\begin{subequations}
	    \begin{empheq}[left = \empheqlbrace]{align}
	    		C_y( K_1, K_2, \tilde{y} ) & = \EE \left[ \sum_{t=0}^\infty \gamma^t c_y( y_t, \yut{1}, \yut{2} ) \ \Big| y_0 = \tilde{y} \right],
		    \label{eq:C_y_K1_K2} 
		\\
	    	C_z( L_1, L_2, \tilde{z} ) & = \EE \left[ \sum_{t=0}^\infty \gamma^t c_z( z_t, \zut{1}, \zut{2} ) \ \Big| z_0 = \tilde{z} \right],
			\label{eq:C_z_L1_L2}
	    \end{empheq}
	\end{subequations}
	in which the control processes are  $(u_{1,t}^{(y)}, u_{1,t}^{(z)}) = (- K_1 y_t, - L_1 z_t)$ and $(u_{2,t}^{(y)}, u_{2,t}^{(z)}) = (K_2 y_t, L_2 z_t)$. 
	Lemma~\ref{lemma:L2_x_and_x_bar} shows that $\bx_t^{\theta_1, \theta_2}$ is $L^2-$integrable if and only if $\by$ and $\bz$ are $L^2-$integrable. With the new notations, we let
	\begin{equation}
	\label{eq:def-Ctheta1theta2}
		C(\theta_1, \theta_2) = \EE_{\tilde{y}}[ C_y(K_1, K_2, \tilde{y} ) ] + \EE_{\tilde{z}} [ C_z( L_1, L_2, \tilde{z}) ].
	\end{equation}

	Now we can define the closed-loop saddle point, also known as the Nash equilibrium in the closed-loop information structure, for the zero-sum game.
 	\begin{definition}
		A closed-loop feedback policy profile $(\theta_1^*, \theta_2^*) \in \Theta_{ad}^{close}$ with $\theta_1^* = (K_1^*, L_1^*)$ and $\theta_2^* = (K_2^*, L_2^*)$ is said to be a closed-loop saddle point for the zero-sum game (CLSP for short) if and only if
		\begin{itemize}
			\item for every $\theta_1 = (K_1, L_1) \in \RR^{\ell\times d} \times \RR^{\ell \times d}$ such that $(\theta_1, \theta_2^*) \in \Theta_{ad}^{close}$, we have 
			\begin{equation*}
			C(\theta_1, \theta_2^*) \geq C(\theta_1^*, \theta_2^*),
			\end{equation*}
			\item and for every $\theta_2 = (K_2, L_2) \in \RR^{\ell \times d} \times \RR^{\ell \times d}$ such that $(\theta_1^*, \theta_2) \in \Theta_{ad}^{close}$, we have 
			\begin{equation*}
			C(\theta_1^*, \theta_2) \leq C(\theta_1^*, \theta_2^*).
			\end{equation*}
			
		\end{itemize}
		
	\end{definition}
 	
	\begin{remark}
		We notice that the state processes associated with $C( \theta_1, \theta_2^*)$, $C(\theta_1^*, \theta_2^*)$, and $C(\theta_1^*, \theta_2)$ are all different. For example, the state process $\bx^{\theta_1, \theta_2^*}$ appearing in $C(\theta_1, \theta^*_2)$  
		follows the dynamics
		\begin{align*}
		    x_{t+1}^{\theta_1, \theta_2^*} & = A x_t^{\theta_1, \theta_2^*} + B_1 u_{1,t} + \bar B_1 \bar u_{1,t}  + B_2 u_{2,t}^* + \bar B_2 \bar u_{2,t}^*  + \epsilon_{t+1}^1 + \epsilon_{t+1}^0
		    \\
		    & = \left( A - B_1 K_1 + B_2 K_2^*  \right) y_t^{\theta_1, \theta_2^*} + \left( \tilde A - \tilde B_1 L_1 + \tilde B_2 L_2^* \right) z_t^{\theta_1, \theta_2^*}+ \epsilon_{t+1}^0 + \epsilon_{t+1}^1
		\end{align*}
		with $x_0^{\theta_1, \theta_2^*} \sim \epsilon_0^0 + \epsilon_0^1$, whereas the state process $\bx^{\theta_1^*, \theta_2^*}$ is given by
		\begin{equation*}
		     x_{t+1}^{\theta_1^*, \theta_2^*} = \left( A - B_1 K_1^* + B_2 K_2^*  \right) y_t^{\theta_1^*, \theta_2^*} + \left( \tilde A - \tilde B_1 L_1^* + \tilde B_2 L_2^* \right) z_t^{\theta_1^*, \theta_2^*}+ \epsilon_{t+1}^0 + \epsilon_{t+1}^1.
		\end{equation*}
		
	\end{remark}

	\begin{remark}
		We can see that the process $(y_t)_{t\geq 0}$ is completely controlled by $(K_1, K_2)$ or by $(\bu_1^{(y)}, \bu_2^{(y)})$, and likewise for the process $(z_t)_{t\geq 0}$ by $(L_1, L_2)$ or by $(\bu_1^{(z)}, \bu_2^{(z)})$. Moreover, the noise processes associated with $(y_t)_{t\geq 0}$ and $(z_t)_{t\geq 0}$ are independent. So when the two players are at CLSP $(\theta_1^*, \theta_2^*)$, and one of them, say controller 1, perturbs her policy with a parameter set $\theta_1 = (K_1, L_1)$ different from $\theta_1^*=(K_1^*, L_1^*)$, we can look at the difference between the costs $\EE_{\tilde{y}}[ C_y(K_1, K_2^*, \tilde{y})]$ and $\EE_{\tilde y}[ C_y(K_1^*, K_2^*, \tilde{y})]$, and separately at the difference between $\EE_{\tilde{z}}[C_z(L_1, L_2^*, \tilde{z})]$ and $\EE_{\tilde{z}}[C_z(L_1^*, L_2^*, \tilde{z})]$.
	\end{remark}
	
	We introduce here two sets related to the admissible policies with respect to the processes $(y_t)_{t \geq 0}$ and $(z_t)_{t \geq 0}$.
	Let us denote by 
	\begin{subequations}
	\begin{empheq}[left=\empheqlbrace]{align}
	        \Theta_y &= \big\{ (K_1, K_2) \in \RR^{\ell \times d} \times \RR^{\ell \times d} \ | \  \by \in \cX \big\},
	        \label{eq:Theta_y}
	    \\
	    \Theta_z &= \big\{ (L_1, L_2) \in \RR^{\ell \times d} \times \RR^{\ell \times d} \ | \  \bz \in \cX \big\},
	    \label{eq:Theta_z}
	   \end{empheq}
	\end{subequations}
	where $\by$ and $\bz$ are two processes following the dynamics \eqref{eq:dyn_y_closed} and \eqref{eq:dyn_z_closed}. The processes $\by$ and $\bz$ can be constructed without any prior knowledge from $\bx$, and they are completely determined by the choice of matrix pairs $(K_1, K_2)$ and $(L_1, L_2)$. From Lemma~\ref{lemma:L2_x_and_x_bar}, the set $\Theta_y$ (and similarly $\Theta_z$) can be understood as the collection of pair of matrices consisting of the first (or second) elements in policies $\theta_1$ and $\theta_2$.
	
	\begin{definition}
	    A pair of matrices $(K_1^*, K_2^*) \in \RR^{\ell \times d} \times \RR^{\ell \times d}$ is said to be a closed-loop feedback saddle point in $\Theta_y$ ($CLSP-y$ for short) if for every $\tilde y \in \RR^d$, for every $K_1, K_2 \in \RR^{\ell \times d}$ such that $(K_1, K_2^*) \in \Theta_y$ and $(K_1^*, K_2) \in \Theta_y$, we have 
	    \begin{equation}
		    \label{eq:cond_CLSP_y}
			C_y(K_1^*, K_2, \tilde{y}) \leq C_y(K_1^*, K_2^*, \tilde{y}) \leq C_y(K_1, K_2^*, \tilde{y}).
		\end{equation}
		A pair of matrices $(L_1^*, L_2^*) \in \RR^{\ell \times d} \times \RR^{\ell \times d}$ is said to be a closed-loop feedback saddle point in $\Theta_z$ ($CLSP-z$ for short) if for every $\tilde z \in \RR^d$, for every $L_1, L_2 \in \RR^{\ell \times d}$ such that $(L_1, L_2^*) \in \Theta_z$ and $(L_1^*, L_2) \in \Theta_z$, we have 
	    \begin{equation}
		    \label{eq:cond_CLSP_z}
			C_z(L_1^*, L_2, \tilde{z}) \leq C_z(L_1^*, L_2^*, \tilde{z}) \leq C_z(L_1, L_2^*, \tilde{z}).
		\end{equation}
	\end{definition}

	\subsection{Notations and useful lemmas}
    
    In the sequel, we will use the following notations:
    \begin{subequations}
	\label{eq:notation_M_N_L}
	    \begin{empheq}[left=\empheqlbrace]{align}
	& \cM(P) = \gamma A^\top P A - P + Q
		\\
		& \cL_1(P) = \gamma A^\top P B_1, \qquad \cL_2(P) = \gamma A^\top P B_2, \qquad \cL_{12} = \gamma B_1^\top P B_2
		\\
		& \cN_1(P) = \gamma B_1^\top P B_1 + R_1, \qquad \cN_2(P) = \gamma B_2^\top P B_2 - R_2.
		\end{empheq}
	\end{subequations}
	
	The following lemma is about the inverse of block matrix.
	\begin{lemma}[Inverse $2 \times 2$ block matrix; Corollary~4.1 in \cite{lu2002inverses}]
	It holds:
		\label{lemma:inverse_block}
		\begin{itemize}
			\item 
			If $\cN_1(P)$ and $S_2 := \cN_2(P) - \cL_{12}(P)^\top \cN_1(P)^{-1} \cL_{12}(P) $ are invertible, then
			\begin{equation*}
			\cN(P)^{-1} = \left[
			\begin{array}{cc}
			\cN_1(P)^{-1} + \cN_1(P)^{-1} \cL_{12}(P) S_2^{-1} \cL_{12}(P)^\top \cN_1(P)^{-1} & -\cN_1(P)^{-1} \cL_{12}(P) S_2^{-1} 
			\\
			\\
			- (\cN_1(P)^{-1} \cL_{12}(P) S_2^{-1} )^\top &  S_2^{-1}
			\end{array}
			\right].
			\end{equation*}
			
			\item 
			If $\cN_2(P)$ and $S_1 := \cN_1(P) - \cL_{12}(P)  \cN_2(P)^{-1} \cL_{12}(P)^\top$ are invertible, then
			\begin{equation*}
			\cN(P)^{-1} = \left[
			\begin{array}{cc}
			S_1^{-1} &  - S_1^{-1} \cL_{12}(P) \cN_2(P) 
			\\
			\\
			-  ( S_1^{-1} \cL_{12}(P) \cN_2(P) )^{-1} & \cN_2(P) + \cN_2(P)^{-1} \cL_{12}(P)^\top S_2^{-1} \cL_{12}(P) \cN_2(P)^{-1}
			\end{array}
			\right]. 
			\end{equation*}
			
		\end{itemize}
		
	\end{lemma}

	\begin{lemma}[Schur's lemma]
	\label{lemma:schur's_lemma}
		For every symmetric matrices $M \in \cS^d$ and $N \in \cS^{\ell}$, and for every matrix $L \in \RR^{d \times \ell}$, if $N$ is invertible, then
		\begin{equation*}
		\left[ \begin{array}{c} x \\ u \end{array} \right]^\top 
		\left[ \begin{array}{cc} M & L \\ L^\top & N \end{array} \right]
		\left[ \begin{array}{c} x \\ u \end{array} \right]  
		= x^\top \left( M - L N^{-1} L^\top \right) x +  \left(u + N^{-1} L^\top x \right)^\top N \left( u + N^{-1} L^\top x \right).
		\end{equation*}\\
	\end{lemma}

	
	\subsection{Algebraic Riccati equations}

    We present here a few lemmas and some notations that will be useful to understand the closed-loop saddle point in $\Theta_y$ ($CLSP-y$). Since the processes $\by$ and $\bz$ follow similar linear dynamics but with different coefficients, we omit the proof for lemmas corresponding to $CLSP-z$. We use the notation $\langle u, v \rangle$ to represent the product $u^\top v$ for two vectors in $\RR^d$. 
    Using the dynamics \eqref{eq:dyn_y_closed} for $(y_t)_{t\geq 0}$, we obtain the following result.

	\begin{lemma}
	\label{lemma:relation_yPy_tp1_t}
	For every symmetric matrix $P \in \cS^d$, we have
	\begin{align*}
		&\gamma^{t+1} \EE \left[  \langle P y_{t+1}, y_{t+1} \rangle \right] 
		\\
		= &  \gamma^t \EE \left[ \langle P y_t, y_t \rangle \right]  + \gamma^{t+1} \EE \left[ (\epsilon_{t+1}^1)^\top P \epsilon_{t+1}^1 \right]
		\nonumber \\
		& + \gamma^t \EE \left[ \left[ \begin{array}{c} y_t \\ \yut{1} \\ \yut{2} \end{array} \right]^\top 
		\left[ 
		\begin{array}{lll} \cM(P) - Q & \cL_1(P) & \cL_2(P) \\
		\cL_1(P)^\top & \cN_1(P) - R_1 & \cL_{12}(P) \\
		\cL_2(P)^\top & \cL_{12}(P)^\top &  \cN_2(P) + R_2 \\
		\end{array}
		\right]
		\left[ \begin{array}{c} y_t \\ \yut{1} \\ \yut{2} \end{array} \right]
		\right]
	\end{align*}
	where we recall the notation~\eqref{eq:notation_M_N_L}.
	
	\end{lemma}
	
	\commentJDG{
	\begin{proof}
	    We use the dynamics \eqref{eq:dyn_y_closed} for $(y_t)_{t\geq 0}$ and we obtain
	    \begin{align*}
		 & \gamma^{t+1} \EE \left[  \langle P y_{t+1}, y_{t+1} \rangle \right] 
	     \\
		 = &    \gamma^t \EE \left[ \langle P y_t, y_t \rangle \right]  + \gamma^{t+1} \EE \left[ \langle P \epsilon_{t+1}^1, \epsilon_{t+1}^1 \rangle \right]	 
		 \\
		 & +  \gamma^t  \EE   \bigg[ \langle (\cM(P) - Q) y_t, y_t \rangle + 2 \langle \cL_1(P)^\top y_t, \yut{1} \rangle  + 2 \langle \cL_2(P)^\top y_t, \yut{2} \rangle 
		\\
		&  + \langle (\cN_1(P) + R_1) \yut{1}, \yut{1} \rangle + \langle (\cN_2(P) - R_2) \yut{2}, \yut{2} \rangle + 2 \langle \cL_{12}(P)^\top \yut{1}, \yut{2} \rangle  \bigg] .
	    \end{align*}
	\end{proof}
    }

	Let us denote 
	\begin{equation*}
		C_y^*(P; \tilde{y}) :=  \tilde{y}^\top P \tilde{y}  + \frac{\gamma}{1 - \gamma} \EE \left[ (\epsilon_1^1)^\top P \epsilon_1^1 \right].
	\end{equation*}

	\begin{corollary}
		\label{corollary:opt_gap_Cy}
		For every $P \in \cS^{d}$ and every $(K_1, K_2) \in \Theta_y$, we have 
		\begin{align}
		\label{eq:opt_gap_Cy}
		&C_y(K_1, K_2, \tilde{y} ) - C_y^*(P; \tilde{y}) 
		\\
		&=  \EE \left[  \left. \sum_{t=0}^\infty
		\gamma^t   \left[ \begin{array}{c} y_t \\ \yut{1} \\ \yut{2} \end{array} \right]^\top 
		\left[ 
		\begin{array}{lll} \cM(P) & \cL_1(P) & \cL_2(P) \\
		\cL_1(P)^\top & \cN_1(P) & \cL_{12}(P) \\
		\cL_1(P)^\top & \cL_{12}(P)^\top &  \cN_2(P)  \\
		\end{array}
		\right]
		\left[ \begin{array}{c} y_t \\ \yut{1} \\ \yut{2} \end{array} \right] \  \right| y_0 = \tilde{y}  \right].
		\notag
		\end{align}
	\end{corollary}
	
	\begin{remark}
		We notice that the difference between $C_y(K_1, K_2, \tilde{y})$ and $C_y^*(P; \tilde{y})$ depends on the cross product $\langle \cL_{12}(P)^\top \yut{1}, \yut{2} \rangle$. When we perturb only one policy parameter, say $K_1$ for example, the change involved in the cost $C_y(K_1, K_2, \tilde{y})$ is not only caused by the state process $(y_t)_{t \geq 0}$ but also by the interactions between the two feedback control processes, even if no term in definition \eqref{eq:C_y_K1_K2} of $C_y(K_1, K_2, \tilde{y})$ is directly related to this cross interaction between strategies. The cross product $\cL_{12}(P)$ in equation \eqref{eq:opt_gap_Cy} makes our main result of this section, namely Proposition~\ref{prop:sufficient_CLSP_y} for the CLSP, harder to prove, compared with a continuous-time result as discussed e.g. in \cite{sun2016linear}.
	\end{remark}
	
	\begin{proof}
	    The process $\by = (y_t)_{t\geq 0}$ following dynamics \eqref{eq:dyn_y_closed} with matrices $(K_1, K_2) \in \Theta_y$ satisfies
		\begin{equation*}
			y_{t+1} = (A - B_1 K_1 + B_2 K_2) y_t + \epsilon_{t+1}^1, \qquad y_0 = \tilde{y}.
		\end{equation*}
		By definition of $\Theta_y$, see~\eqref{eq:Theta_y}, $\by$ is $L^2-$discounted globally integrable. By Lemma~\ref{lemma:x_L2_integ_L2_asym_stable}, we have $ \lim_{t \to 0} \EE[ \gamma^t \| y_t \|^2 ] = 0$.
		Thus, by Cauchy-Schwartz inequality and Jensen's inequality, we obtain
		\begin{equation*}
			\lim_{t \to \infty} \left| \EE \left[  \gamma^t \langle P y_t, y_t \rangle | y_0 = \tilde{y} \right] \right| \leq \lim_{t \to \infty} \| P \| \EE\left[ \gamma^t \| y_t \|^2 | y_0 = \tilde{y} \right] = 0.
		\end{equation*}
	    By applying recursively Lemma~\ref{lemma:relation_yPy_tp1_t} from $t=T$ down to $0$, and then letting $T$ tends to infinity, we obtain equation \eqref{eq:opt_gap_Cy}.\\
	\end{proof}
	
	We introduce here another Discrete Algebraic Riccati equation in the set of symmetric matrices $\cS^d$ for the discrete-time process $(y_t)_t$ and the cost $C_y$: 
	\begin{align*}
	\label{eq:ARE_y_complete}
    	P 
    	&= \gamma A^\top P A + Q 
    	\\
    	&\qquad - \gamma^2 [ A^\top P B_1, A^\top P B_2] 
    	\left[
    	\begin{array}{ll}
        	\gamma B_1^\top P B_1 + R_1 & \gamma B_1^\top P B_2 
        	\\
        	\gamma B_2^\top P B_1 & \gamma B_2^\top P B_2 - R_2 
        	\end{array} 
        	\right]^{-1} 
        	\left[
        	\begin{array}{l}
        	B_1^\top P A 
        	\\
        	B_2^\top P A
        	\end{array}
        	\right].
        	\notag
	\end{align*}
	The above equation can also be written in a compact form:
	\begin{equation}
	    \label{eq:ARE_y}
    	0 = \cM(P) - \cL(P) \cN(P)^{-1} \cL(P)^\top
	\end{equation}
	with (using the notations introduced in~\eqref{eq:notation_M_N_L})
	\begin{equation}
	    \label{eq:cL_P}
	    \cL(P) = [ \cL_1(P), \cL_2(P)] \in \RR^{d \times (\ell + \ell)}
	\end{equation}
	and a $2 \times 2$ block matrix
	\begin{equation}
    	\label{eq:cN_P}
    	\cN(P) = \left[
    	\begin{array}{ll}
    	\cN_1(P) & \cL{12}(P) \\
    	\cL_{12}(P)^\top &  \cN_2(P)  \\
    	\end{array}
    	\right] \in \RR^{(\ell + \ell) \times (\ell + \ell)} .
    \end{equation}
	To distinguish with other Algebraic Riccati equations introduced earlier sections, we may refer \eqref{eq:ARE_y} as (ARE-y).\\

	In the spirit of Nash equilibrium, we discuss in the following a few results related to the situations when only one controller intends to change her strategy but her opponent keeps the original one.
	
	Let us denote by $\by^{2*} = (y_t^{2*})_{t \geq 0}$ the state process associated to a pair of strategy $(K_1, K_2^*) \in \Theta_y$ where $K_2^*$ is a predetermined matrix in $\RR^{\ell \times d}$. Then the process $\by^{2*}$ follows the dynamics
	\begin{equation}
	\label{eq:dyn_y_K2*}
		y_{t+1}^{2*}= (A + B_2 K_2^*) y_{t}^{2*} + B_1 \yut{1} + \epsilon_{t+1}^1 = (A + B_2 K_2^* - B_1 K_1) y_t^{2*} + \epsilon_{t+1}^1, \, y_0^{2*} = \tilde{y},
	\end{equation}	
	where $(\yut{1})_{t \geq 0}$ is the control process adopted by player 1 (with parameter $K_1$).
	
	\begin{corollary}
	\label{corollary:C_y_K1_K2*}
	There holds
		\begin{align*}
			&C_y(K_1, K_2^*, \tilde{y}) - C^*_y(P; \tilde{y}) 
		    \\
		    &=  \EE \left[  \left. \sum_{t=0}^\infty
			\gamma^t   \left[ \begin{array}{c} y_t^{2*} \\ \yut{1} \end{array} \right]^\top 
			\left[ 
			\begin{array}{ll} \cM^{2*}(P) & \cL_1^{2*}(P) \\
			\cL_1^{2*}(P)^\top & \cN_1^{2*}(P)
			\end{array}
			\right]
			\left[ \begin{array}{c} y_t^{2*} \\ \yut{1} \end{array} \right] \  
			\right| y_0 = \tilde{y}  \right]
		\end{align*}		
		where
		\begin{equation*}
		\left\{
		    \begin{aligned}
			\cM^{2*}(P) & = \gamma ( A + B_2 K_2^*)^\top P (A + B_2 K_2^*) - P + Q - K_2^* R_2 K_2^*
			\\
			& = \cM(P) + \cL_2(P) K_2^* + (\cL_2(P) K_2^*)^\top + (K_2^*)^\top \cN_2(P) K_2^*
			\\
			\cL_1^{2*}(P) &= \gamma (A + B_2 K_2^*)^\top P B_1
						 = \cL_1(P) + (\cL_{12}(P) K_2^*)^\top
			\\
			\cN_1^{2*}(P) & = \gamma B_1^\top P B_1 + R_1 = \cN_1(P).
		\end{aligned}
		\right.
	    \end{equation*}
	    The ARE associated to the process $\by^{2*}$ and the value function $C_y(\cdot, K_2^*,\tilde y)$ is given by 
	    \begin{equation}
		    \label{eq:ARE_y_K2*}
		    0 =\cM^{2*}(P) - \cL_{1}^{2*}(P) ( \cN_1^{2*}(P) )^{-1} (\cL_1^{2*}(P))^\top.
		\end{equation}
	\end{corollary}
	
	\vspace{0.5cm}
    With a proper choice of $K_2^*$, we can connect equation \eqref{eq:ARE_y_K2*} to the (ARE-y).
    \begin{lemma}
	\label{lemma:equi_ARE_fix_K2*}
		For any symmetric matrix $P \in \cS^d$ such that $\cN_1(P)$ and $S_2 = \cN_2(P) - \cL_{12}(P)^\top \cN_1(P)^{-1} \cL_{12}(P)$ are invertible, let us consider that player 2 fixes her strategy with parameter
		\begin{equation}
	    	K_2^*  =  \Big( \cL_{12}(P)^\top \cN_2(P)^{-1} \cL_{12}(P) - \cN_2(P) \Big)^{-1} \Big( \cL_2(P)^\top - \cL_{12}(P)^\top \cN_1(P)^{-1} \cL_{1}(P)^\top \Big).
		\end{equation}
		Then $P$ is a solution to \eqref{eq:ARE_y_K2*} if and only if it is a solution to the (ARE-y) \eqref{eq:ARE_y}.
	\end{lemma}

	\begin{proof}
    To alleviate the notations, we omit the matrix $P$ in this proof. The idea is based on algebraic matrix manipulations with the help of Lemma~\ref{lemma:inverse_block} to invert a $2 \times 2$ block matrix $\cN(P)$.	
    First, we notice that 
    \begin{equation*}
         K_2^* = - S_2^{-1}  \Big( \cL_2^\top - \cL_{12}^\top \cN_1^{-1} \cL_{1}^\top \Big).
    \end{equation*}
    Then, the right hand side of ARE \eqref{eq:ARE_y_K2*} becomes: 
	\begin{align*}
		& \cM^{2*} - \cL_1^{2*} (\cN_1^{2*})^{-1} (\cL_1^{2*})^\top
		\\
		= & \cM + \cL_2 K_2^* + (K_2^*)^\top \cL_2^\top + (K_2^*)^\top \cN_2 K_2^* 
		\\
		& - \Big( \cL_1 \cN_1^{-1} \cL_1^\top + (K_2^*)^\top \cL_{12}^\top \cN_1^{-1} \cL_1^\top + \cL_1 \cN_1^{-1} \cL_{12} K_2^* + (K_2^*)^\top \cL_{12}^\top \cN_1^{-1} \cL_{12} K_2^* \Big)
		\\
		= & \cM - \cL_1 \cN_1^{-1} \cL_1^\top + \Big(\cL_2 - \cL_1 \cN_1^{-1} \cL_{12} \Big) K_2^* 
		\\
		& \qquad + (K_2^*)^\top \Big(\cL_2^\top - \cL_{12}^\top \cN_1^{-1} \cL_1^\top \Big) 
		+ (K_2^*)^\top \Big( \cN_2 - \cL_{12}^\top \cN_1^{-1} \cL_{12} \Big) K_2^*   .
	\end{align*}
	Since 
	$
		K_2^* = - S_2^{-1} \Big( \cL_2^\top - \cL_{12}^\top \cN_1^{-1} \cL_1^\top \Big)
	$
	and $ S_2 = \cN_2 - \cL_{12}^\top \cN_1^{-1} \cL_{12}$, we have
	\begin{equation*}
		\cM^{2*} - \cL_1^{2*} (\cN_1^{2*})^{-1} (\cL_1^{2*})^\top = \cM - \cL_1 \cN_1^{-1} \cL_1^\top - \Big( \cL_2 - \cL_1 \cN_1^{-1} \cL_{12} \Big) S_2^{-1} \Big( \cL_2 - \cL_1 \cN_1^{-1} \cL_{12} \Big)^\top .
	\end{equation*}
	\\
	Now, under the assumption that $\cN_1$ and $S_2$ are invertible,  we apply Lemma~\ref{lemma:inverse_block} to $\cN$ and obtain
	\begin{align*}	
		\cM - \cL \cN^{-1} \cL^\top
		= & \cM - [ \cL_1, \cL_2] 
			\left[ \begin{array}{ll} 
				\cN_1^{-1} + \cN_1^{-1} \cL_{12} S_2^{-1} \cL_{12}^\top \cN_1^{-1} &  - \cN_1^{-1} \cL_{12} S_2^{-1} \\
				- S_2^{-1} L_{12}^\top \cN_1^{-1}  & S_2^{-1}
			\end{array}
			\right]
			\left[ \begin{array}{c} \cL_1^\top \\ \cL_2^\top \end{array} \right]
		\\
		= & \cM - \Big( \cL_1 \cN_1^{-1} \cL_1 + (\cL_1 \cN_1^{-1} \cL_{12}) S_2^{-1} (\cL_{12}^\top \cN_1^{-1} \cL_1^\top ) 
		\\
		& \qquad - (\cL_1 \cN_1^{-1} \cL_{12}) S_2^{-1} \cL_2^\top - \cL_2 S_2^{-1} (\cL_{12}^\top \cN_1^{-1} \cL_{1}^\top ) + \cL_2 S_2^{-1} \cL_2^\top \Big)
		\\
		= & \cM - \cL_1 \cN_1^{-1} \cL_1^\top - \Big( \cL_2 - \cL_1 \cN_1^{-1} \cL_{12} \Big) S_2^{-1} \Big( \cL_2 - \cL_1 \cN_1^{-1} \cL_{12} \Big)^\top .
 	\end{align*}
 	Hence,
 	$
 		\cM^{2*} - \cL_1^{2*} (\cN_1^{2*})^{-1} (\cL_1^{2*})^\top
 		= 	\cM - \cL \cN^{-1} \cL^\top.
 	$
	\end{proof}

	We state here briefly the counterparts of Corollary~\ref{corollary:C_y_K1_K2*} and Lemma~\ref{lemma:equi_ARE_fix_K2*} for the situation when player 1 fixed her strategy to some predetermined matrix $K_1^* \in \RR^{\ell \times d}$.
	
	\begin{corollary}
		\label{Corollary:C_y_K1*_K2}

		\begin{align*}
		& C_y(K_1^*, K_2, \tilde{y}) - C^*_y(P; \tilde{y}) 
		\\
		&=  \EE \left[  \left. \sum_{t=0}^\infty
    		\gamma^t   \left[ \begin{array}{c} y_t^{1*} \\ \yut{2} \end{array} \right]^\top 
    		\left[ 
    		\begin{array}{ll} \cM^{1*}(P) & \cL_2^{1*}(P) \\
    		\cL_2^{1*}(P)^\top & \cN_2^{1*}(P)
    		\end{array}
    		\right]
    		\left[ \begin{array}{c} y_t^{1*} \\ \yut{2} \end{array} \right] \  
    		\right| y_0 = \tilde{y}  \right],
		\end{align*}
		where
		\begin{equation*}
		\left\{
		    \begin{aligned}
    		\cM^{1*}(P) & = \gamma ( A - B_1 K_1^*)^\top P (A - B_1 K_1^*) - P + Q + K_1^* R_1 K_1^*
    		\\
    		& = \cM(P) - \cL_1(P) K_1^* - (\cL_1(P) K_1^*)^\top + (K_1^*)^\top \cN_1(P) K_1^*
    		\\
    		\cL_2^{1*}(P) &= \gamma (A - B_1 K_1^*)^\top P B_2
    		= \cL_2(P) - (\cL_{12}(P)^\top K_1^*)^\top
    		\\
    		\cN_2^{1*}(P) & = \gamma B_2^\top P B_2 - R_2 = \cN_2(P),
    		\end{aligned}
		\right.
		\end{equation*}
		and the state process $(y_t^{1*})_{t \geq 0}$ follows the dynamics
		\begin{equation}
		\label{eq:dyn_y_K1*}
			y_{t+1}^{1*} = (A - B_1 K_1^*) y_t^{1*} + B_2 \yut{2} + \epsilon_{t+1}^1, \qquad y_0^{1*} = \tilde{y}.
		\end{equation}
	\end{corollary}

	\begin{lemma}
		If player 1 chooses her strategy with parameter
		\begin{align}
			K_1^* & = - \Big( \cL_{12}(P) \cN_2(P)^{-1} \cL_{12}(P)^\top - \cN_1(P) \Big)^{-1} \Big( \cL_1(P)^\top - \cL_{12}(P) \cN_2(P)^{-1} \cL_{2}(P)^\top \Big),  
		\end{align}
		where $P \in \cS^d$ and the matrices $\cN_2(P)$ and $S_1 = \cN_1(P) - \cL_{12}(P) \cN_2(P)^{-1} \cL_{12}(P)^\top $ are invertible, 
		then we have 
		\begin{equation}
			\cM^{1*}(P) - \cL_{2}^{1*}(P) ( \cN_2^{1*}(P) )^{-1} (\cL_2^{1*}(P))^\top = \cM(P) - \cL(P) \cN(P)^{-1} \cL(P)^\top.
		\end{equation}
	\end{lemma}
	
	\subsection{Sufficient condition}
	\label{sec:closed-loop-sufficient}
	 We now phrase a sufficient condition of optimality. The necessary part will be discussed in  Section~\ref{sec4:algo}, since it will serve as a basis for our numerical algorithms.  
	For a symmetric matrix $P \in \cS^d$, let us denote 
    \begin{align}
        K_i^* &= - \Big( \cL_{12}(P) \cN_j(P)^{-1} \cL_{12}(P)^\top - \cN_i(P) \Big)^{-1} \Big( \cL_i(P)^\top - \cL_{12}(P) \cN_j(P)^{-1} \cL_{j}(P)^\top \Big)
			  \label{eq:K_i^*} 
	\end{align}
    for $i \neq j$, $i,j \in \{1,2\}$, under the condition that the inverse of matrices involved above exist.
    
	\begin{proposition}
		\label{prop:sufficient_CLSP_y} 
		Assume that we have the following two conditions:
		\begin{enumerate}
		    \item The (ARE-y) \eqref{eq:ARE_y} $ \cM(P) - \cL(P) \cN(P)^{-1} \cL(P)^\top  = 0 $
		          admits a symmetric solution $P \in \cS^d $ satisfying
            		\begin{equation}
            		\label{eq:cond_R1_R2_CLSP_y}
            			\gamma B_1^\top P B_1 + R_1 \succ 0, \qquad \gamma B_2^\top P B_2 - R_2 \prec 0.
            		\end{equation}
		    \item  The pair of matrices $(K_1^*, K_2^*) \in \Theta_y$.
		\end{enumerate}
		Then $(K_1^*, K_2^*)$ is a $CLSP-y$ in $\Theta_y$.  
		Moreover, we have
		\begin{equation*}
			C_y(K_1^*, K_2^*, \tilde{y}) = C_y^*(P; y) = \tilde{y}^\top P \tilde{y} + \frac{\gamma}{1- \gamma} \EE\left[  (\epsilon_1^1)^\top P \epsilon_1^1 \right].
		\end{equation*}
		\\	
		The control processes $(u_{1,t}^{(y),*})_{t\geq 0}$ and $(u_{1,t}^{(y),*} )_{t \geq 0}$ corresponding to the $CLSP-y$ $(K_1^*, K_2^*)$ are given by, for every $t\geq 0$,
		\begin{equation}
		\label{eq:feedback_opt_controls}
			u_{1,t}^{(y),*} = -K_1^* y_t^*, \qquad  	u_{2,t}^{(y),*} = K_2^* y_t^*
		\end{equation}
		where the process $(y_t^*)_{t \geq 0}$ follows the dynamics
		\begin{equation*}
			y_{t+1}^* = (A - B_1 K_1^* + B_2 K_2^*) y_t^* + \epsilon_{t+1}^1, \qquad y_0^* = \tilde{y}.
		\end{equation*}	
		These two control processes satisfy the optimality condition: for every $t \geq 0$,
		\begin{equation}
		\label{eq:relationship_ut_yt_compact}
		    \cN(P) u_t^{(y), *} + \cL(P)^\top y_t^* = 0
		\end{equation}
		 where $u_t^{(y), *} = [ (u_{1,t}^{(y), *})^\top, ( u_{2,t}^{(y),*} )^\top ]^\top \in \RR^{2\ell},$ or equivalently
		\begin{subequations}
		\label{eq:relationship_ut_yt}
		    \begin{empheq}[left = \empheqlbrace]{align}
		        & \cN_1(P) u_{1,t}^{(y),*}+ \cL_{12}(P) u_{2,t}^{(y),*} = - \cL_1(P)^\top y_t^* \, ,
		        \label{eq:relation_ut_yt_1}
		        \\
	        	& \cL_{12}^\top(P) u_{1,t}^{(y),*} + \cN_2(P) u_{2,t}^{(y),*} = - \cL_2(P)^\top y_t^* \, .
	        	\label{eq:relation_ut_yt_2}
		    \end{empheq}
		\end{subequations}
	\end{proposition}

	\begin{proof}
		From condition \eqref{eq:cond_R1_R2_CLSP_y}, we have
		$$
			\cN_1(P) = \gamma B_1^\top P B_1 + R_1 \succ 0, \qquad \text{and} \qquad   \cL_{12}(P)^\top \cN_1(P)^{-1} \cL_{12}(P) - \cN_2(P)  \succ 0,
		$$
		so that the matrix $K_2^*$ is well defined in $\RR^{\ell \times d}$. Similarly, we have $\cN_2(P) \prec 0$ and $\cL_{12}(P) \cN_2(P)^{-1} \cL_{12}(P)^\top - \cN_1(P) \prec 0$, so that $K_1^*$ is well-defined too.
		Moreover, Lemma~\ref{lemma:inverse_block} implies that the $2 \times 2$ block matrix 
		$$
			\cN(P) = \left[ 
			\begin{array}{cc} 
				\cN_1(P) & \cL_{12}(P) \\
				\cL_{12}(P)^\top &  \cN_2(P)
			\end{array}
			\right]
		$$ is invertible. By applying Schur's lemma (Lemma~\ref{lemma:schur's_lemma}) to the block matrix shown in Corollary~\ref{corollary:opt_gap_Cy}, we get, for every time $t\geq 0$:
		\begin{align*}
			& \left[ \begin{array}{c} y_t \\ \yut{1} \\ \yut{2} \end{array} \right]^\top
			\left[ 
			\begin{array}{lll} \cM(P) & \cL_1(P) & \cL_2(P) \\
			\cL_1(P)^\top & \cN_1(P) & \cL_{12}(P) \\
			\cL_1(P)^\top & \cL_{12}(P)^\top &  \cN_2(P)  \\
			\end{array}
			\right]
			\left[ \begin{array}{c} y_t \\ \yut{1} \\ \yut{2} \end{array} \right]
			\nonumber \\
			= & y_t^\top \Big( \cM(P) - \cL(P) \cN(P)^{-1} \cL(P)^\top \Big) y_t 
			 \\
			 &\quad+ \Big(u_t^{(y)} + \cN(P)^{-1} \cL(P)^\top y_t \Big)^\top \cN(P) \Big(u_t^{(y)} + \cN(P)^{-1} \cL(P)^\top  y_t \Big)
			\nonumber \\
			= & (i)_t + (ii)_t,
		\end{align*}
		 where $u_t^{(y)} = [ (u_{1,t}^{(y)})^\top, ( u_{2,t}^{(y)} )^\top ]^\top \in \RR^{2 \ell}.$
		Since $P$ satisfies (ARE-y) \eqref{eq:ARE_y}, so that the first term $(i)_t = 0$ for every $t \geq 0$. Thus, by Corollary~\ref{corollary:opt_gap_Cy}, for every $(K_1, K_2) \in \Theta_y$ and $\tilde{y} \in \RR^d$,
		\begin{align*}
		    &C_y(K_1, K_2, \tilde{y}) - C^*(P; \tilde{y}) 
			\\
			&=   \EE \left[ \left. \sum_{t=0}^\infty \gamma^t \Big(u_t^{(y)} + \cN(P)^{-1} \cL(P)^\top y_t \Big)^\top \cN(P) \Big(u_t^{(y)} + \cN(P)^{-1} \cL(P)^\top y_t \Big)  \right| y_0 = \tilde{y} \right].
		\end{align*}
		If we choose $(K_1^*, K_2^*) \in \RR^{\ell \times d} \times \RR^{\ell \times d}$ satisfying equation \eqref{eq:relationship_ut_yt_compact}, i.e.
		$
			\cN(P) u_t^{(y), *} + \cL(P)^\top y_t^* = 0,
		$
	    then we obtain $(ii)_t = 0$ for every $t \geq 0$. In this case, we have
		$$
			C_y(K_1^*, K_2^*, \tilde{y}) = C^*_y(P; \tilde{y}) = \tilde{y}^\top P \tilde{y} + \frac{\gamma}{1- \gamma} \EE\left[  (\epsilon_1^1)^\top P \epsilon_1^1 \right] < \infty.
		$$
        			
		Let us move on to obtain expressions for $(K_1^*, K_2^*)$. Since the matrix $\cN(P)$ is invertible, there exists a unique solution $u_t^{(y), *}$ to \eqref{eq:relationship_ut_yt_compact} for every $t \geq 0$.	
		We plug-in the definition of $\cL(P)$, $\cN(P)$, and $u_t^{(y),*} = [ ( u_{1,t}^{(y),*} )^\top, ( u_{2,t}^{(y), *} )^\top ]^\top$,
		equation \eqref{eq:relationship_ut_yt_compact} is equivalent to the system of equations \eqref{eq:relation_ut_yt_1} and \eqref{eq:relation_ut_yt_2}.		
		So, by multiplying $\cL_{12}(P) \cN_2(P)^{-1}$ on both sides of \eqref{eq:relation_ut_yt_2}, and subtract it to \eqref{eq:relation_ut_yt_1}, we obtain 
		$$
			\Big(\cL_{12}(P) \cN_2(P)^{-1} \cL_{12}(P)^\top - \cN_1(P) \Big) u_{1,t}^{(y), *} = \Big( \cL_1(P)^\top - \cL_{12}(P) \cN_2(P)^{-1} \cL_2(P)^\top \Big) y_t^*.
		$$
		By the assumption on $\cN_1(P) \succ 0$ and $\cN_2(P) \prec 0$, we have $\cL_{12}(P) \cN_2(P)^{-1} \cL_{12}(P)^\top - \cN_1(P) \prec 0$ which is invertible. As a consequence, we obtain the optimal feedback control for player 1 by
		$$
			u_{1,t}^{(y), *} = - K_1^* y_t^*
		$$
		where $K_1^*$ is given by \eqref{eq:K_i^*}. Similarly, we can derive that $u_{2,t}^{(y), *} = K_2^* y_t^*$ with $K_2^*$ given by \eqref{eq:K_i^*}.
		Moreover, when we replace $u_{1,t=0}^{(y),*}$ and $u_{2,t=0}^{(y), *}$ with their expressions in \eqref{eq:feedback_opt_controls} back into \eqref{eq:relationship_ut_yt_compact}, and by noticing that it holds true for every $\tilde{y} \in \RR^{d}$, we have
		\begin{equation}
		\label{eq:sys_equation_K1*_K2*}
		\left\{
		    \begin{array}{l}
			    - \cN_1(P) K_1^* + \cL_{12}(P) K_2^* = - \cL_1(P)^\top
			    \\
			    - \cL_{12}(P)^\top K_1^* + \cN_2(P) K_2^* = - \cL_2(P)^\top.
			\end{array}
		\right.
		\end{equation}
		
		In the following, we will show that the pair $(K_1^*, K_2^*)$ is a $CLSP-y$, which means that it satisfies condition \eqref{eq:cond_CLSP_y}. First, under the assumption in the statement, we know that $(K_1^*, K_2^*)\in \Theta_y$. 
        Then, we look at the case when player 2 fixes her strategy to $K_2^*$, but player 1 adopts an alternative strategy $K_1$ satisfying $(K_1, K_2^*) \in \Theta_y$. The corresponding control at time $t$ is given then by $\yut{1} = - K_1 y_t^{2*}$, where the state process $(y_t^{2*})_{t\geq 0}$ follows dynamics \eqref{eq:dyn_y_K2*}:
		\begin{equation*}
			y_{t+1}^{2*}  = (A - B_1 K_1 + B_2 K_2^*) y_t^{2*}  + \epsilon_{t+1}^1  = (A + B K^*_2) y_t^{2*} + B_1 \yut{1} + \epsilon_{t+1}^1
		\end{equation*}
		with $y_0^{2*} = \tilde{y}$.
		By Corollary~\ref{corollary:C_y_K1_K2*} and Schur's lemma, we have
		\begin{align*}
			&	C_y(K_1, K_2^*, \tilde{y}) - C^*_y(P; \tilde{y})
			\\
			= & \EE \left[ \left. \sum_{t=0}^{\infty} \gamma^t (y_t^{2*} )^\top \Big( \cM^{2*}(P) - \cL_{1}^{2*}(P) ( \cN_1^{2*}(P) )^{-1} \cL_1^{2*}(P)^\top \Big) y_t^{2*} \right| y_0^{2*} = \tilde{y} \right]
			\\
			& + \EE \left[ \left. \sum_{t=0}^\infty \gamma^t \Big( \yut{1} + (\cN_1^{2*}(P) )^{-1} \cL_1^{2*}(P)^\top y_t^{2*} \Big)^\top
			\cN_1^{2*}(P) \right.\right.
			\\
			&\qquad \left.\left.\Big( \yut{1} + (\cN_1^{2*}(P) )^{-1} \cL_1^{2*}(P)^\top y_t^{2*}\Big)  \right| y_0^{2*} = \tilde{y}\right].
		\end{align*}
		\\
		From Lemma~\ref{lemma:equi_ARE_fix_K2*}, we know that a solution $P$ to (ARE-y) \eqref{eq:ARE_y} is also a solution to \eqref{eq:ARE_y_K2*}:
		$$
			0 =	\cM^{2*}(P) - \cL_{1}^{2*}(P) ( \cN_1^{2*}(P) )^{-1} (\cL_1^{2*}(P))^\top.
		$$
		Moreover, we have
		$
			- \cN_1(P) K_1^* + \cL_{12}(P) K_2^* = - \cL_1(P)^\top
		$
		which implies, by definition of $\cN_1^{2*}(P)$ and $\cL_1^{2*}(P)$, that
		\begin{equation*}
			- K_1^*  = - (\cN_1(P) )^{-1} \Big( \cL_1(P)^\top + \cL_{12}(P) K_2^* \Big) = - (\cN_1^{2*}(P) )^{-1} \cL_1^{2*}(P)^\top.
		\end{equation*}
		Thus, together with $\yut{1} = -K_1 y_t^{2*}$ for every $t \geq 0$, we have
		\begin{equation*}
			C_y(K_1, K_2^*, \tilde{y}) - C^*(P; \tilde{y})=  \EE \left[ \left. \sum_{t=0}^\infty \gamma^t (y_t^{2*})^\top (K_1^* - K_1)^\top \cN_1(P) (K_1^* - K_1) y_t^{2*} \right| y_0^{2*} = \tilde{y} \right] . 	
		\end{equation*}
		Consequently, the condition $\cN_1(P) = \gamma B_1^\top P B_1 + R_1 \succ 0$ implies
		\begin{equation*}
			C_y(K_1, K_2^*, \tilde{y}) - C_y(K_1^*, K_2^*, \tilde{y}) \geq 0.
		\end{equation*}
		Similarly, we have $K_2^* = - ( \cN_2^{1*}(P) )^{-1} \cL_2^{1*}(P)^\top$, so
		\begin{equation*}
			C_y(K_1^*, K_2, \tilde{y}) - C^*_y(P; \tilde{y})=  \EE \left[ \left. \sum_{t=0}^\infty \gamma^t (y_t^{1*})^\top (K_2^* - K_2)^\top \cN_2(P) (K_2^* - K_2) y_t^{1*} \right| y_0^{1*} = \tilde{y} \right]
		\end{equation*}
		where the process $(y_t^{1*})_{t \geq 0}$ satisfies the dynamics \eqref{eq:dyn_y_K1*} with intial $y_0^{1*} = \tilde y$.
		Thus, the condition $\cN_2(P) = \gamma B_2^\top P B_2 - R_2 \prec 0$ implies
		\begin{equation*}
			C_y(K_1^*, K_2, \tilde{y}) - C_y(K_1^*, K_2^*, \tilde{y}) \leq 0.
		\end{equation*}
	
	\end{proof}

	\begin{remark}
		We have emphasized previously that one difficulty in Proposition~\ref{prop:sufficient_CLSP_y} comes from the presence of a cross interaction between control processes $(\yut{1})_{t \geq 0}$ and $(\yut{2})_{t \geq 0}$ shown in Corollary~\ref{corollary:opt_gap_Cy}, which disappears in the continuous-time setting. Moreover, the solution $P$ to (ARE-y) is not involved in $\cN_1$ nor $\cN_2$ in the continuous-time case. Consequently, with continuous-time closed-loop feedback controls, a sufficient condition based only on model parameters ``$R_1 \succ 0$ and $- R_2 \prec 0$", which is similar to condition \eqref{eq:cond_R1_R2_CLSP_y} in Proposition~\ref{prop:sufficient_CLSP_y}, turns out to be a necessary condition for a $CLSP-y$ in $\Theta_y$ (see \cite{sun2016linear}). This is not the case in our discrete-time setting.
	\end{remark}

    We present here similar results corresponding to the $CLSP-z$.    Let the matrices $(\tilde N_1, \tilde N_2, \tilde L_1, \tilde L_2, \tilde L_{12}, \tilde M, \tilde N, \tilde L)(\bar{P})$ be defined by using the same expressions in equations \eqref{eq:notation_M_N_L}(a), (b), (c), but by replacing $(A, B_1, B_2, Q)$ to $(\tilde A, \tilde B_1, \tilde B_2, \tilde Q)$.
    
    For a symmetric matrix $\bar P \in \cS^d$, let us denote 
    \begin{align}
            L_i^* &= - \Big( \tilde \cL_{12}(\bar P) \tilde \cN_j(\bar P)^{-1} \tilde \cL_{12}(\bar P)^\top - \tilde \cN_i(\bar P) \Big)^{-1} \Big( \tilde  \cL_i(\bar P)^\top - \tilde \cL_{12}(\bar P) \tilde \cN_j(\bar P)^{-1} \tilde \cL_{j}(\bar P)^\top \Big)
		  \label{eq:L_i^*} 
	\end{align}
    for $i \neq j$, $i,j \in \{1,2\}$, under the condition that the inverse of matrices appearing above exist.

    \begin{lemma}
    \label{prop:sufficient_CLSP_z}
         Assume the following Algebraic Riccati equation (ARE-z):
            \begin{equation}
            \label{eq:ARE_z}
                0 = \tilde \cM(\bar P) - \tilde \cL( \bar P) \tilde \cN(\bar P ) \tilde \cL(\bar P)^\top
            \end{equation}
        admits a solution $\bar P \in \cS^d$ which is such that 
        \begin{equation}
            \label{eq:cond_R1_R2_CLSP_z}
        	\gamma \tilde B_1^\top \bar P \tilde B_1 + \tilde R_1 \succ 0
        	\qquad 
        	\text{ and }
    		\qquad
    		\gamma \tilde B_2^\top \bar P \tilde B_2 - \tilde R_2 \prec 0.
    	\end{equation}
		 Assume in addition that the pair of matrices $(L_1^*, L_2^*)$ given by \eqref{eq:L_i^*} is in $\Theta_z$. Then, $(L_1^*, L_2^*)$ is a $CLSP-z$ in $\Theta_z$.
    \end{lemma}

    \begin{corollary}
    \label{cor:sufficient_CLSP_z}
        If the two pair of matrices $(K_1^*, K_2^*) \in \Theta_y$ and $(L_1^*, L_2^*) \in \Theta_z$ defined in Lemmas~\ref{prop:sufficient_CLSP_y} and \ref{prop:sufficient_CLSP_z} are $CLSP-y$ and $CLSP-z$ respectively, then the closed-loop feedback policy profile $(\theta_1^*, \theta_2^*) \in \Theta_{ad}^{close}$ defined by
        $$
            \theta_1^* = (K_1^*, L_1^*), \qquad \text{and} \qquad \theta_2^* = (K_1^*, L_2^*)
        $$
        is a closed-loop saddle point for the zero-sum game. The optimal value of the utility function is given by
        \begin{equation*}
            C(\theta_1^*, \theta_2^*) = \EE \left[ \tilde{y}^\top P \tilde{y} \right] + \EE \left[ \tilde z^\top \bar P \tilde z \right] + \frac{\gamma}{1- \gamma} \EE\left[  (\epsilon_1^1)^\top P \epsilon_1^1 + (\epsilon_1^0)^\top \bar P \epsilon_1^0 \right],
        \end{equation*}
        where $P$ and $\bar P$ are solutions to the (ARE-y)~\eqref{eq:ARE_y} and (ARE-z)~\eqref{eq:ARE_z} satisfying conditions \eqref{eq:cond_R1_R2_CLSP_y} and \eqref{eq:cond_R1_R2_CLSP_z} respectively.
    \end{corollary}

%% file: sec_link-open-closed.tex
\section{Connection between closed-loop and open-loop Nash equilibria}
\label{sec:connect-open-closed}

In this section, we show that the open-loop and closed-loop equilibria are tightly related. To this end, we impose the following assumption on the model parameters.
\begin{assumption}
\label{assumption_B1_B2_invertible}
	We assume that $\ell = d$, and the matrices $B_1, B_2, R_1, R_2$ and $\tilde B_1, \tilde B_2, \tilde R_1, \tilde R_2$ are all invertible.
\end{assumption}
For an invertible matrix $S \in \RR^{d \times d}$, we denote $S^{-\top} = (S^\top)^{-1} = (S^{-1})^\top$. 
The following lemma relates the inverse of a 2-by-2 block matrix to its off-diagonal term \cite{lu2002inverses}. 
\begin{lemma}[Corollary~4.1 in \cite{lu2002inverses}]
\label{lemma:inverse_N_by_crossdiagnal}
	If the matrices $\cL_{12}(P) = \gamma B_1^\top P B_2 \in \RR^{2\ell}$ and  $S_3 = \cL_{12}(P)^\top - \cN_2(P) \cL_{12}(P)^{-1} \cN_1(P)$ are invertible, then matrix $\cN(P)$ has an inverse expressed by
	\begin{equation}
	\label{eq:N_inverse_crossdiag}
		\cN(P)^{-1} = 
		\left[
		\begin{array}{cc}
			\cN_1 & \cL_{12} \\
			\cL_{12}^\top & \cN_2
		\end{array}
		\right]^{-1}(P)
		= \left[
		\begin{array}{cc}
				-  S_3^{-1} \cN_2  \cL_{12}^{-1} & S_3^{-1}
				\\
				\cL_{12}^{-1} + \cL_{12}^{-1} \cN_1 S_3^{-1} \cN_2 \cL_{12}^{-1} &
				- \cL_{12}^{-1} \cN_1 S_3^{-1}
		\end{array}
			\right](P)\, .
	\end{equation}
\end{lemma}

If a solution $P^c \in \cS^d$ to (ARE-y) is invertible, we have an alternative expression for (ARE-y).
\begin{lemma}
\label{lemma:alternative_expression_ARE_y_close}
	We suppose that Assumption~\ref{assumption_B1_B2_invertible} holds. Let $P^c \in \cS^d$ be a solution to the Algebraic Riccati equation \eqref{eq:ARE_y}
	$$
	    0 = \cM(P^c) - \cL(P^c) \cN(P^c)^{-1}\cL(P^c)^\top.
	$$ 
	If $P^c$ and $(P^c)^{-1} + \gamma B_1 R_1^{-1} B_1^\top - \gamma B_2 R_2^{-1} B_2^\top$ are invertible, then 
   
   \begin{align}
   \label{eq:ARE_closed_loop_expression2}
		    &\cM(P^c) - \cL(P^c) \cN(P^c)^{-1}\cL(P^c)^\top =
			\nonumber\\
			&  \qquad \qquad =Q - P^c  + A^\top \left( \frac{1}{\gamma} (P^c)^{-1}  + B_1 R_1^{-1} B_1^\top - B_2 R_2^{-1} B_2^\top  \right)^{-1} A.
		\end{align}
\end{lemma}

\begin{proof}
	First, under Assumption~\ref{assumption_B1_B2_invertible}, we observe that
	\begin{align}
	S_3 & = \cL_{12}(P^c)^\top - \cN_2(P^c) \cL_{12}(P^c)^{-1} \cN_1(P^c) 
	\nonumber \\
	& =  \gamma B_2^\top P^c B_1 - (\gamma B_2^\top P^c B_2 - R_2) \left(\frac{1}{\gamma} B_2^{-1} (P^c)^{-1} B_1^{-\top} \right) (  \gamma B_1^\top P^c B_1 + R_1)
	\nonumber \\
	& = \gamma B_2^\top P^c B_1 - \left( B_2^\top B_1^{-\top} R_1 + \gamma B_2^\top P^c B_1 - \frac{1}{\gamma} R_2 B_2^{-1} (P^c)^{-1} B_1^{-\top} R_1 - R_2 B_2^{-1} B_1 \right)
	\nonumber \\
	& = -  B_2^\top B_1^{-\top} R_1 + \frac{1}{\gamma} R_2 B_2^{-1} (P^c)^{-1} B_1^{-\top} R_1 + R_2 B_2^{-1} B_1
	\nonumber \\
	& = R_2^\top B_2^{-1} \left( - B_2 R_2^{-1} B_2^\top + \frac{1}{\gamma} (P^c)^{-1} + B_1 R_1^{-1} B_1^\top \right) B_1^{-\top} R_1
	\label{eq:S3}
	\end{align}
	and also
	\begin{equation}
	\label{eq:N_1_L12}
    	\left\{
    	\begin{array}{rcl}
    	\cN_2(P^c) \cL_{12}(P^c)^{-1} &=& \displaystyle B_2^\top B_1^{-\top} - \frac{1}{\gamma} R_2 B_2^{-1} (P^c)^{-1} B_1^{-\top}
    	\\
    	\cL_{12}(P^c)^{-1} \cN_1(P^c) & = & \displaystyle \frac{1}{\gamma} B_2^{-1} (P^c)^{-1} B_1^{-T} R_1 + B_2^{-1} B_1.
    	\end{array} 
    	\right.
	\end{equation}
  Since $\cL_{12}(P^c)$ and $S_3$ are invertible, by Lemma~\ref{lemma:inverse_N_by_crossdiagnal} we obtain
	\begin{align}
	&\frac{1}{\gamma^2} \cL(P^c) \cN(P^c)^{-1} \cL(P^c)^\top
	\nonumber
	\\
	&= \left[ A^\top P^c B_1,\   A^\top P^c B_2 \right] 
	\cN(P^c)^{-1}
	\left[ \begin{array}{c} 
	B_1^\top P^c A 
	\nonumber \\
	B_2^\top P^c A
	\end{array}
	\right]	
	\nonumber \\
	&=  - A^\top P^c B_1  S_3^{-1} \cN_2 \cL_{12}^{-1} B_1^\top P^c A  + A^\top P^c B_1 S_3^{-1} B_2^\top P^c A
	\nonumber \\
	& \qquad + A^\top P^c B_2 \left( \cL_{12}^{-1} + \cL_{12}^{-1} \cN_1 S_3^{-1} \cN_2 \cL_{12}^{-1} \right) B_1^{\top} P^c A - A^\top P^c B_2 \cL_{12}^{-1} \cN_1 S_3^{-1} B_2^\top P^c A 
	\nonumber \\
	&=   (i) + (ii) + (iii) + (iv).	
	\label{eq:LNL}
	\end{align}
	We then use equation \eqref{eq:N_1_L12} to simplify $(i)$ and $(iv)$:
	\begin{align*}
		(i) &= - A^\top P^c B_1 S_3^{-1} \left( \cN_2 \cL_{12}^{-1} B_1^\top P^c A \right) 
		= - A^\top P^c B_1 S_3^{-1} \left( B_2^\top P^c A - \frac{1}{\gamma} R_2 B_2^{-1} A \right)
		\\
		&= - (ii) + \frac{1}{\gamma} A^\top P^c B_1 S_3^{-1} R_2 B_2^{-1} A ,
	\end{align*}
	and
	\begin{align*}
		(iv) &=  \left( - A^\top P^c B_2 \cL_{12} \cN_1 \right) S_3^{-1} B_2^\top P^c A
		=  \left( - \frac{1}{\gamma} A^\top B_1^{-\top} R_1 + A^\top (P^c)^\top B_1 \right) S_3^{-1} B_2^\top P^c A
		\\
		&=  - (ii) - \frac{1}{\gamma} A^\top B_1^{-T} R_1 S_3^{-1} B_2^\top P^c A .
	\end{align*}
	Moreover we have 
	\begin{align*}
		(iii) & =  A^\top P^c B_2 \cL_{12}^{-1} B_1^\top P^c A + \left( A^\top P^c B_2 \cL_{12}^{-1} \cN_1 \right) S_3^{-1} \left( \cN_2 \cL_{12}^{-1} B_1^{\top} P^c A \right)
		\\
		& = \frac{1}{\gamma} A^\top P^c A +  \left( \frac{1}{\gamma} A^\top B_1^{-\top} R_1 + A^\top (P^c)^\top B_1 \right) S_3^{-1} \left( B_2^\top P^c A - \frac{1}{\gamma} R_2 B_2^{-1} A \right)
		\\
		& = \frac{1}{\gamma} A^\top P^c A + (ii) - \frac{1}{\gamma^2} A^\top B_1^{-\top} R_1 S_3^{-1} R_2 B_2^{-1} A
		+ \frac{1}{\gamma} A^\top B_1^{-T} R_1 S_3^{-1} B_2^\top P^c A 
		\\
		&\qquad\qquad
		- \frac{1}{\gamma} A^\top P^c B_1 S_3^{-1} R_2 B_2^{-1} A.
	\end{align*}
	Then, equation \eqref{eq:LNL} becomes
	\begin{equation*}
		\cL(P^c) \cN(P^c)^{-1} \cL(P^c)^\top = \gamma A^\top P^c A - A^\top B_1^{-\top} R_1 S_3^{-1} R_2 B_2^{-1} A.
	\end{equation*}
	Together with equation \eqref{eq:S3}, we conclude that

	\begin{align*}
    	0 &= \cM(P^c) - \cL(P^c) \cN(P^c)^{-1} \cL(P^c)^\top 
    	\\
    	&= Q - P^c + A^\top \left( \frac{1}{\gamma} (P^c)^{-1}  + B_1 R_1^{-1} B_1^\top - B_2 R_2^{-1} B_2^\top  \right)^{-1} A.
   \end{align*}

\end{proof}

\begin{lemma}
\label{lemma:connect_ARE_open_closed}
    Assume that Assumption~\ref{assumption_B1_B2_invertible} holds. Let $P^o \in \RR^{d \times d}$ be a solution to the Algebraic Riccati equation \eqref{eq:main_ARE_P_ZS} derived from the open-loop information structure, namely:
	\begin{equation}
	\label{eq:ARE_y_open_repeat}
		P^o = \gamma\left( A^\top P^o + 2 Q \right) \left( A + (B_1 \Gamma_1 + B_2 \Gamma_2) P^o  \right)
	\end{equation}
	where $\Gamma_1 = -\frac{1}{2} R_1^{-1} B_1^\top$ and $\Gamma_2 = \frac{1}{2} R_2^{-1}B_2^\top$.
	We consider a matrix given by
	\begin{equation}
	\label{eq:ARE_open_close_transformation}
		P^c = \frac{1}{2} A^\top P^o + Q.
	\end{equation}
	If $A^\top P^o = (P^o)^\top A$, and the symmetric matrices $P^c$ and $(P^c)^{-1} + \gamma \left( B_1 R_1^{-1} B_1^{\top} - B_2 R_2^{-1} B_2^\top \right)$ are invertible, then the matrix $P^c \in \cS^d$ is a solution to (ARE-y) \eqref{eq:ARE_y} derived from the closed-loop information structure:
	$$
	    0 = \cM(P^c) - \cL(P^c) \cN(P^c)^{-1} \cL(P^c)^\top.
	$$
\end{lemma}

\begin{proof}
    By plugging the expression of $\gamma_1$ and $\gamma_2$ into equation \eqref{eq:ARE_y_open_repeat}, we obtain
    \begin{equation*}
        P^o + \gamma (A^\top P^o + 2 Q) \left(\frac{1}{2} B_1 R_1^{-1} B_1^{\top} - \frac{1}{2} B_2 R_2^{-1} B_2^\top  \right) P^o  = \gamma (A^\top P^o + 2 Q) A.
    \end{equation*}
    Let $P^c = \frac{1}{2} A^\top P^o + Q$, the above equation is equivalent to
    $$
        \left( \frac{1}{\gamma} I_d + P^c \left(B_1 R_1^{-1} B_1^{\top} - B_2 R_2^{-1} B_2^\top \right) \right)  P^o = 2 P^c A,
    $$
    Since we have assumed that $P^c$ and $(P^c)^{-1}+ \gamma \left( B_1 R_1^{-1} B_1^{\top} - B_2 R_2^{-1} B_2^\top \right) $ are invertible, we must have
    $$
    	  \frac{1}{2}P^o = \left( \frac{1}{\gamma} (P^c)^{-1}  + B_1 R_1^{-1} B_1^\top - B_2^\top R_2^{-1} B_2^\top  \right)^{-1}  A.
    $$
    If we multiply on both sides by $A^\top$ and then add $Q$, by rearranging the terms, we obtain
    $$
        0 = Q - P^c + A^\top \left( \frac{1}{\gamma} (P^c)^{-1}  + B_1 R_1^{-1} B_1^\top - B_2 R_2^{-1} B_2^\top  \right)^{-1} A.
    $$
\end{proof}

\begin{remark}
\label{rem:from-Pc-to-Po}
    In addition to Assumption~\ref{assumption_B1_B2_invertible}, if $A$ is invertible, then from a positive definite solution $P^c$ to (ARE-y) \eqref{eq:ARE_y}, we can define a matrix $P^o = 2 A^{-\top} (P^c- Q) \in \RR^{d \times d}$. By inverting the steps used in Lemma~\ref{lemma:connect_ARE_open_closed}, we can show that $P^o$ is a solution to the equation \eqref{eq:main_ARE_P_ZS}.
\end{remark}

The following corollary shows that both the pair of control processes associated to a closed-loop saddle point and the pair of processes for an open-loop Nash equilibrium will lead to the same state process, hence the same value function for the zero-sum game.

\begin{corollary}
    We assume that Assumption~\ref{assumption_B1_B2_invertible} holds and $A$, $\tilde A$ are invertible. Suppose that there exists unique invertible solutions $P^o$ (resp. $\bar P^o$) and $P^c$ (resp. $\bar{P}^c$) to the corresponding Algebraic Riccati equations in the open-loop information structure \eqref{eq:main_ARE_P_ZS} (resp.\eqref{eq:main_ARE_Pbar_ZS}) and in the closed-loop information structure \eqref{eq:ARE_y} (resp. \eqref{eq:ARE_z}).  Then, we have :
    \begin{enumerate}
        \item[(i)] The following holds, where  $( (K_1^*, L_1^*), (K_2^*, L_2^*) )$ are given in  \eqref{eq:K_i^*} and \eqref{eq:L_i^*}:
        \begin{equation}
            \label{eq:connection_feedback_coefficient}
            \left\{ 
                \begin{aligned}
                    - B_1 K_1^* + B_2 K_2^*  & = B_1 \Gamma_1 P^o + B_2 \Gamma_2 P^o
                    \\
                    - \tilde B_1 L_1^* + \tilde B_2 L_2^*  & = \tilde B_1 \Lambda_1 \bar P^o + \tilde B_2 \Lambda_2 \bar P^o.
                \end{aligned}
            \right.
        \end{equation}

        \item[(ii)] For every time $t \geq 0$, the state variable $x_t^{\theta_1^*, \theta_2^*}$ corresponding to a pair of closed-loop feedback control $(\bu^{c, *}_1, \bu^{c, *}_2)$ with policies  $(\theta_1^*, \theta_2^*) = ( (K_1^*, L_1^*), (K_2^*, L_2^*) )$ \eqref{eq:CLFB_control}  has the same distribution as the state variable $x_t^{\bu_1^{o,*}, \bu_2^{o,*} }$ controlled by an open-loop Nash equilibrium $(\bu_1^{o,*}, \bu_2^{o,*})$ with parameters $(P^o, \bar P^o)$  \eqref{eq:MKV-opt-ctrl-formula}. 
        \end{enumerate}
\end{corollary}

\begin{proof}
    According to Lemma~\ref{lemma:connect_ARE_open_closed}, the unique solutions $P^o$ (resp. $\bar P^o$) and $P^c$ (resp. $\bar P^c$) to the corresponding Algebraic Riccai equation satisfy:
    $$
        P^c = \frac{1}{2} A^\top P^o + Q, \qquad \text{and} \qquad \bar P^c = \frac{1}{2} (A + \bar A)^\top \bar P^o + (Q + \bar Q).
    $$
    
    (i) $\quad$ It is enough to show the connection between $(K_1^*, K_2^*)$ to the pair of matrices $(-\Gamma_1 P^o, - \Gamma_2 P^o)$, and the situation for $(L_1^*, L_2^*)$ can be proved with similar arguments.
    Let us denote by 
     $
        \check B = [ B_1, B_2] \in \RR^{d \times 2\ell}
    $
    and 
    $\check R = 
        \left[ \begin{array}{cc} R_1 & 0 \\ 0 & -R_2 \end{array} \right].
    $
    Then, by equation \eqref{eq:sys_equation_K1*_K2*} and Lemma \ref{lemma:alternative_expression_ARE_y_close}, since $A$ is invertible, we have:
    \begin{align*}
         - B_1 K_1^* + B_2 K_2^* & = - \check B ( \gamma \check B^\top P^c \check B + \check R)^{-1}\left( \gamma \check B^\top P^c A \right)
         \\
          & = \left( I_{d} + \gamma \check B \check R^{-1} \check B^\top P^c \right)^{-1} A - A
          \\
          &= - \left( \check B \check R^{-1} \check B^\top\right) \left( A^{-\top} (P^c - Q) \right).
    \end{align*}
    Together with $P^o = 2 A^{-\top}(P^c - Q)$ and the definition of $\Gamma_1, \Gamma_2$, we obtain
    \begin{equation*}
        - B_1 K_1^* + B_2 K_2^* = -\frac{1}{2} \check B \check R^{-1} \check B^\top P^o = B_1 \Gamma_1 P^o + B_2 \Gamma_2 P^o.
    \end{equation*}

    (ii) By comparing the state dynamics of $ \left( x_t^{\bu_1^{o,*}, \bu_2^{o,*}}  - \bar x_t^{\bu_1^{o,*}, \bu_2^{o,*} } \right)_{t \geq 0}$ in the open-loop case and that of $(y_t^{\theta_1^*, \theta_2^*})_{t \geq 0}$ \eqref{eq:dyn_y_closed} in the closed-loop case, we have 
    $$
        y_t^{\theta_1^*, \theta_2^*} \stackrel{d}{=} x_t^{\bu_1^{o,*}, \bu_2^{o,*}}  - \bar x_t^{\bu_1^{o,*}, \bu_2^{o,*} }
    $$
    in the sense of distribution. Similar arguments shows  $z_t^{\theta_1^*, \theta_2^*} \stackrel{d}{=} \bar x_t^{\theta_1^*, \theta_2^*}$. Thus, the conclusion holds.
     
\end{proof}

\begin{remark}
In the case with only one decision-maker, similar to equation~\eqref{eq:connection_feedback_coefficient}, we can prove that $- B_1 K_1^* = B_1 \Gamma_1 P^o$ and $-\tilde B_1 L_1^* = \tilde B_1 \Lambda_1 \bar P^o$. Then the control process in a single population model for the open-loop and for the closed-loop information structures are identical at every time $t \geq 0$, namely
$$
    u_{1,t}^{c, *} = - B_1 K_1^* (x_t^{\theta_1^*} - \bar{x}_t^{\theta_1^*} ) - \tilde{B}_1 L_1^* \bar{x}_t^{\theta_1^*} = B_1 \Gamma_1 P^o (x_t^{\bu_1^*} - \bar{x}_t^{\bu_1^*}) + \tilde{B}_1 \Lambda_1 \bar P^o \bar{x}^{\bu_1^*} = u_{1,t}^{o,*}. 
$$
\end{remark}

%% file: sec_numerics.tex
\section{Algorithms}
\label{sec4:algo}

In this section, we propose policy-gradient based algorithms to find the Nash equilibrium of the zero-sum mean-field type game. We start with a convenient expression for the gradient of the utility function, which leads to a necessary condition of optimality (counterpart to the sufficient condition studied in \S~\ref{sec:closed-loop-sufficient}). Then, after introducing model-based methods, we explain how to extend them to sample-based algorithms in which the gradient is estimated using a simulator providing stochastic realizations of the utility. The results of this section have initially been presented in~\cite{carmona2020lqzsmftg}.

\subsection{Gradient expression} 
\label{sec2:opt}

We henceforth replace problem~\eqref{eq:prob_for} by the following problem based on closed-loop controls, introduced in Section~\ref{sec:closed-loop-structure}.   Each player $i=1,2$ chooses parameter $\theta^*_i = ( K^*_i, L^*_i)$ 
such that
$$
    J(\mathrm{\bf u}^{ \theta^*_1}_1, \mathrm{\bf u}^{\theta^*_2}_2)
    =
    \inf_{\theta_1} \sup_{\theta_2} J(\mathrm{\bf u}^{\theta_1}_1, \mathrm{\bf u}^{\theta_2}_2),
$$
where for $\theta = (\theta_1,\theta_2) \in \Theta$ (recall the parameter set $\Theta$ defined in~\eqref{eq:Theta_stable_set}),
$$
    u^{\theta_1,\theta_2}_{i,t} = (-1)^i K_i(x^{\mathrm{\bf u}_1^{\theta_1},\mathrm{\bf u}_2^{\theta_2}}_t 
    - \bar{x}^{\mathrm{\bf u}_1^{\theta_1},\mathrm{\bf u}_2^{\theta_2}}_t) 
    + (-1)^i L_i \bar{x}^{\mathrm{\bf u}_1^{\theta_1},\mathrm{\bf u}_2^{\theta_2}}_t, \qquad i=1,2.
$$

For simplicity, we introduce the following notation
$
    x^{\mathrm{\bf u}_1^{\theta_1},\mathrm{\bf u}_2^{\theta_2}}
    =
    x^{\theta_1,\theta_2},
$
and since we focus on linear controls, using the notation $C$ introduced in~\eqref{eq:def-Ctheta1theta2}, we have
$$
    C(\theta_1,\theta_2)
    =
    J(\mathrm{\bf u}^{\theta_1}_1, \mathrm{\bf u}^{\theta_2}_2).
$$
Moreover, we introduce $y^{K_1,K_2}_t = x^{\theta_1,\theta_2}_t - \bar{x}^{\theta_1,\theta_2}_t$ and $z^{L_1,L_2}_t = \bar{x}^{\theta_1,\theta_2}_t$, which is justified by the fact that the dynamics of $\by$ and $\bz$ depend respectively only on $(K_1,K_2)$ and $(L_1,L_2)$. 

Let $P^y_{K_1,K_2}$ be a solution to the linear equation
    \begin{align}
        \label{eq:eq-Py_K1K2}
        P^y_{K_1,K_2}
        &=
        Q + K_1^\top R_1 K_1 - K_2^\top R_2 K_2
        \\
        & \qquad + \gamma (A - B_1 K_1 + B_2 K_2)^\top P^y_{K_1,K_2} (A - B_1 K_1 + B_2 K_2),
        \notag
    \end{align}
and let $P^z_{L_1,L_2}$ be a solution to the linear equation
    \begin{align}
        \label{eq:eq-Pz_L1L2}
        P^z_{L_1,L_2}
        &=
        \tilde Q  + L_1^\top \tilde R_1 L_1 - L_2^\top \tilde R_2 L_2
        \\
        \notag
        & \qquad + \gamma (\tilde A - \tilde B_1 L_1 + \tilde B_2 L_2)^\top P^z_{L_1,L_2} (\tilde A - \tilde B_1 L_1 + \tilde B_2 L_2).
    \end{align}
     
\commentJDG{
In order to guarantee that the above equations have solutions, we introduce the notion of stabilizing parameters. 
\begin{definition}
The set of stabilizing parameters is defined as follows:
\begin{align}
    \Theta =
    &\Big\{(K_1,L_1,K_2,L_2) \,:\,
    \gamma \|A - B_1 K_1 + B_2 K_2\|^2 < 1,
    \nonumber \\
    &\quad \gamma \| \tilde A - \tilde B_1 L_1 + \tilde B_2 L_2\|^2 < 1 
    \Big\}.
    \label{eq:Theta_stable_set}
\end{align}
\end{definition}}

\commentJDG{
We consider the parameter set $\Theta$ introduced in~\eqref{eq:Theta_stable_set}.

More details on this closed-loop information structure and the corresponding optimality conditions are provided in Appendix~\ref{sec:closed-loop-structure}.
}

We now provide an explicit expression for the gradient of the utility function with respect to the control parameters in terms of the solution to the equations~\eqref{eq:eq-Py_K1K2} and~\eqref{eq:eq-Pz_L1L2}. Let us denote
\begin{align*}
    & 
    \begin{bmatrix}
        E^{y,1}_{K_1,K_2}\\
        E^{y,2}_{K_1,K_2}
    \end{bmatrix}
    = - \gamma  
        \begin{bmatrix} 
        B_1^\top  P^y_{K_1,K_2} A
        \\
      - B_2^\top  P^y_{K_1,K_2} A
        \end{bmatrix}
    +
    \mathbf{R} 
    \begin{bmatrix}
        K_1\\
        K_2
    \end{bmatrix},
    \begin{bmatrix}
        E^{z,1}_{L_1,L_2}\\
        E^{z,2}_{L_1,L_2}
    \end{bmatrix}
    = - \gamma  
        \begin{bmatrix} 
        \tilde B_1^\top  P^z_{L_1,L_2} \tilde A
        \\
      - \tilde B_2^\top  P^z_{L_1,L_2} \tilde A
        \end{bmatrix}
    +
    \tilde {\mathbf{R}} 
    \begin{bmatrix}
        L_1\\
        L_2
    \end{bmatrix}
\end{align*}
with 
$$
    \mathbf{R} =
    \begin{bmatrix}
        R_1 + \gamma B_1^\top  P^y_{K_1,K_2} B_1 & - \gamma B_1^\top P^y_{K_1,K_2} B_2   
    \\
    - \gamma B_2^\top P^y_{K_1, K_2} B_1 &  - R_2 + \gamma B_2^\top  P^y_{K_1,K_2} B_2 
    \end{bmatrix},
$$
$$
    \tilde {\mathbf{R}} =
    \begin{bmatrix}
        \tilde R_1 + \gamma \tilde B_1^\top  P^z_{L_1,L_2} \tilde B_1 & - \gamma \tilde B_1^\top P^z_{L_1,L_2} \tilde B_2   
    \\
    - \gamma \tilde B_2^\top P^z_{L_1,L_2} \tilde B_1 &  - \tilde R_2 + \gamma \tilde B_2^\top  P^z_{L_1,L_2} \tilde B_2 
    \end{bmatrix}
$$
where
$$
    \Sigma^y_{K_1,K_2}
    =\EE\left[ \sum_{t \ge 0} \gamma^t y^{K_1,K_2}_t (y^{K_1,K_2}_t)^\top \right],
    \qquad
    \Sigma^z_{L_1,L_2}
    =\EE\left[ \sum_{t \ge 0} \gamma^t z^{L_1,L_2}_t (z^{L_1,L_2}_t)^\top \right].
$$

\begin{proposition}[Policy gradient expression]
\label{lem:PG_expression}
For any $\theta = (\theta_1,\theta_2) \in \Theta$, we have for $j=1,2$,
\begin{align}
\label{eq:PG_expression_Kj}
    \nabla_{K_j}C(\theta_1,\theta_2)
    & =
    2 E^{y,j}_{K_1,K_2} \Sigma^y_{K_1,K_2},
    \qquad
    \nabla_{L_j}C(\theta_1,\theta_2)
    & =
    2 E^{z,j}_{L_1,L_2} \Sigma^z_{L_1,L_2}.
\end{align} 

\end{proposition}
\begin{proof}
    We note that the utility can be split as
    $$
        C(\theta_1,\theta_2)
        =
        \EE_{\tilde y, \tilde z}\Big[
        C_{y}(K_1,K_2, \tilde y)
        +
        C_{z}(L_1,L_2, \tilde z)
        \Big]
    $$
    where $C_y, C_z$ are defined by~\eqref{eq:C_y_K1_K2}--\eqref{eq:C_z_L1_L2}. Let us consider the first part. We have 
    \begin{align*}
    C_{y}(K_1,K_2, \tilde y) 
    &=
    \EE \sum_{t \ge 0} \gamma^t \Big[
    (y^{L_1,L_2}_t)^{\top} Q y^{K_1,K_2}_t
    \\ 
    &\quad + (u_{1,t}-\bar{u}_{1,t})^{\top} R_{1} (u_{1,t}-\bar{u}_{1,t}) 
     - (u_{2,t} - \bar{u}_{2,t})^\top R_{2} (u_{2,t} - \bar{u}_{2,t}) 
    \,|\, y_0 = \tilde y
    \Big].
    \end{align*}
    We note, using the above definition together with~\eqref{eq:eq-Py_K1K2} and the dynamics satisfied by $y^{K_1,K_2}_t = x^{\theta_1,\theta_2}_t - \bar{x}^{\theta_1,\theta_2}_t$, that
    \begin{align*}
    C_{y}(K_1,K_2, \tilde y) 
    &= \tilde y^{\top} P^y_{K_1,K_2} \tilde y
    + \frac{\gamma}{1-\gamma} \EE[(\epsilon^1_1)^{\top} P^y_{K_1,K_2} \epsilon^1_1],
    \end{align*}
    from which we deduce that
    $
    \nabla_{\tilde y} C_y(K_1,K_2,\tilde y)
    = 2 P^y_{K_1,K_2} \tilde y.
    $ 
    Moreover,
    \begin{align*}
    C_{y}(K_1,K_2, \tilde y) 
    &=
    \tilde y^\top (Q + K_1^\top R_1 K_1 - K_2^\top R_2 K_2) \tilde y
    \\
    &\quad + \gamma \EE \Big[
    C_{y}\Big(K_1,K_2, (A - B_1 K_1 + B_2 K_2) \tilde y\Big)
    \,|\, y_0 = \tilde y
    \Big].
    \end{align*}
    Using the two above equalities and the chain rule, we obtain (using the fact that $R_1$ is symmetric)
    \begin{align*}
    \nabla_{K_1} C_{y}(K_1,K_2, \tilde y) 
    &=
    2 R_1 K_1 \tilde y \tilde y^\top 
    - 2 \gamma B_1^\top  P^y_{K_1,K_2} (A - B_1 K_1 + B_2 K_2) \tilde y \tilde y^\top
    \\
    &\quad + \gamma \EE \Big[
    \nabla_{K_1} C_{y}\Big(K_1,K_2, \tilde y'\Big)_{\big| \tilde y' = (A - B_1 K_1 + B_2 K_2) \tilde y + \epsilon^1_1}
    \Big].
    \end{align*}
     Using recursion and the equation satisfied by $P^y_{K_1,K_2}$ leads to the expression~\eqref{eq:PG_expression_Kj} for $\nabla_{K_1} C(\theta_1,\theta_2) $. 
    One can proceed similarly for the gradients with respect to $K_2$, $L_1$ and $L_2$.
    \commentJDG{
    Using recursion and the equation satisfied by $P^y_{K_1,K_2}$ leads to
    \begin{align*}
    \nabla_{K_1} C_{y}(K_1,K_2, \tilde y) 
    &=
    2 [R_1 K_1  - \gamma B_1^\top  P^y_{K_1,K_2} (A - B_1 K_1 + B_2 K_2)] \times
    \\
    &\quad 
    \left(\tilde y \tilde y^\top 
    + \EE\left[ \sum_{t \ge 1} \gamma^t y^{K_1,K_2}_t (y^{K_1,K_2}_t)^\top \right] \right) .
    \end{align*}
    Using similar computations, we obtain the expression for $\nabla_{K_2} C_{y}(K_1,K_2, \tilde y)$. 
    }
    \commentJDG{
    \begin{align*}
    \nabla_{K_2} C_{y}(K_1,K_2, \tilde y) 
    &=
    2 [ - R_2 K_2  + \gamma B_2^\top  P^y_{K_1,K_2} (A - B_1 K_1 + B_2 K_2)] \times
    \\
    &\quad 
    \left(\tilde y \tilde y^\top 
    + \EE\left[ \sum_{t \ge 1} \gamma^t y^{K_1,K_2}_t (y^{K_1,K_2}_t)^\top \right] \right) .
    \end{align*}
    One can proceed similarly for the gradients with respect to $L_1$ and $L_2$.
    }
\end{proof}

\subsection{Model-based policy optimization} 
Let us assume that the model is known and both players can see the actions of one another at the end of each time step. 
To explain the intuition behind the iterative methods, we first express the optimal control of a player when the other player has a fixed control. 
 For some given $\theta_2=(K_2,L_2)$, the inner minimization problem for player $1$ becomes an LQR problem with instantaneous utility at time $t$:
\begin{align*}
    &(x_t-\bar{x}_t)^{\top} \mathbf{Q}_{K_2} (x_t-\bar{x}_t) + \bar{x}^{\top} \mathbf{\tilde Q}_{K_2} \bar{x}
    \\ 
    &\qquad + (u_{1,t}-\bar{u}_{1,t})^{\top} R_{1} (u_{1,t}-\bar{u}_{1,t}) + \bar{u}^\top_{1,t} (R_{1}+\bar{R}_{1}) \bar{u}_{1,t}, 
\end{align*}
when player $1$ uses control $u_1$, where $\mathbf{Q}_{K_2}=Q-K_2 R_2 K_2$ and $\mathbf{\tilde Q}_{L_2}= \tilde Q - L_2\tilde R_2 L_2$, and state dynamics given by:
\begin{equation*}
    x_{t+1} = \mathbf{A}_{K_2} x_t + \mathbf{\bar A}_{K_2,L_2} \bar{x}_t
    + B_1 u_{1,t} 
    + \bar{B}_1 \bar{u}_{1,t}
    + \epsilon^0_{t+1} + \epsilon^1_{t+1},
\end{equation*}
where $\mathbf{A}_{K_2}=A+B_2K_2$ and $\mathbf{\bar A}_{K_2,L_2}=\bar A + \bar B_2L_2 + B_2(L_2 - K_2)$. 
Inspired by the results in~\cite{fazel2018global}, we propose to find the stationary point $\theta_1^*(\theta_2) = (K_1^*(K_2), L_1^*(L_2))$ of the inner problem. By setting
$\nabla_{\theta_1}C(\theta_1,\theta_2) = 0$ and by Proposition~\ref{lem:PG_expression}, this yields
\begin{equation}
\label{eq:K1star_K2}
    K_1^*(K_2) = \gamma (R_1+\gamma B_1^{\top}P^{y}_{K_2}B_1)^{-1} B_1^{\top} P^y_{K_2}\left[A + B_2 K_2\right],
\end{equation}
where $P^y_{K_2} = P^y_{K_1^*(K_2),K_2}$ solves 

  \begin{align*}
        &P^y_{K_2}
        =
        \tilde Q_{K_2} 
        + \gamma \tilde A_{K_2}^\top P^y_{K_2} \tilde A_{K_2}
        - \gamma^2 \tilde A_{K_2}^\top P^y_{K_2} B_1 (R_1+\gamma B_1^{\top}P^{y}_{K_2}B_1)^{-1} B_1^\top P^y_{K_2} \tilde A_{K_2},
    \end{align*}
where $\tilde{Q}_{K_2}=Q-K_2^{\top}R_2 K_2$ and $\tilde{A}_{K_2} = A+B_2K_2$. This equation is obtained by considering the equation~\eqref{eq:eq-Py_K1K2} for $P^y_{K_1,K_2}$ and replacing $K_1$ by the above expression~\eqref{eq:K1star_K2} for $K_1^*(K_2)$. One can similarly introduce $K_2^*(K_1)$, which is the optimal $K_2$ for a given $K_1$, and likewise for $L_1^*(L_2), L_2^*(L_1)$.

Based on this idea and inspired by the works of Fazel et al.~\cite{fazel2018global} and Zhang et al.~\cite{zhang2019policy}, we propose two iterative algorithms relying on policy-gradient methods, namely alternating-gradient and gradient-descent-ascent, to find the optimal values of $\theta_1$ and $\theta_2$. Starting from an initial guess of the control parameters, the players update either alternatively or simultaneously their parameters by following the gradients of the utility function. In the \emph{alternating-gradient} (AG) method, the players take turn in updating their parameters. Between two updates of $\theta_2$, $\theta_1$ is updated $\Iiter_1$ times. This procedure is summarized in Algorithm~\ref{algo:AG-ZSMFC}, which is based on nested loops. In the \emph{gradient-descent-ascent} (GDA) method, all the control parameters are updated synchronously at each iteration, as presented in Algorithm~\ref{algo:GDA-ZSMFC}. 

At each step of these methods, the gradients can be computed directly using the formulas provided in Proposition~\ref{lem:PG_expression}. For instance, in the inner loop of the alternating-gradient method, based on~\eqref{eq:PG_expression_Kj}, parameter $K_1$ can be updated as follows:
\begin{align*}
    K_1^{\iiter_1+1,\iiter_2} 
    &= K_1^{\iiter_1,\iiter_2}-\eta_1 \nabla_{K_1} C(\theta_1^{\iiter_1,\iiter_2},\theta_2^{\iiter_2-1})
    \\ 
    \nonumber 
    &=K^{\iiter_1,\iiter_2}-2\eta_1\Big[(R_1+B_1^{\top}P^{y}_{K_1^{\iiter_1,\iiter_2},K_2^{\iiter_2-1}}B_1)K_1^{\iiter_1,\iiter_2}
    \\
    &\qquad\qquad -B^{\top}_1P^{y}_{K_1^{\iiter_1,\iiter_2},K_2^{\iiter_2-1}}\tilde{A}_{K_2^{\iiter_2-1}}\Big]\Sigma_{K_1^{\iiter_1,\iiter_2},K_2^{\iiter_2-1}}.  
\end{align*}
Then, in the outer loop, one can compute $\nabla_{K_2} C$ at the point $(\theta_1^{\Iiter_1,\iiter_2},\theta_2^{\iiter_2-1})$ using again~\eqref{eq:PG_expression_Kj}.

In order to have a benchmark, one can compute the equilibrium $(\theta_1^*,\theta_2^*)$ by solving the Riccati equations~\eqref{eq:ARE_y}--\eqref{eq:ARE_z}. Alternatively, the Nash equilibrium can be computed by finding $K_2$ such that $\nabla_{K_2} C_y(K_1^*(k_2),K_2)_{\big|k_2 = K_2} = 0$. The left-hand side has an explicit expression obtained by combining~\eqref{eq:PG_expression_Kj} and~\eqref{eq:K1star_K2}.

\begin{algorithm}[H]
\caption{Alternating-Gradient method}
\label{algo:AG-ZSMFC}
\begin{algorithmic} 
    \REQUIRE Number of inner and outer iterations $\Iiter_1, \Iiter_2$; initial guess $\theta_1^0, \theta_2^0$; learning rates $\eta_1,\eta_2$
    \ENSURE $(K^*_1,K^*_2)$ 
        \STATE $\theta_1^{0,1} \leftarrow \theta_1^0$\;
        \FOR{$\iiter_2 = 1, 2, \dots, \Iiter_2$ }  
			\FOR{$\iiter_1 = 1, 2, \dots, \Iiter_1$ } 
			\STATE  
			$\theta_1^{\iiter_1,\iiter_2} \leftarrow \theta_1^{\iiter_1-1,\iiter_2} - \eta_1 \nabla_{\theta_1} C(\theta_1^{\iiter_1-1,\iiter_2},\theta_2^{\iiter_2-1})$\; 
	     	\ENDFOR
			\STATE $\theta_2^{\iiter_2} \leftarrow \theta_2^{\iiter_2-1} + \eta_2 \nabla_{\theta_2} C(\theta_1^{\Iiter_1,\iiter_2},\theta_2^{\iiter_2-1})$\; 
		\ENDFOR
		\RETURN $(\theta_1^{\Iiter_1,\Iiter_2},\theta_2^{\Iiter_2})$
\end{algorithmic}
\end{algorithm}   

\begin{algorithm}[H]
\caption{Gradient-Descent-Ascent method}
\label{algo:GDA-ZSMFC}
\begin{algorithmic} 
    \REQUIRE Number of iterations $\Iiter$; initial guess $\theta_1^{0}, \theta_2^0$; learning rates $\eta_1,\eta_2$
    \ENSURE $(K^*_1,K^*_2)$ 
        \FOR{$\iiter = 1, 2, \dots, \Iiter$ } 
			\STATE  
			$\theta_1^{\iiter} \leftarrow \theta_1^{\iiter-1} - \eta_1 \nabla_{\theta_1} C(\theta_1^{\iiter-1},\theta_2^{\iiter-1})
			$\; 
			\STATE $\theta_2^{\iiter} \leftarrow \theta_2^{\iiter-1} + \eta_2 \nabla_{\theta_2} C(\theta_1^{\iiter-1},\theta_2^{\iiter-1})$\; 
	     	\ENDFOR
		\RETURN $(\theta_1^{\Iiter},\theta_2^{\Iiter})$
\end{algorithmic}
\end{algorithm}

\subsection{Sample-based policy optimization}

The aforementioned methods use explicit expressions for the gradients, which rely on the knowledge of the model (the coefficients of the dynamics and the utility function). However, in many situations these coefficients are not known. Instead, let us assume that we have access to the following (stochastic) simulator, called \emph{MKV simulator} and denoted by $\cS^{\mathcal{T}}_{MKV}$: given a control parameter $\theta = (\theta_1, \theta_2) = (K_1, L_1, K_2, L_2)$, $\cS^{\mathcal{T}}_{MKV}(\theta)$ returns a sample of the mean-field utility (i.e., the quantity inside the expectation in equation~\eqref{fo:MKV-discounted_utility_ZS}) for the MKV dynamics~\eqref{eq:MKV-state_ZS} using the control $\theta$ and truncated at time horizon $\mathcal{T}$, which is similar to the one introduced in~\cite{carmona2019linear} (when there is a single controller). In other words, it returns a realization of the social utility $\sum_{t=0}^{\mathcal{T}-1} \gamma^t c_{t}$, where $c_t$ is the instantaneous mean-field utility at time $t$, see~\eqref{eq:instantaneous-utility}. This is used in Algorithm~\ref{algo:ZSMFC-MKVestim}, which provides a way to estimate the gradient of the utility with respect to the control parameters of the first player. One can estimate the gradient with respect to the control parameters of the second player in an analogous way. The estimation algorithm uses the simulator to obtain realizations of the (truncated) utility when using perturbed versions of the controls. In order to estimate the gradient of $C_y$, we use $2M$ perturbations $v_{1,1,i}, v_{1,2,i}$ which are i.i.d. with uniform distribution $\mu_{\mathbb{S}_\tau}$ over the sphere $\mathbb{S}_\tau$ of radius $\tau$. The first index corresponds to the player ($1$ or $2$), the second index corresponds to the part of the control being perturbed ($K$ or $L$) and the last index corresponds to the index of the perturbation (between $1$ and $M$). See e.g.~\cite{fazel2018global} for more details.  Notice that, although the simulator needs to know the model in order to sample the dynamics and compute the utilities, Algorithm~\ref{algo:ZSMFC-MKVestim} uses this simulator as a black-box (or an oracle), and hence uses only samples from the model and not the model itself.

\begin{algorithm}
	\caption{Sample-Based Gradient Estimation for Player~1}
	\label{algo:ZSMFC-MKVestim}
	\begin{algorithmic}
	\STATE {\bfseries Data:} {Parameter $\theta = (\theta_1, \theta_2) = (K_1,L_1,K_2,L_2)$; number of perturbations $M$; length $\mathcal{T}$; radius $\tau$}
	\STATE {\bfseries Result:} {An estimator for $\nabla_{\theta_1} C(\theta)$}
	
		\FOR{$i = 1, 2, \dots, M$}
			\STATE Sample $v_{1,1,i}, v_{1,2,i}$ i.i.d. $\sim \mu_{\mathbb{S}_\tau}$\; 
			\STATE Set $\check\theta_{1,i} := (K_{1,i}, L_{1,i}) := (K_1 + v_{1,1,i}, L_1 + v_{1,2,i})$\;
			\STATE Set $\check\theta_i = (\check\theta_{1,i}, \theta_{2})$\; 
			\STATE Sample $\tilde{C}^i$ using MKV simulator  $\cS^{\mathcal{T}}_{MKV}(\check\theta_i)$\;
		\ENDFOR
		\STATE Set $\tilde{\nabla}_{K_1} C(\theta) =   \frac{d}{\tau^2} \frac{1}{M} \sum_{i=1}^M \tilde{C}^i v_{1,1,i},$ 
		\STATE and $\tilde{\nabla}_{L_1} C(\theta) = \frac{d}{\tau^2} \frac{1}{M} \sum_{i=1}^M \tilde{C}^i  v_{1,2,i}$ \;
		\STATE {\bfseries Return: }{$\tilde\nabla_{\theta_1} C(\theta)  := \diag\left(\tilde{\nabla}_{K_1} C(\theta), \tilde{\nabla}_{L_1} C(\theta)   \right)$}
\end{algorithmic}
\end{algorithm}

\subsection{Numerical Results}
\label{sec5:num}

We now provide numerical results both for model-based and sample-based versions of the two methods presented in the previous section.  

\textbf{Setting.} The specification of the model used in the simulations is given in Table~\ref{tab:simulation_parameters_ZS}. This setting has been chosen so that it allows us to illustrate the convergence of the method when the equilibrium controls are not symmetric, i.e. $\theta_1 \neq \theta_2$. To be able to visualize the convergence of the controls, we focus on a one-dimensional example, that is, $d = \ell = 1$.

\textbf{Model-based results. }
The parameters used are given in Table~\ref{tab:simulation_parameters_ZS}. This choice of parameters is based on the values used for a single controller in~\cite{carmona2019linear} and numerical experiments.

Fig.~\ref{fig:1d-exact-surfaces} displays the trajectory of $(K_1,K_2) \mapsto C_y(K_1,K_2)$ and  $(L_1,L_2) \mapsto C_z(L_1,L_2)$ generated by the iterations of AG and DGA methods. Iterations are counted in the following way:  in AG at iteration $k$, $(\theta_1^k, \theta_2^k) = (\theta_1^{k \, \textrm{mod}\, \Iiter_1, \lceil k/\Iiter_1 \rceil}, \theta_2^{\lceil k/\Iiter_1 \rceil-1})$, while in DGA one step of for-loop corresponds to one iteration.  The utility at the starting point and the utility at the Nash equilibrium are respectively given by a black star and a red dot.
In the AG method, since $\theta_1$ is updated $\Iiter_1$ times between two updates of $\theta_2$, the trajectory moves faster in the  $\theta_1$-direction until it reaches an approximate best response against $\theta_2$, after which the trajectory moves towards the Nash equilibrium.  This is also confirmed by the convergence of the parameters $\theta = (K_1,L_1,K_2,L_2)$ in Fig.~\ref{fig:1d-exact-params}. 
The relative error on the utility is shown in Fig.~\ref{fig:1d-exact-utility}. We observe that the convergence is slower with AG because player $2$ updates her control only every $\Iiter_1$ iterations.

\begin{figure}[h]
	\begin{subfigure}{0.45\columnwidth}
		\centering
		\includegraphics[width=1.05\textwidth]{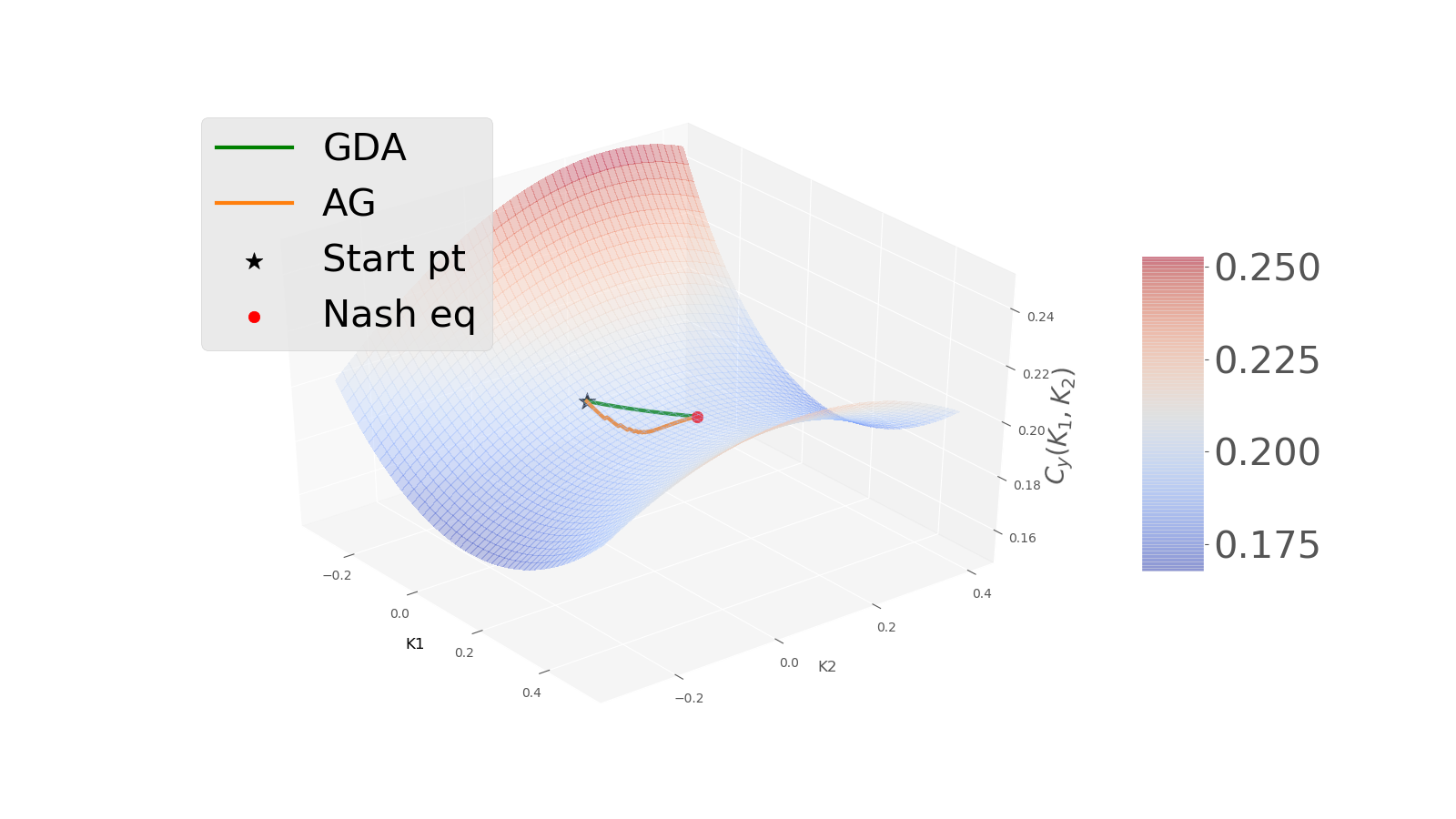}
		\caption{\,}
		\label{fig:1d-exact-surface-K}
	\end{subfigure}%
	\begin{subfigure}{0.45\columnwidth}
		\centering 
		\includegraphics[width=1.05\textwidth]{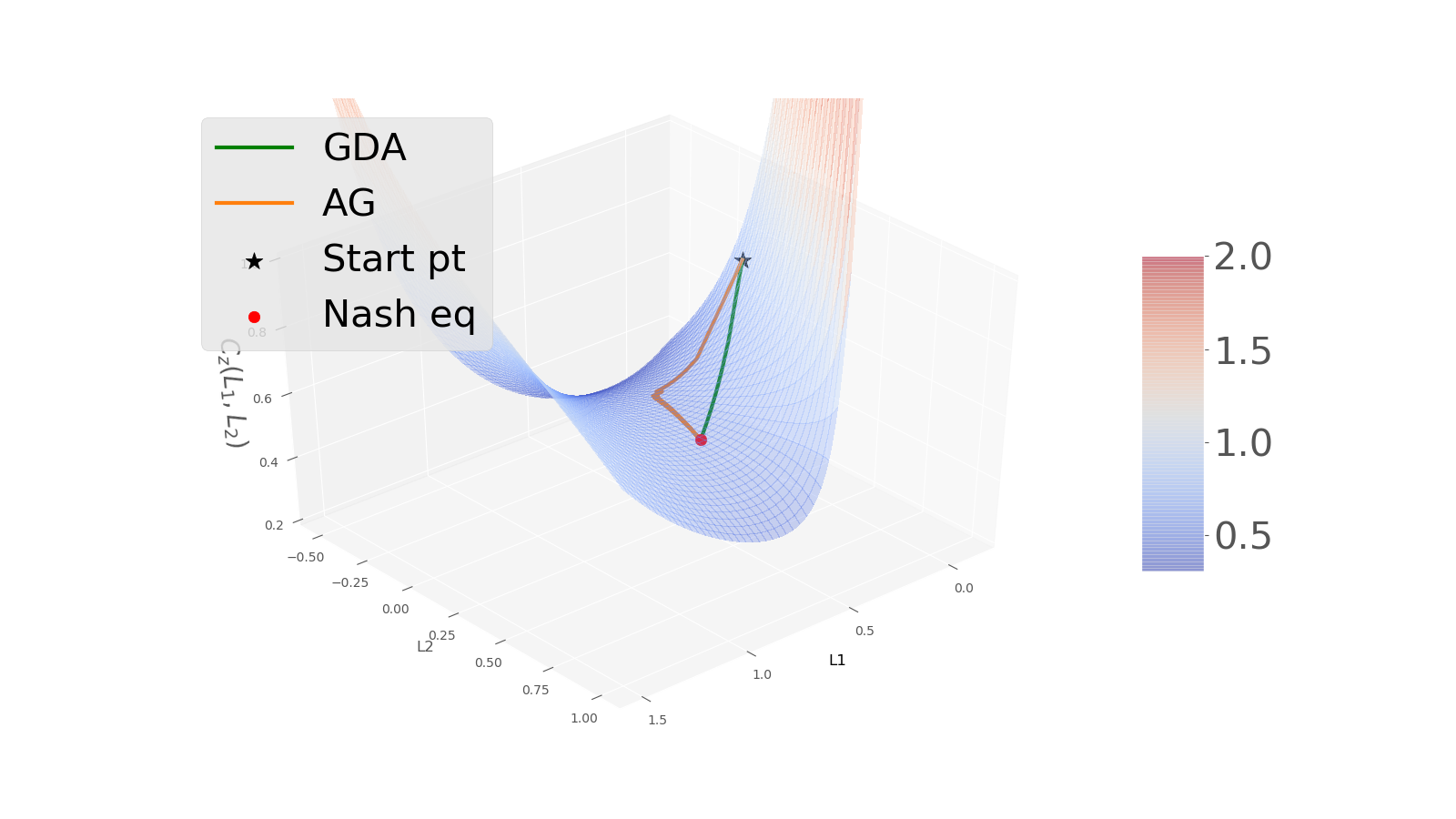}
		\caption{\,}
		\label{fig:1d-exact-surface-L}
	\end{subfigure}
	\caption{Model-based policy optimization: Convergence of each part of the utility. (a)~$C_y$ as a function of $(K_1,K_2)$. (b)~$C_z$ as a function of $(L_1,L_2)$.}
	\label{fig:1d-exact-surfaces}
\end{figure}

\begin{figure}[h]
	\begin{subfigure}{0.45\columnwidth}
		\centering 
		\includegraphics[width=1.05\columnwidth]{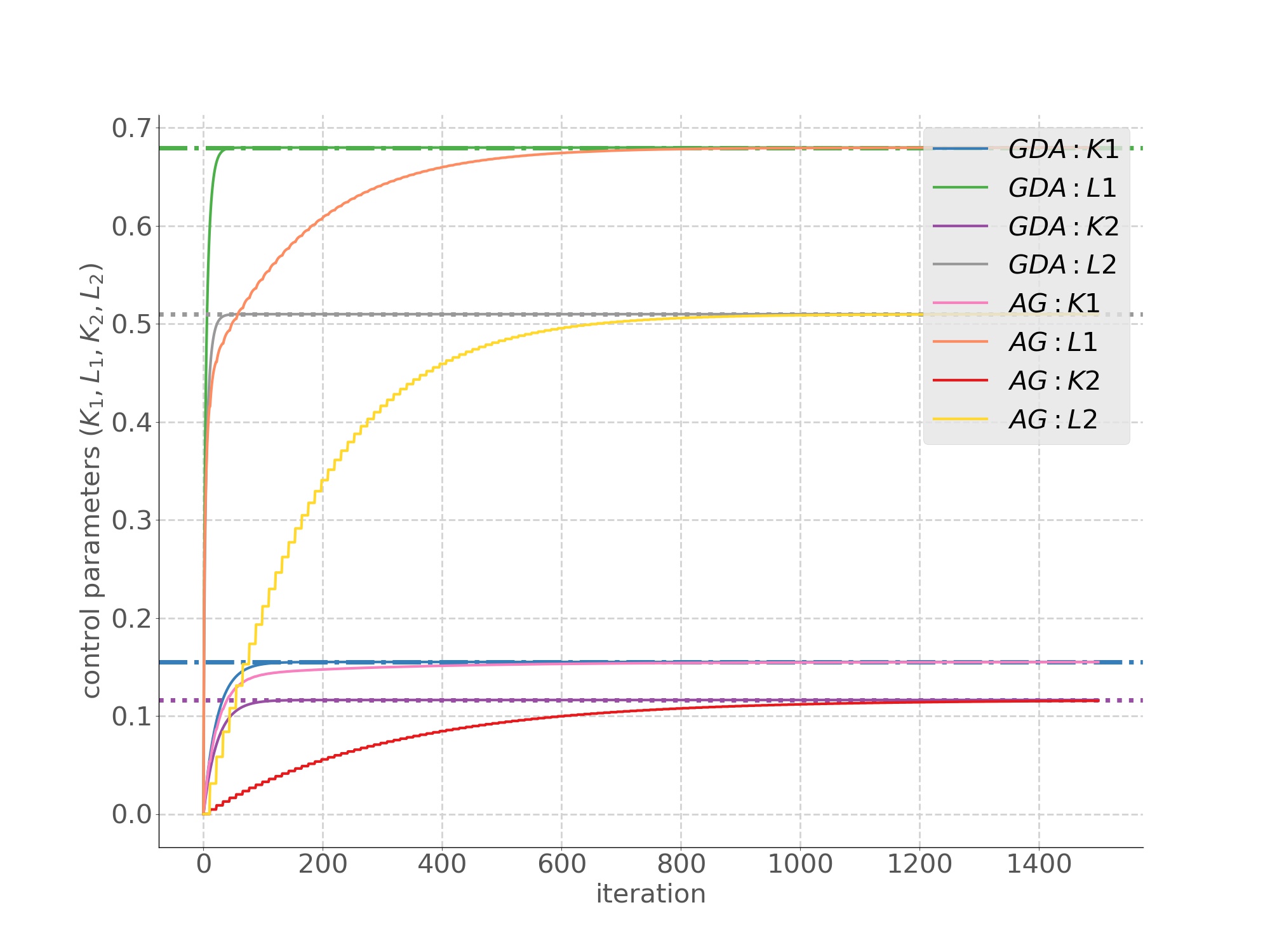}
		\caption{\,}
		\label{fig:1d-exact-params}
	\end{subfigure}%
	\begin{subfigure}{0.45\columnwidth}
		\centering  
		\includegraphics[width=1.05\columnwidth]{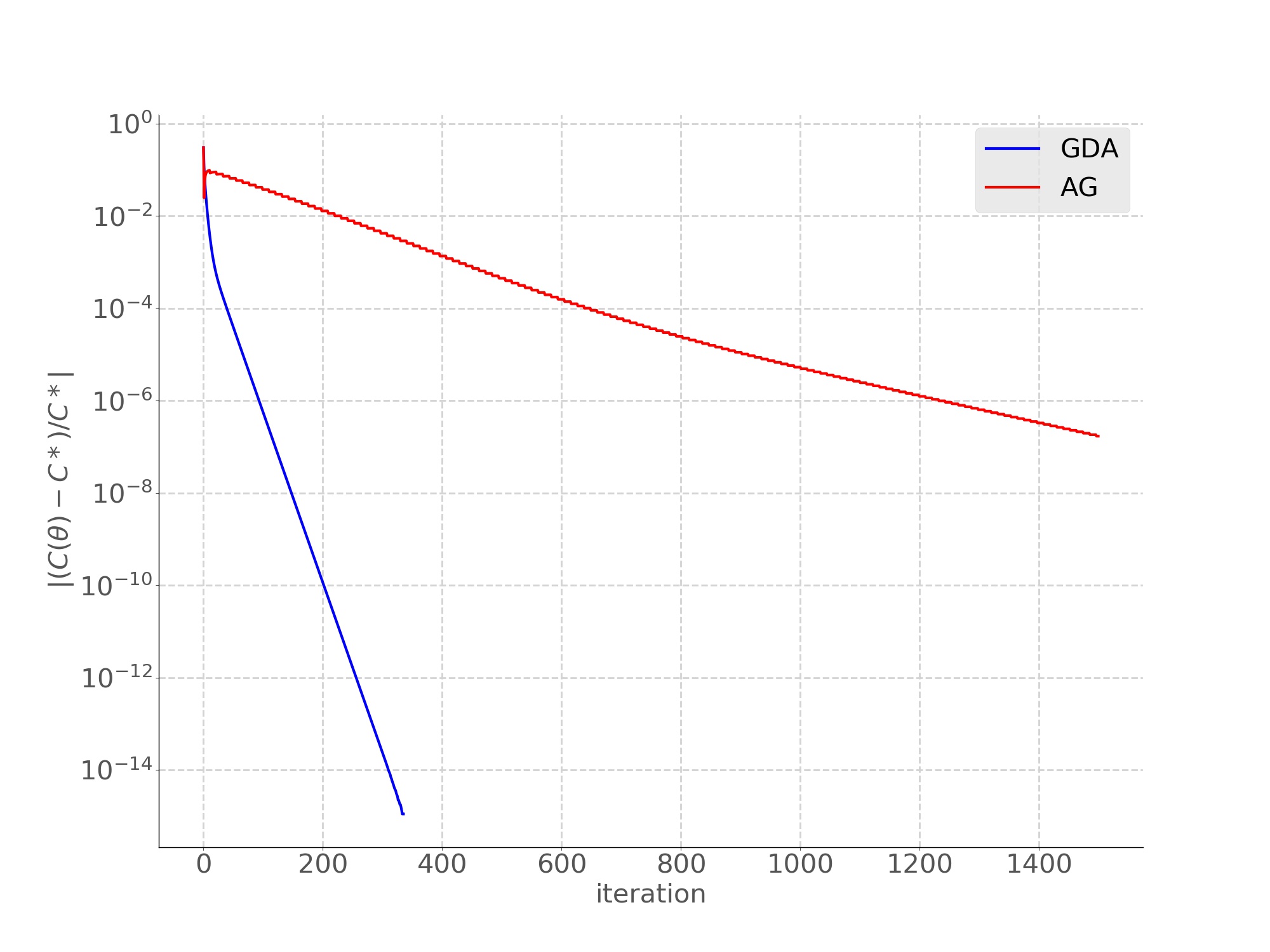}
		\caption{\,}
		\label{fig:1d-exact-utility}
	\end{subfigure}
	\caption{Model-based policy optimization: Convergence of the control parameters in~(a) and of the relative error on the utility in~(b).}
	\label{fig:1d-exact-params-utility}
\end{figure}

\textbf{Sample-based results. } The parameters used are given in Table~\ref{tab:simulation_parameters_ZS} and were chosen based on the values in~\cite{carmona2019linear} as well as numerical experiments. The figures are obtained by averaging the results over 5 experiments, each based on a different realization of the randomness in the initial points, in the dynamics and in the gradient estimation. 

Fig.~\ref{fig:1d-modelfree-surfaces} displays the trajectory of $(K_1,K_2) \mapsto C_y(K_1,K_2)$ and  $(L_1,L_2) \mapsto C_z(L_1,L_2)$ generated by the iterations of AG and DGA methods. The convergence of the parameters $\theta = (K_1,L_1,K_2,L_2)$ is shown in Fig.~\ref{fig:1d-modelfree-params}. 
The evolution of the relative error on the utility is shown in Fig.~\ref{fig:1d-modelfree-utility}.

\begin{figure}[h]
	\begin{subfigure}{0.45\columnwidth}
		\centering 
		\includegraphics[width=1.05\columnwidth]{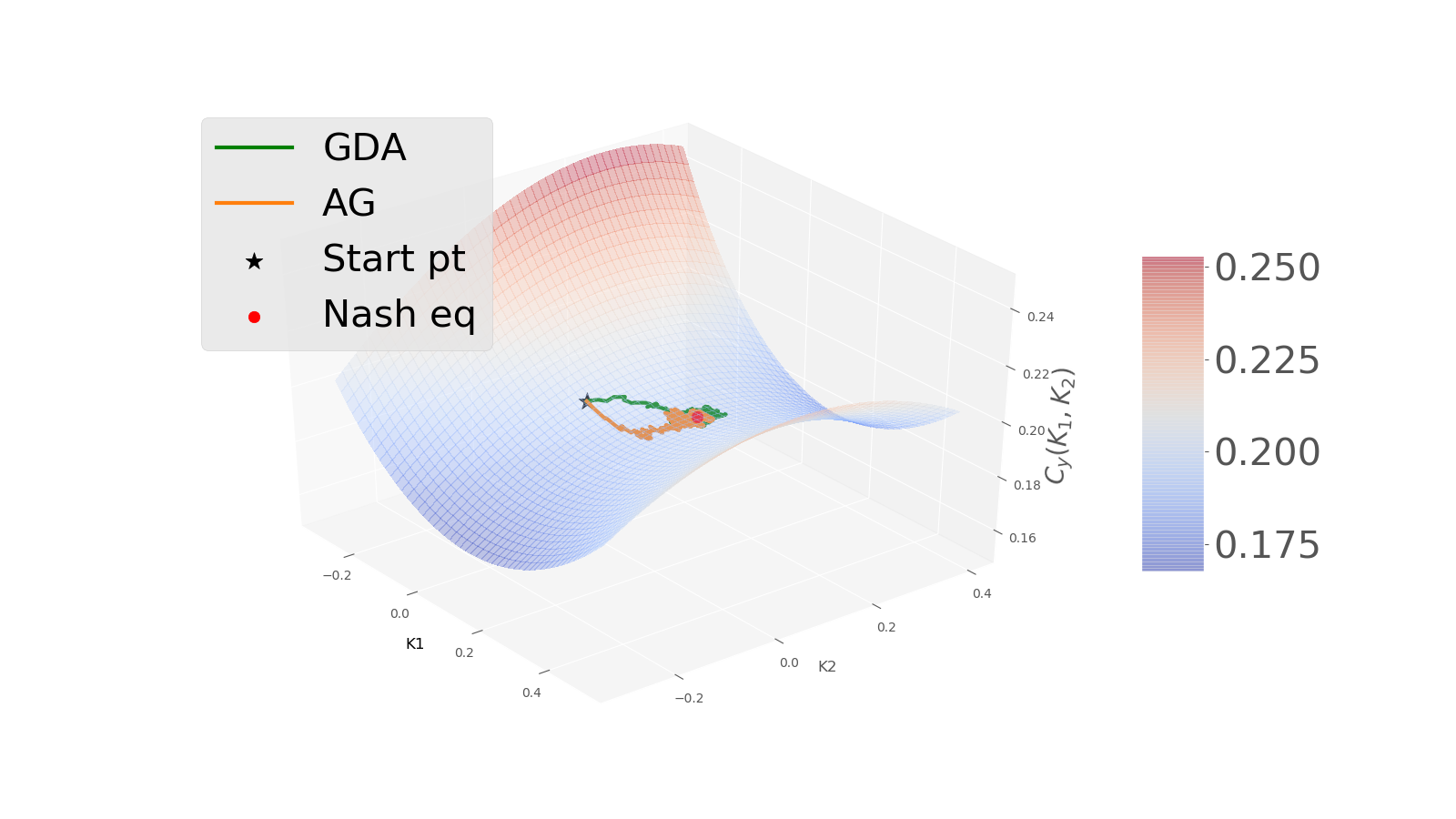}
		\caption{\,}
		\label{fig:1d-modelfree-surface-K}
	\end{subfigure}%
	\begin{subfigure}{0.45\columnwidth}
		\centering  
		\includegraphics[width=1.05\columnwidth]{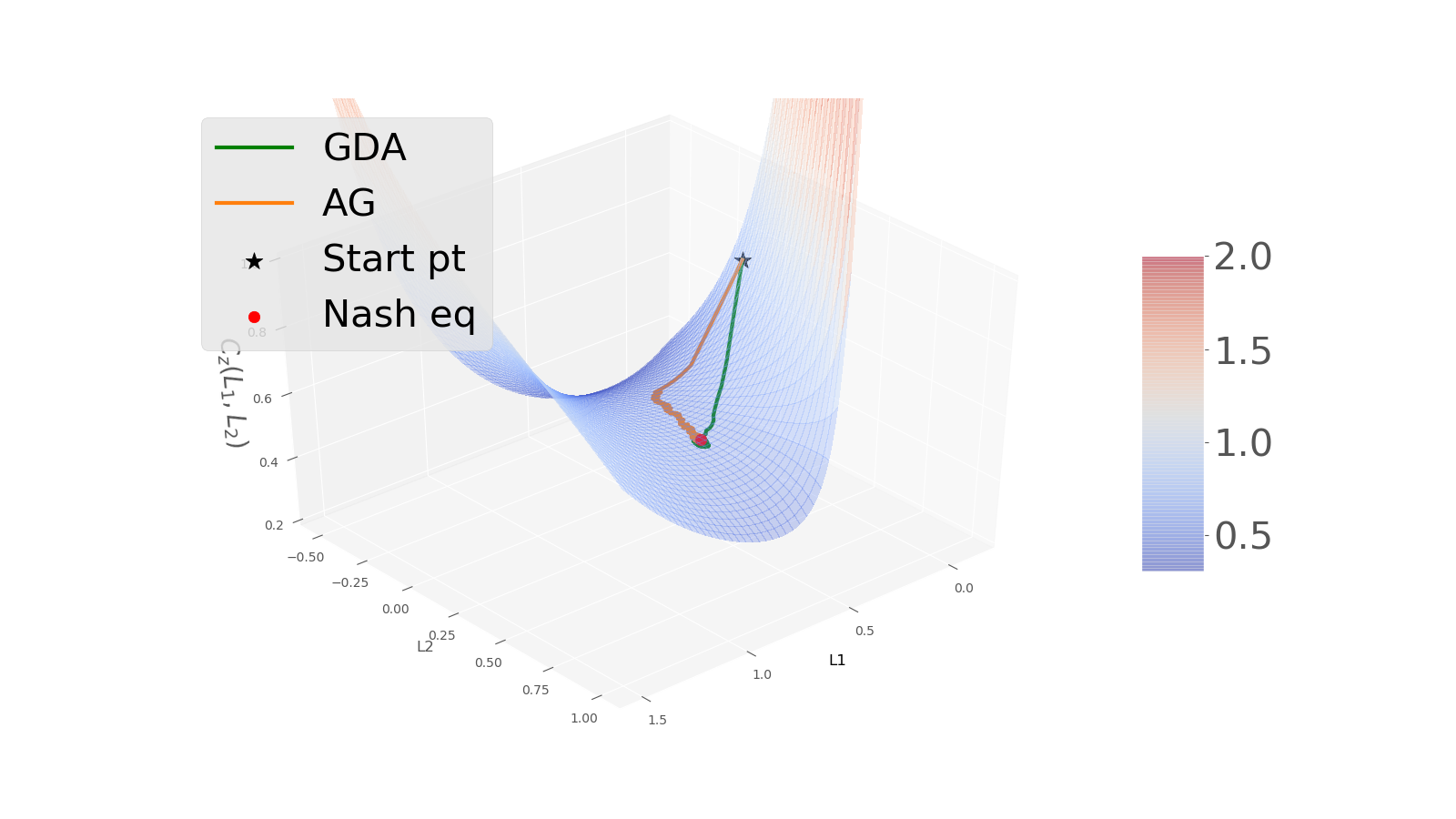}
		\caption{\,}
		\label{fig:1d-modelfree-surface-L}
	\end{subfigure}
	\caption{Sample-based policy optimization: Convergence of each part of the utility. (a)~$C_y$ as a function of $(K_1,K_2)$. (b)~$C_z$ as a function of $(L_1,L_2)$.}
	\label{fig:1d-modelfree-surfaces}
\end{figure}

\begin{figure}[h]
	\begin{subfigure}{0.45\columnwidth}
		\centering 
		\includegraphics[width=1.05\columnwidth]{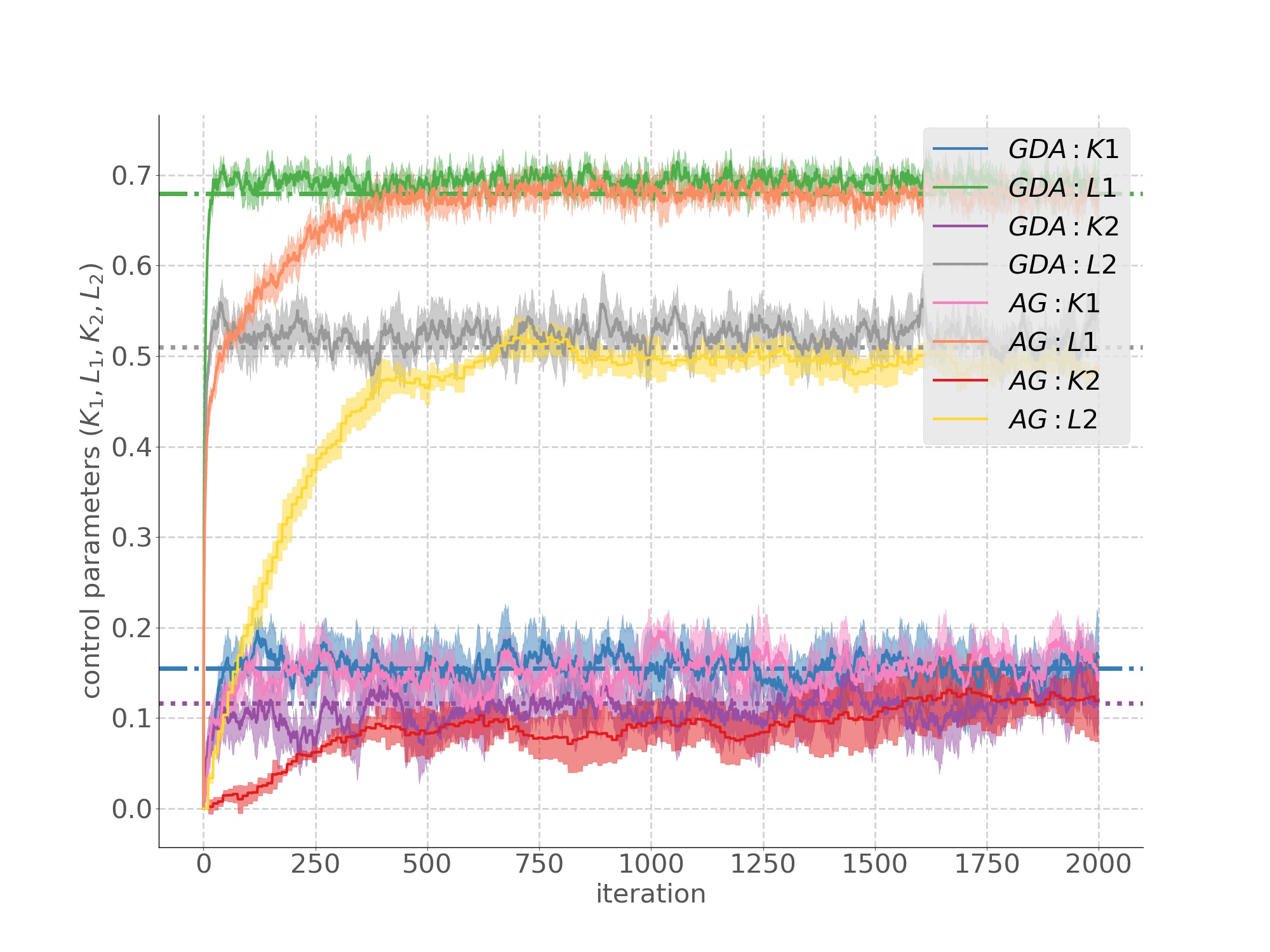}
		\caption{\,}
		\label{fig:1d-modelfree-params}
	\end{subfigure}%
	\begin{subfigure}{0.45\columnwidth}
		\centering  
		\includegraphics[width=1.05\columnwidth]{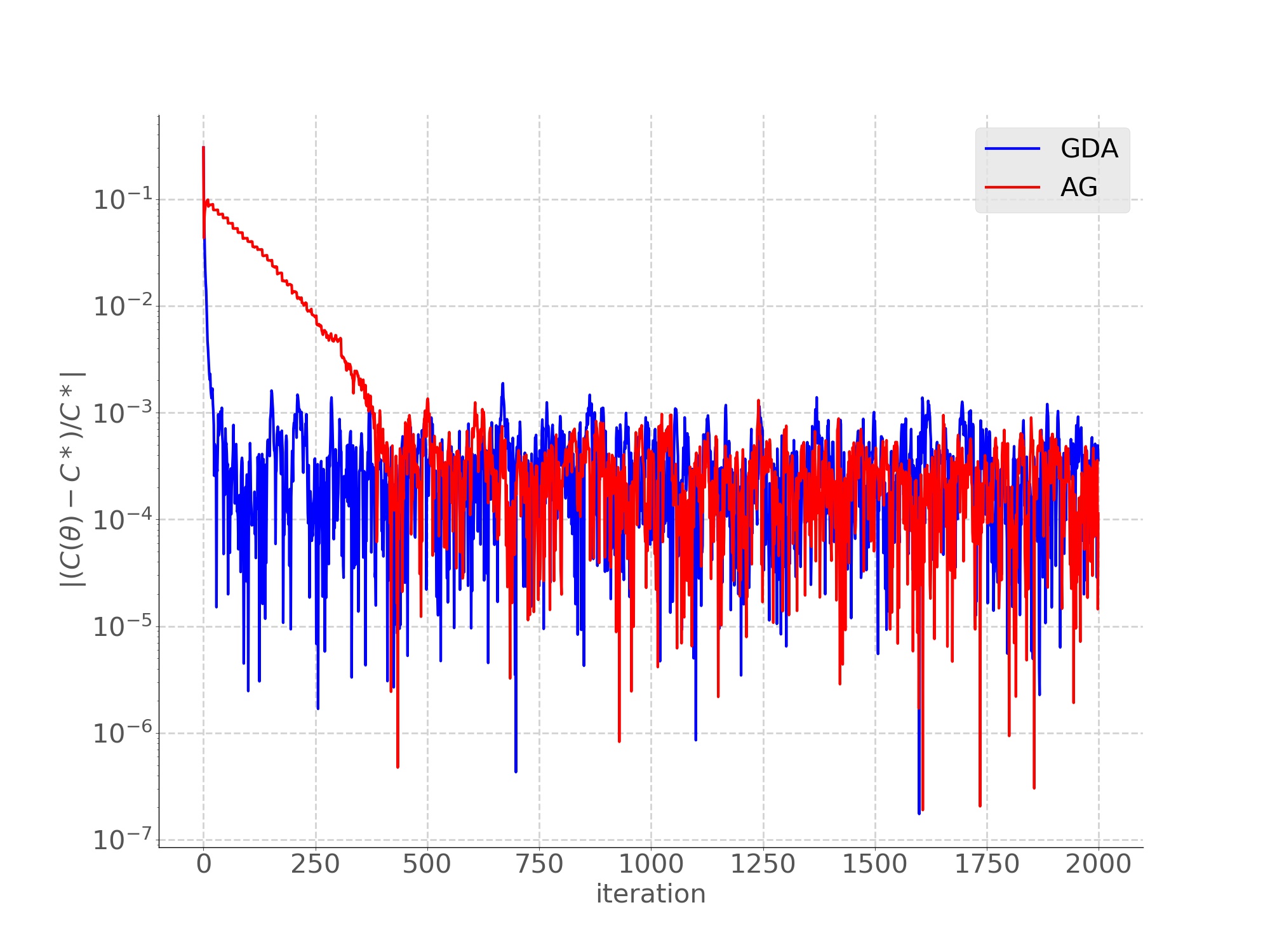}
		\caption{\,}
		\label{fig:1d-modelfree-utility}
	\end{subfigure}
	\caption{Sample-based policy optimization: Convergence of the control parameters in~(a) and of the relative error on the utility in~(b).}
	\label{fig:1d-modelfree-params-utility}
\end{figure}

\renewcommand{\arraystretch}{1.1}

\begin{table}[h!]
	
	\begin{center}
		\caption{Simulation parameters}    
		\begin{adjustbox}{width=0.7\columnwidth,center}
			\begin{tabular}{ccccccccc}
				
				\hline
				
				\multicolumn{9}{c}{Model parameters}\\
				
				\hline \hline
				
				 $A$ & $\overline{A}$ & $B_1=\overline{B}_1$ & $B_2=\overline{B}_2$ & $Q$ & $\overline{Q}$ & $R_1=\overline{R}_1$ & $R_2=\overline{R}_2$ & $\gamma$ \\
				
				\hline
				
				 0.4 & 0.4 & 0.4 & 0.3 & 0.4 & 0.4 & 0.4 & 0.4 & 0.9 \\
				
				\hline \hline
				
				\\
				
				\hline
				
				\multicolumn{9}{c}{Initial distribution and noise processes}\\
				
				\hline \hline
				
				 & \multicolumn{2}{c}{$\epsilon_0^0$} &  \multicolumn{2}{c}{$\epsilon^1_0$} &  \multicolumn{2}{c}{$\epsilon^0_t$} &  \multicolumn{2}{c}{$\epsilon^1_t$}  
				 \\
				
				\hline
				
				 & \multicolumn{2}{c}{$\mathcal{U}([-1, 1])$} & \multicolumn{2}{c}{$\mathcal{U}([-1, 1])$} & \multicolumn{2}{c}{$\mathcal{N}(0, 0.01)$} & \multicolumn{2}{c}{$\mathcal{N}(0, 0.01)$} 
				 \\
				
				\hline \hline
				
				\\
				
				\hline
				
				\multicolumn{9}{c}{AG and DGA methods parameters}\\
				
				\hline \hline
				
				$\Iiter_1$ & $\Iiter_2$ & $T$  & $\eta_1$ & $\eta_2$ & $K_1^0$ & $L_1^0$ & $K_2^0$ & $L_2^0$\\
				
				\hline
				
				10 & 200 & 2000 & 0.1 & 0.1 & 0.0 & 0.0 & 0.0 & 0.0 \\
				
				\hline \hline
				
				\\
				
				\hline
				
				\multicolumn{9}{c}{Gradient estimation algorithm parameters}\\
				
				\hline \hline
				\multicolumn{3}{c}{$\mathcal{T}$} &
				\multicolumn{3}{c}{$M$} &
				\multicolumn{3}{c}{$\tau$} 
				 \\
				
				\hline
				\multicolumn{3}{c}{50} &
				\multicolumn{3}{c}{10000} &
				\multicolumn{3}{c}{0.1} 
				\\
				
				\hline \hline

			\end{tabular}
		\end{adjustbox}

		\label{tab:simulation_parameters_ZS} 
	\end{center}
	
\end{table}